\documentclass[11pt]{article}
\usepackage{amsfonts}
\usepackage{amsthm}
\usepackage{amsmath,amssymb}
\usepackage{graphicx,floatrow,subfig}
\usepackage{multirow,makecell}
\usepackage{longtable}
\usepackage{cite}
\allowdisplaybreaks[4]

\newtheorem{lm}{Lemma}[section]
\newtheorem{thm}{Theorem}[section]
\newtheorem{pro}{Proposition}[section]
\newtheorem{exmp}{Example}
\newtheorem{rem}{Remark}[section]

\allowdisplaybreaks[4]

\title{\bf  The discontinuous limit case of an  archetypal oscillator with constant excitation and van der Pol damping: A single equilibrium}

\author{
{Xiuli Cen}$^1$,~
{Hebai Chen}$^1$,~
{Yilei Tang}$^{2}$, ~{Zhaoxia Wang}$^3$
	\footnote{ Email: cenxiuli2010@163.com (X. Cen), ~cenxiuli@csu.edu.cn (X. Cen), ~chen\_hebai@csu.edu.cn (H. Chen), ~mathtyl@sjtu.edu.cn (Y. Tang), ~zxwang@uestc.edu.cn (Z. Wang, corresponding author)
	}
		\\
	{\footnotesize  $^1$ School of Mathematics and Statistics, HNP-LAMA, Central South University,}\\
	{\footnotesize Changsha, Hunan 410083, China}
	\\
	{\footnotesize $^2$ School of Mathematical Sciences, CMA-Shanghai, Shanghai
	Jiao Tong University, Shanghai, 200240, China}
\\
	{\footnotesize  $^3$ School of Mathematical Sciences, University of Electronic Science and Technology of China,}\\
	{\footnotesize Chengdu, Sichuan 611731, China}
}

\date{}

\begin{document}
\maketitle

\begin{abstract}
This paper investigates the global dynamics of the discontinuous limit case of an  archetypal oscillator with constant excitation that exhibits a single equilibrium.
For parameter regions in which this oscillator possesses two or three equilibria, the global bifurcation diagram and the corresponding phase portraits on the Poincar\'e disc have been presented in [Phys. D, 438 (2022) 133362].
The present  work completes the global structure of the discontinuous limit case of an archetypal oscillator with constant excitation.
Although the dynamical phenomena are less rich compared to systems with more than one equilibrium,
the presence of a single equilibrium gives rise to additional limit cycles surrounding it,
thereby enriching the overall dynamics and making the analysis substantially more intricate than in the previously studied cases.



\end{abstract}

\baselineskip 16pt
\parskip 10pt

\numberwithin{equation}{section}


\section{Introduction and main results}

Many models in the natural sciences, engineering and economics are described by nonsmooth differential equations.
Examples include impact and collision models in mechanical systems, Coulomb friction models, threshold--fire neuron models, and switched circuit models in power electronics--none of which can be captured by smooth differential equations.
Nonsmooth differential systems have attracted sustained and intensive attention from a large number of researchers, leading to a wealth of monographs and scholarly works (see, e.g., \cite{AEG, BBCK, BBCKNOP, BCT, CJ, Jeff, JH, ML, SM, SST} and the references therein).
It is worth noting that nonsmooth differential systems can exhibit dynamical phenomena that are impossible in their smooth counterparts. For instance, while linear systems possess no limit cycles, even a planar continuous piecewise-linear system with a single switching line can support a limit cycle \cite{LOP}.
Moreover, in the study of nonsmooth differential systems, many classical tools and theories from smooth dynamical systems--such as linearization, center manifold reduction, and normal form theory--no longer apply, which poses significant theoretical and practical challenges. Furthermore,
the investigation of classical nonsmooth differential models not only resolves the practical problems embodied in these models themselves, but also advances the theory and methodology of nonsmooth differential systems.

A model that describes the motion of an elastic arch rigidly supported at both ends and subjected to external forces
is proposed by Cao et al. in \cite{Cao06,Cao08}, as shown in Fig. \ref{Intro} (a).
This system can be reduced to a single degree of freedom nonlinear oscillator if only the symmetric vibration mode of the elastic arch is considered.
As shown in Fig.  \ref{Intro} (b),
a lumped mass $m$ is subjected to an external force $p(t)$ and is symmetrically suspended between two identical,
linearly elastic springs of stiffness $k$ and unstressed length $L$, with their outer ends are pinned to rigid supports.
Let $X$ denote the displacement of the mass and $l$ the half distance between two rigid supports.
By Newton's second Law,  the equation of this motion is
\begin{eqnarray}
m\frac{d^2X}{dt^2}+2kX\left(1-\frac{L}{\sqrt{X^2+l^2}}\right)=p(t).
\label{SD1}
\end{eqnarray}
Introducing
\[
x=\frac{X}{L},\ \tau=t\sqrt{\frac{2k}{m}}, \ ~\alpha=\frac{l}{L},\  ~p_1(t)=\frac{p(t)}{2kL},
\]
equation \eqref{SD1} is changed into
\begin{eqnarray}
\frac{d^2x}{d\tau^2}+x\left(1-\frac{1}{\sqrt{x^2+\alpha^2}}\right)=p_1(t).
\label{SD2}
\end{eqnarray}
Equation \eqref{SD2}
is referred to as smooth and discontinuous (SD for short) oscillator,  because it is smooth for $\alpha\neq0$ and discontinuous for $\alpha=0$.
The SD oscillator can  be derived from the shallow elastic arch system studied by Thompson and Hunt \cite{TH},
which has wide application in aerospace, high-speed rail vibration isolation, and other engineering fields.

\begin{figure}[!htb]
	\centering
	\subfloat[an elastic arch consisting of a curved and tapered beam pinned to rigid abutments]
	{\includegraphics[scale=0.22]{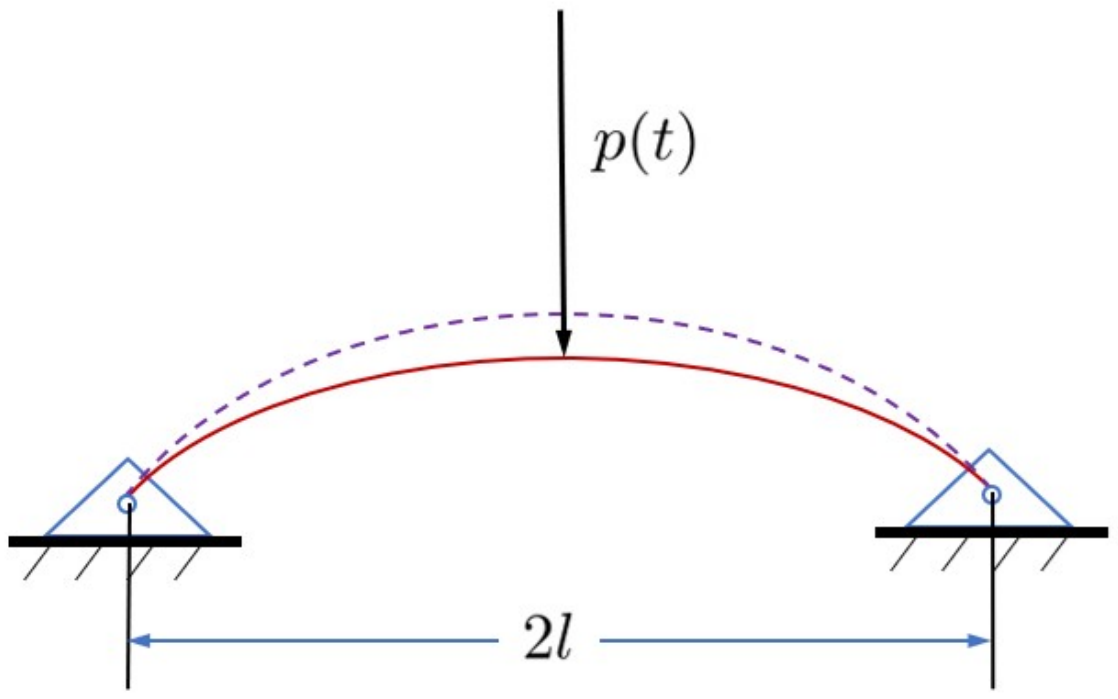}}
\hskip 2cm
	\subfloat[a vertical oscillator with a pair of springs pinned to rigid supports]
	{\includegraphics[scale=0.22]{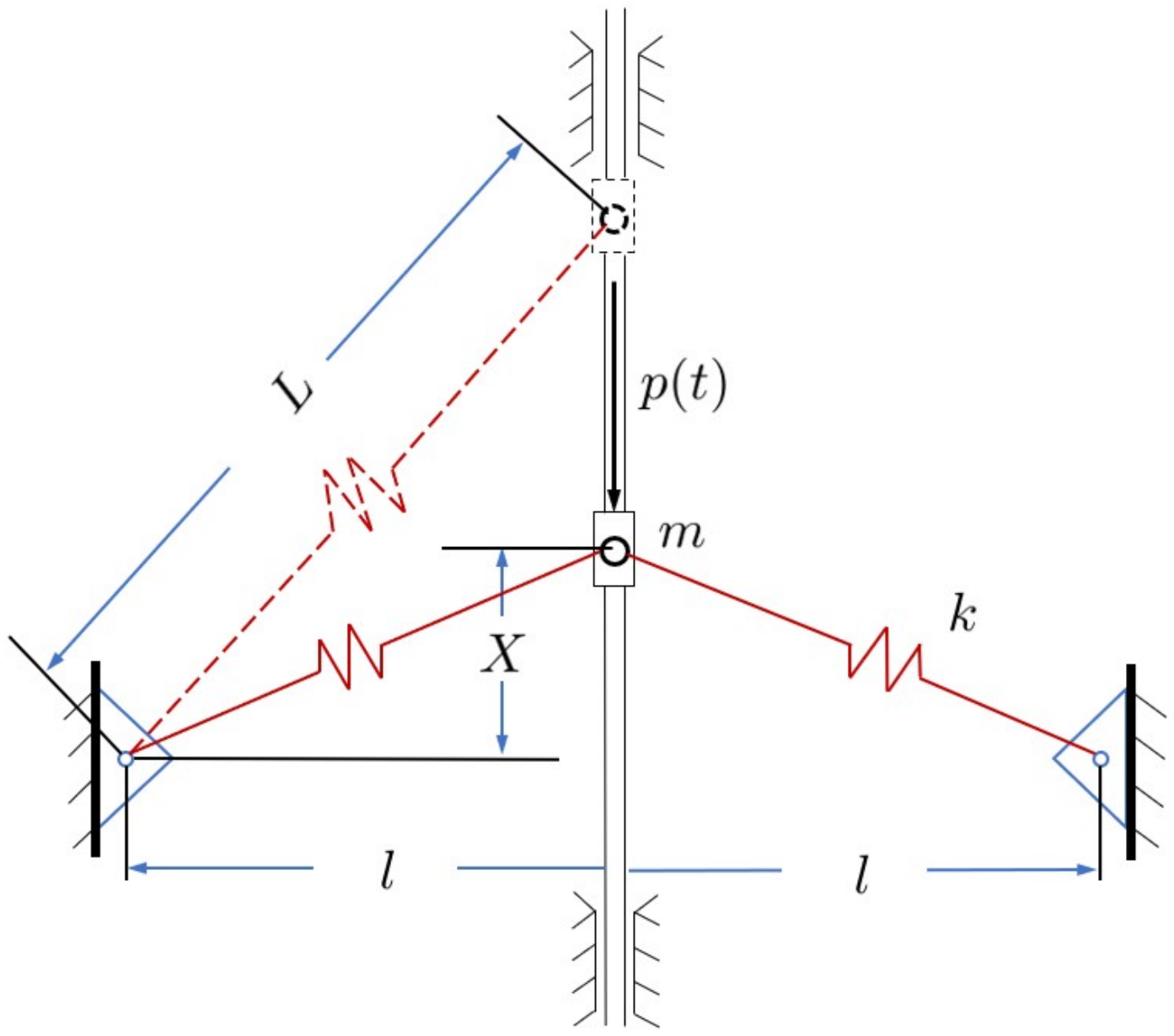}}
	\caption{{\footnotesize A shallow elastic arch and the simplified oscillator.}}
	\label{Intro}
\end{figure}

Consider equation \eqref{SD2} with  van der Pol damping
\begin{eqnarray}
\ddot{x}+\delta( x^2+b )\dot{x}+x\left(1-\frac{1}{\sqrt{x^2+\alpha^2}}\right)=p_1(t),
\label{SDv}
\end{eqnarray}
where   $\alpha,\delta,  b \in \mathbb{R}$ and  $p_1(t)$ is a continuous  function in $t$.
In the absence of external forces,  equation \eqref{SDv} can be expressed as
\begin{eqnarray}
\ddot{x}+\delta( x^2+b )\dot{x}+x\left(1-\frac{1}{\sqrt{x^2+\alpha^2}}\right)=0.
\label{SD-0}
\end{eqnarray}
Local bifurcation diagrams and phase portraits of equation \eqref{SD-0} are presented  in \cite{TCY} for small $|\alpha-1|$ and $|b|$.
For $\alpha=0$, the global bifurcation diagram and global phase portraits in the Poincar\'e disc of  equation \eqref{SD-0}
have been given in \cite{Chen}.
For $\alpha\ne0$,
the global bifurcation diagram and global phase portraits in the Poincar\'e disc are provided in \cite{CLT} except for
the exact number of limit cycles surrounding a single equilibrium over a wide parameter range.
Subsequently, Sun and Liu in \cite{SL} show that there are at most two limit cycles of equation \eqref{SD-0} surrounding the same single equilibrium.
Chen, Tang and Zhang in \cite{CTZ} extend the results on the number of limit cycles to the full parameter range
and yield a complete characterization of global dynamics of equation \eqref{SD-0}
by  the theory of rotated vector fields.

Equation  \eqref{SDv} with a constant  external force can  describe
the motion of an elastic arch that is rigidly supported at both ends, either carrying a suspended mass or being subjected to a constant external load.
Consider the discontinuous limit case of equation \eqref{SDv} with constant external excitation $p_1(t)=a$,
\begin{eqnarray}
\ddot{x}+\delta( x^2+b )\dot{x}+x-{\rm sgn}(x)-a=0.
\label{SD-x}
\end{eqnarray}
Clearly, equation \eqref{SD-x} can be rewritten as
\begin{eqnarray}
\dot{x}=y-F(x), ~~~~~~~\dot{y}=-g(x),
\label{SD}
\end{eqnarray}	
by introducing $y=\dot{x}+F(x)$, where
\[
F(x):=\int_0^x f(t) \mathrm{d}t=\delta\left(\frac{x^3}{3}+b x\right),\hskip 1cm f(x):=\delta( x^2+b ),\hskip 1cm g(x):=x-{\rm sgn}(x)-a.
\]
Since system \eqref{SD} is invariant by the transformations
$(x,y,t,\delta,a,b)\rightarrow(x,-y,-t,-\delta,a,-b)$ and
$(x,y,t,\delta,a,b)\rightarrow(-x,y,-t,-\delta,-a,b)$,
and
the global dynamic behaviours of system \eqref{SD} with $\delta=0$ and system  \eqref{SD} with $a=0$ are studied in \cite{Chen},
it is enough to consider  system \eqref{SD} with $\delta>0$ and $a> 0$.
System \eqref{SD} with
$0<a\leq1$ is considered in
the companion paper \cite{CTW} and  we continue to consider the case $a>1$ in this paper.

For simplicity, let
\[
\mathcal{G}:=\{(a,b,\delta)\in\mathbb{R}^+\times \mathbb{R}\times \mathbb{R}^+: a>1\}.
\]

\begin{thm}
 The global bifurcation diagram of system \eqref{SD} in the parameter region $\mathcal{G}$ consists of the following bifurcation
	surfaces :{\small
	\begin{description}
		\item[(a)]  Hopf bifurcation surface $H:=\{(a,b,\delta)\in\mathcal{G}:  b=-(a+1)^2 \}$;

	    \item[(b)]  grazing bifurcation surface
                       $G:=\{(a,b,\delta)\in\mathcal{G}: b=\varphi_1(a, \delta)\}$;

       	\item[(c)]  double crossing limit cycle bifurcation surfaces
                       $DL_1:=\{(a,b,\delta)\in\mathcal{G}:  b=\varrho_1(a, \delta), \,a<a_0\}$
                    and
                    $DL_2:=\{(a,b,\delta)\in\mathcal{G}_1: b=\varrho_2(a, \delta), \,a<a_0\}$,
	\end{description}
   } where
    $a_0>1$ is a constant,
    $\varphi_1(a, \delta)$ is a decreasing $\mathcal{C}^{\infty}$ function with respect to $a>1$,  and $\varrho_1(a, \delta)$ and $\varrho_2(a, \delta)$ are $\mathcal{C}^{0}$ functions for $1<a<a_0$ such that $-3(a+1)^2<\varrho_2(a, \delta)<\varphi_1(a, \delta), \varrho_1(a, \delta)<-(a+1)^2$
    and $\varrho_1(a, \delta)>\varphi_1(a, \delta)$ as $a-1$ is sufficiently small.
    \label{mr1}
\end{thm}

The grazing bifurcation surface $G$ intersects the double crossing limit cycle bifurcation surface  $DL_1$ at a curve $P$.
Define
        \begin{eqnarray*}
        P:&=&\{(a,b,\delta)\in\mathcal{G}:~ b=\varphi_1(a, \delta)=\varrho_1(a, \delta)\},
        \\
         Q:&=&\{(a,b,\delta)\in\mathcal{G}:~ b=\varrho_1(a, \delta)=\varrho_2(a, \delta)\},
        \\
        G_{1}:&=&\{(a,b,\delta)\in\mathcal{G}:~ b=\varphi_1(a, \delta), ~\varphi_1(a, \delta)<\varrho_1(a, \delta)\},
        \\
        G_{2}:&=&\{(a,b,\delta)\in\mathcal{G}:~ b=\varphi_1(a, \delta), ~\varphi_1(a, \delta)>\varrho_1(a, \delta)\},
        \\
        DL_{11}:&=&\{(a,b,\delta)\in\mathcal{G}:~ b=\varrho_1(a, \delta), ~\varphi_1(a, \delta)>\varrho_1(a, \delta) \},
        \\
		DL_{12}:&=&\{(a,b,\delta)\in\mathcal{G}:~ b=\varrho_1(a, \delta), ~\varphi_1(a, \delta)<\varrho_1(a, \delta) \}.
	    \end{eqnarray*}

 \begin{figure}[!htb]
		\centering
		\subfloat[the open regions divided by bifurcation curves]
		{\includegraphics[scale=0.25]{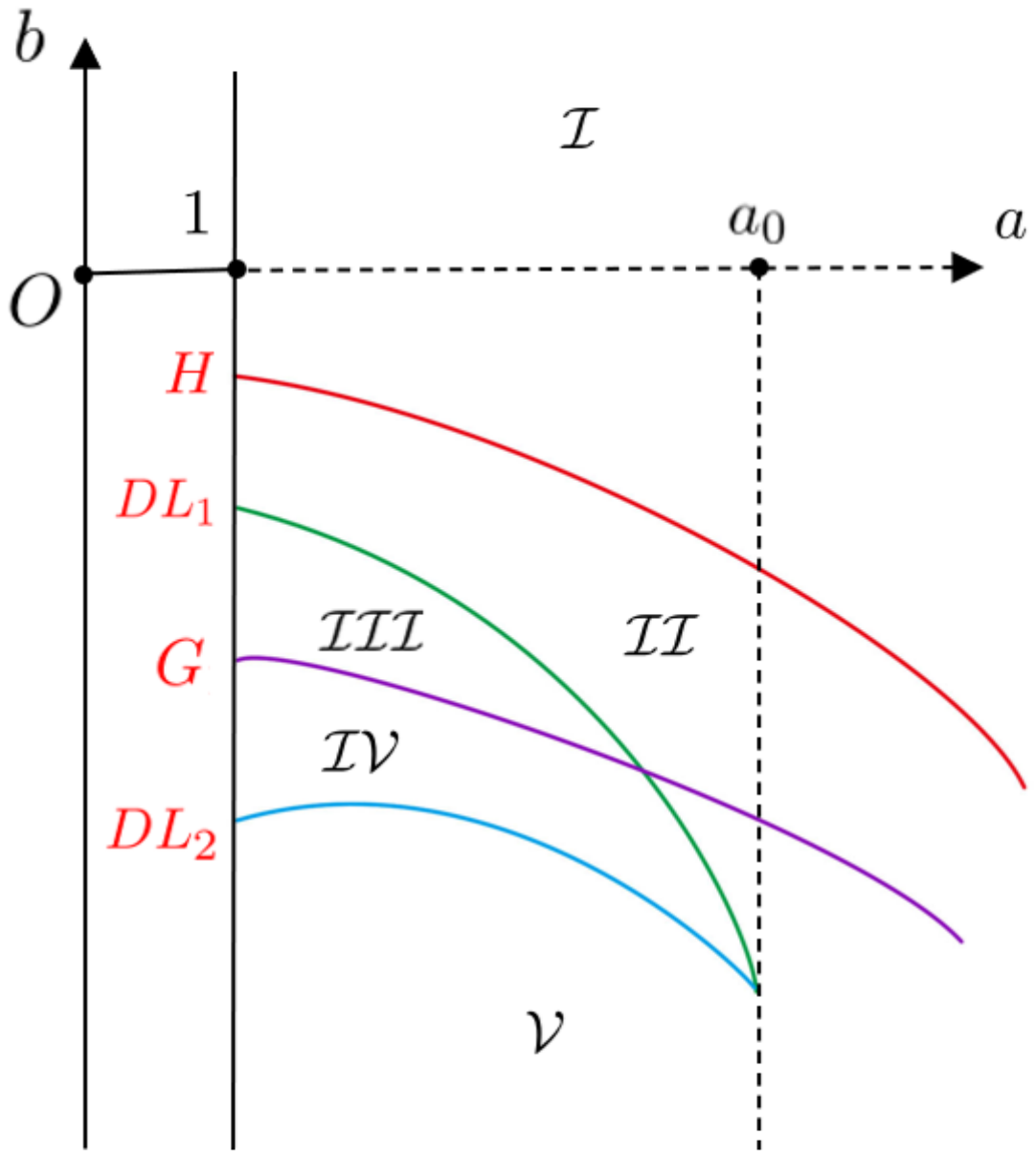}}\hspace{10pt}
		\subfloat[ bifurcation curves]
		{\includegraphics[scale=0.25]{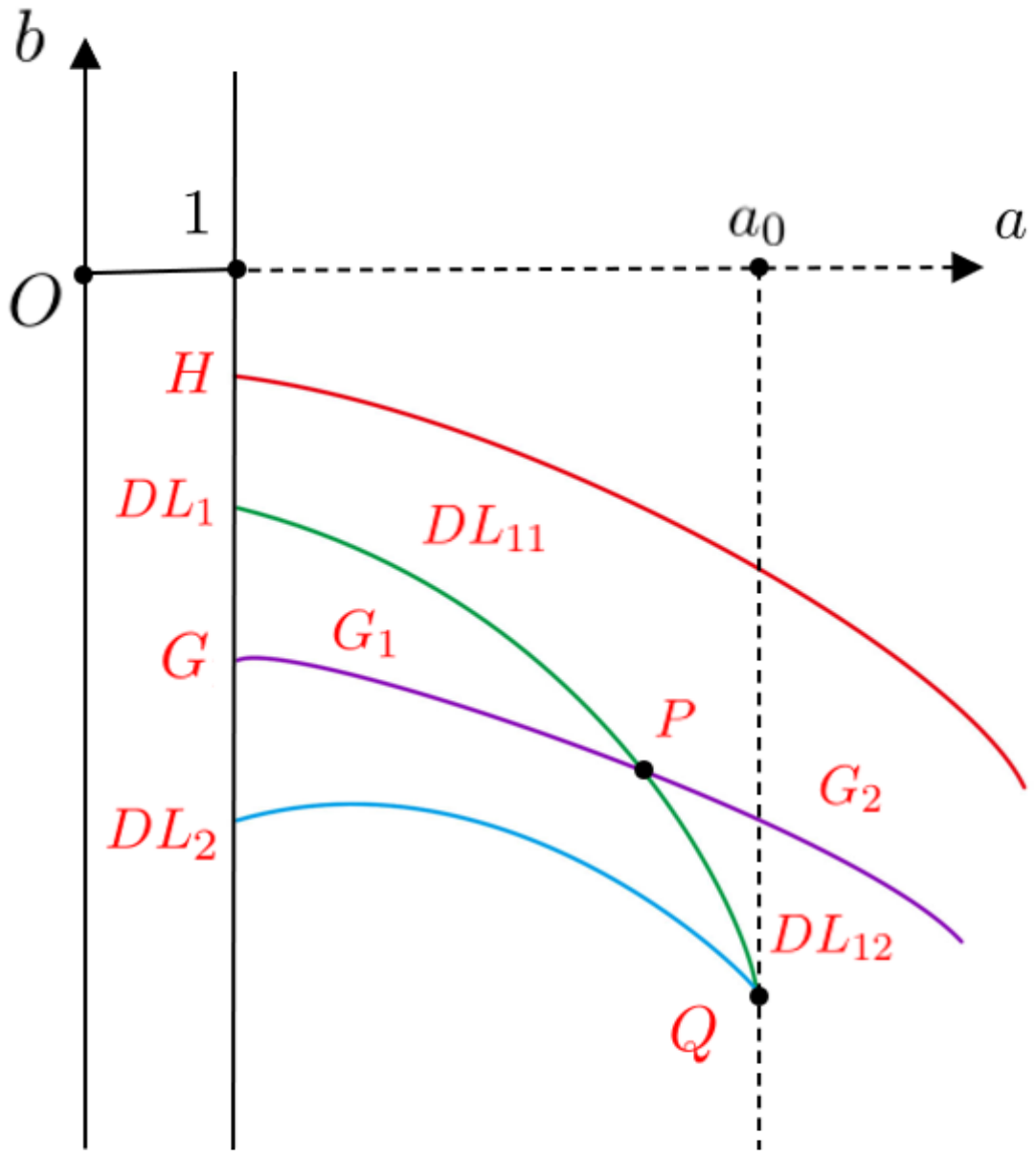}}
		\caption{{\footnotesize The slice $\delta=\delta_*$ of the global bifurcation diagram in $\mathcal{G}$.}}
		\label{global-1}
	   \end{figure}

\begin{thm}
	For a fixed  $\delta_*>0$, the slice $\delta=\delta_*$ of the global bifurcation diagram and
	global phase portraits in the Poincar\'e disc of \eqref{SD} are given in  {\rm Figs. \ref{global-1}} and {\rm \ref{global-2}},
	where
	   \begin{eqnarray*}
		{\cal I}&:=&\{(a,b,\delta)\in\mathcal{G}:~\delta=\delta_*,  ~b>-(a+1)^2\},
		\\
		{\cal II}&:=&\{(a,b,\delta)\in\mathcal{G}:~\delta=\delta_*, ~a<a_0, ~\max\{\varphi_1(a, \delta_*), \varrho_1(a, \delta_*)\}<b<-(a+1)^2\}
        \\
                 &&\cup \{(a,b,\delta)\in\mathcal{G}:~\delta=\delta_*, ~a>a_0, ~\varphi_1(a, \delta_*)<b<-(a+1)^2\},
		\\
        {\cal III}&:=&\{(a,b,\delta)\in\mathcal{G}:~\delta=\delta_*, ~a<a_0, ~\varphi_1(a, \delta_*)<b<\varrho_1(a, \delta_*)\},
		\\
        {\cal IV}&:=&\{(a,b,\delta)\in\mathcal{G}:~\delta=\delta_*, ~a<a_0, ~\varrho_2(a, \delta_*)<b<\min\{\varphi_1(a, \delta_*), \varrho_1(a, \delta_*)\}\},
		\\
        {\cal V}&:=&\{(a,b,\delta)\in\mathcal{G}:~\delta=\delta_*, ~a<a_0, ~b<\varrho_2(a, \delta_*)\}
        \\
                &&\cup  \{(a,b,\delta)\in\mathcal{G}:~\delta=\delta_*, ~a<a_0, ~\varrho_1(a, \delta_*)<b< \varphi_1(a, \delta_*)\}
        \\
                &&\cup  \{(a,b,\delta)\in\mathcal{G}:~\delta=\delta_*, ~a>a_0, ~b< \varphi_1(a, \delta_*)\}.
       \end{eqnarray*}
Moreover, there are at most three limit cycles of system \eqref{SD}.
       \begin{figure}[!htb]
       \centering
	   \includegraphics[scale=0.55]{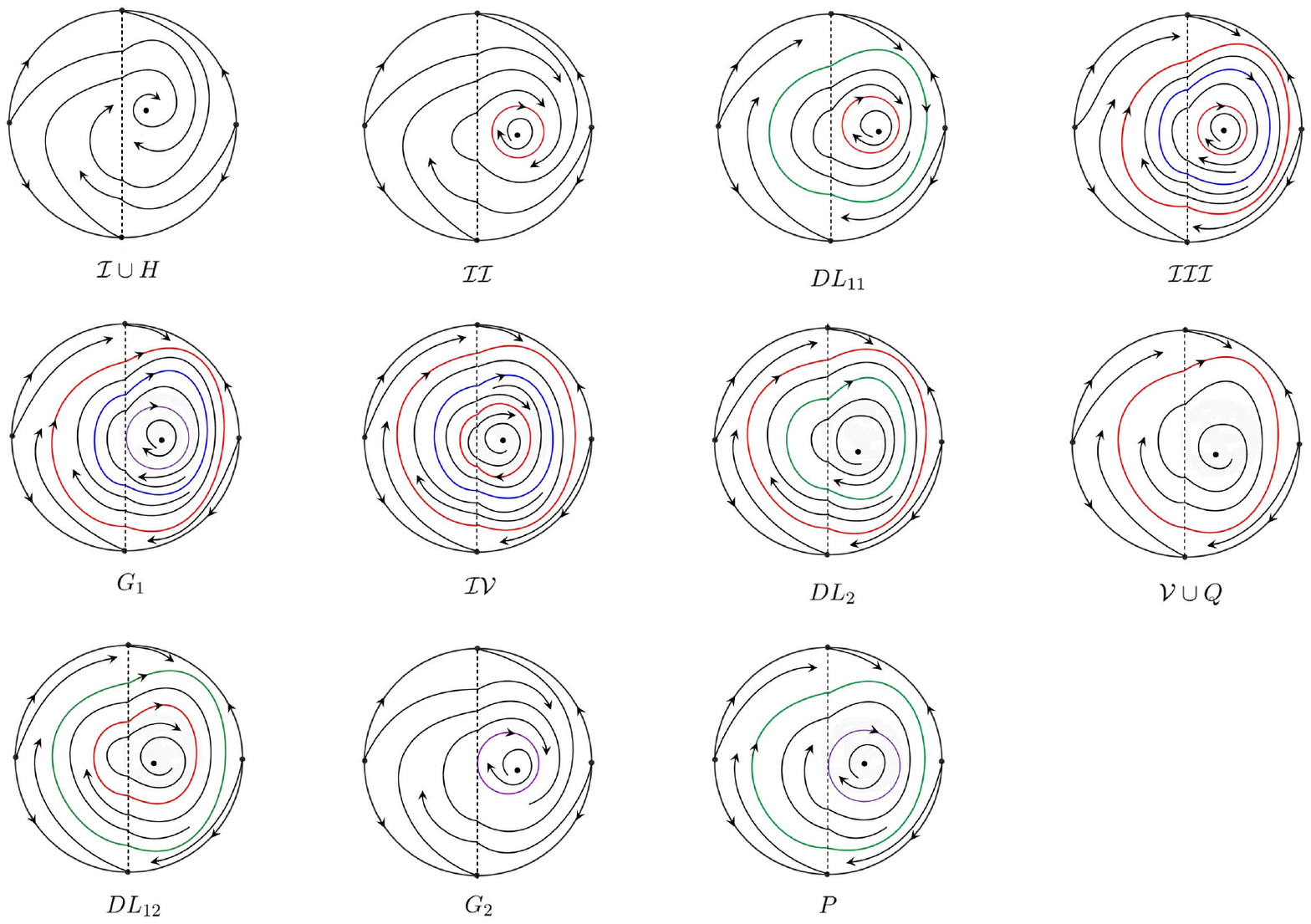}
	  \caption{\footnotesize{ The global phase portraits in $\mathcal{G}$.}}
	  \label{global-2}
      \end{figure}

	\label{mr2}
\end{thm}

\begin{rem}
The orbit tangent to the $y$-axis has two possible ways of connecting the saddles at infinity in the global phase portraits of ${\cal I}\cup H$ and ${\cal II}$.
Then, there may be more subcases in ${\cal I}\cup H$ and ${\cal II}$ of Theorem \ref{mr2}, as shown in {\rm Fig. \ref{fig-IG}}.
\end{rem}

    \begin{figure}[!htb]
		\centering
		\subfloat[$({\cal I}\cup H)_1$]
		{\includegraphics[scale=0.138]{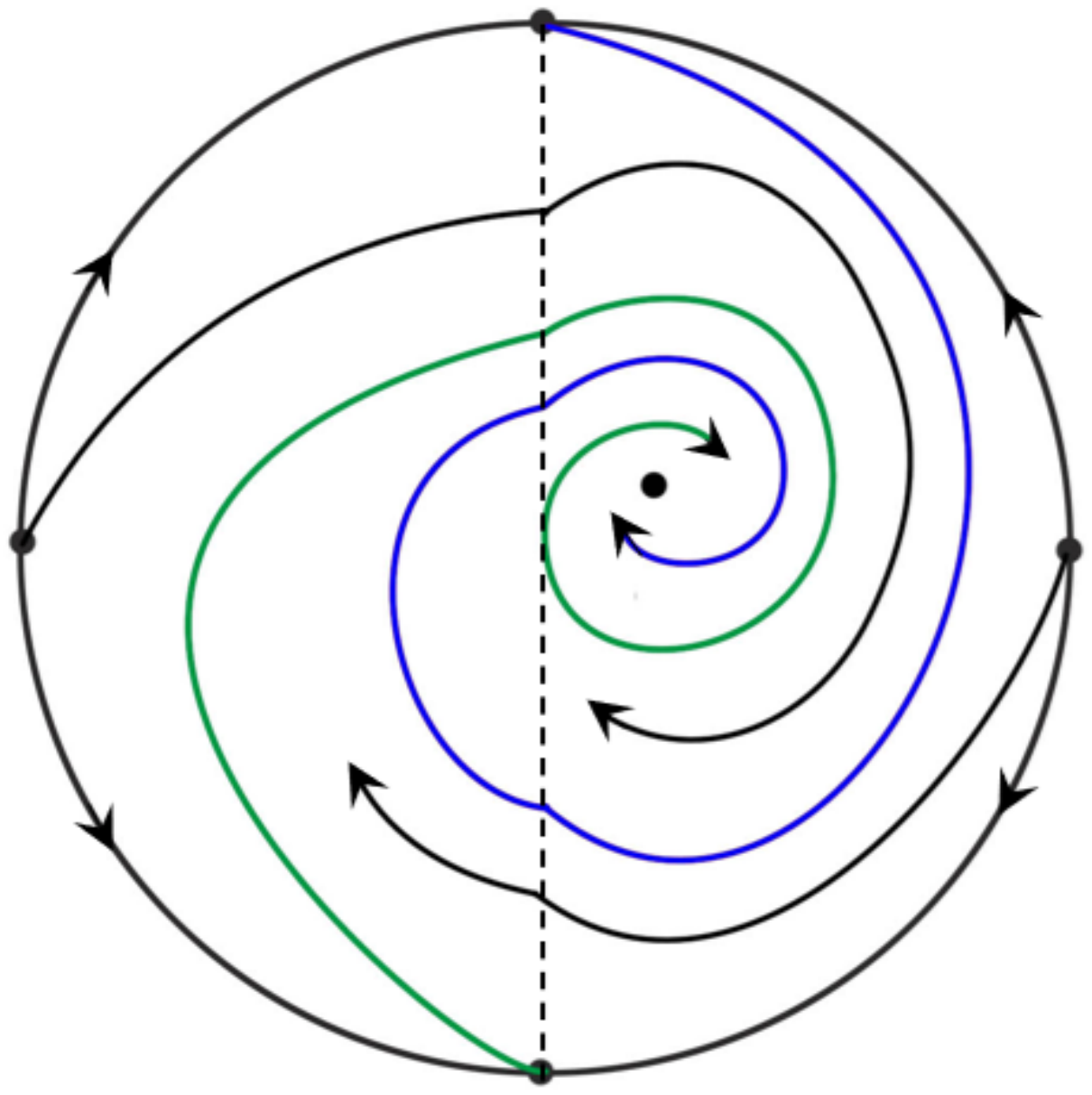}}
		\subfloat[$({\cal I}\cup H)_2$]
		{\includegraphics[scale=0.138]{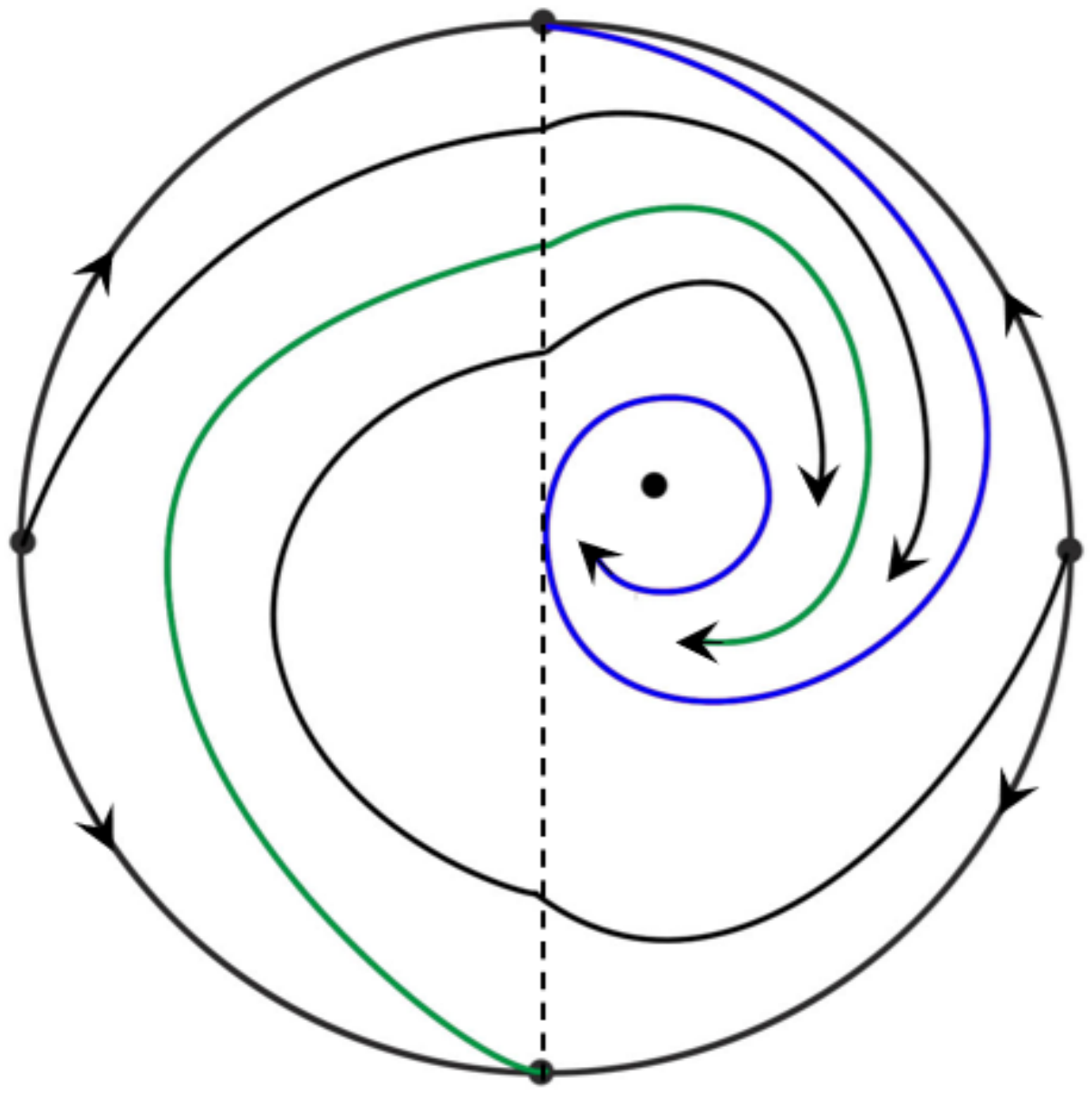}}
		\subfloat[${\cal II}_1$]
		{\includegraphics[scale=0.138]{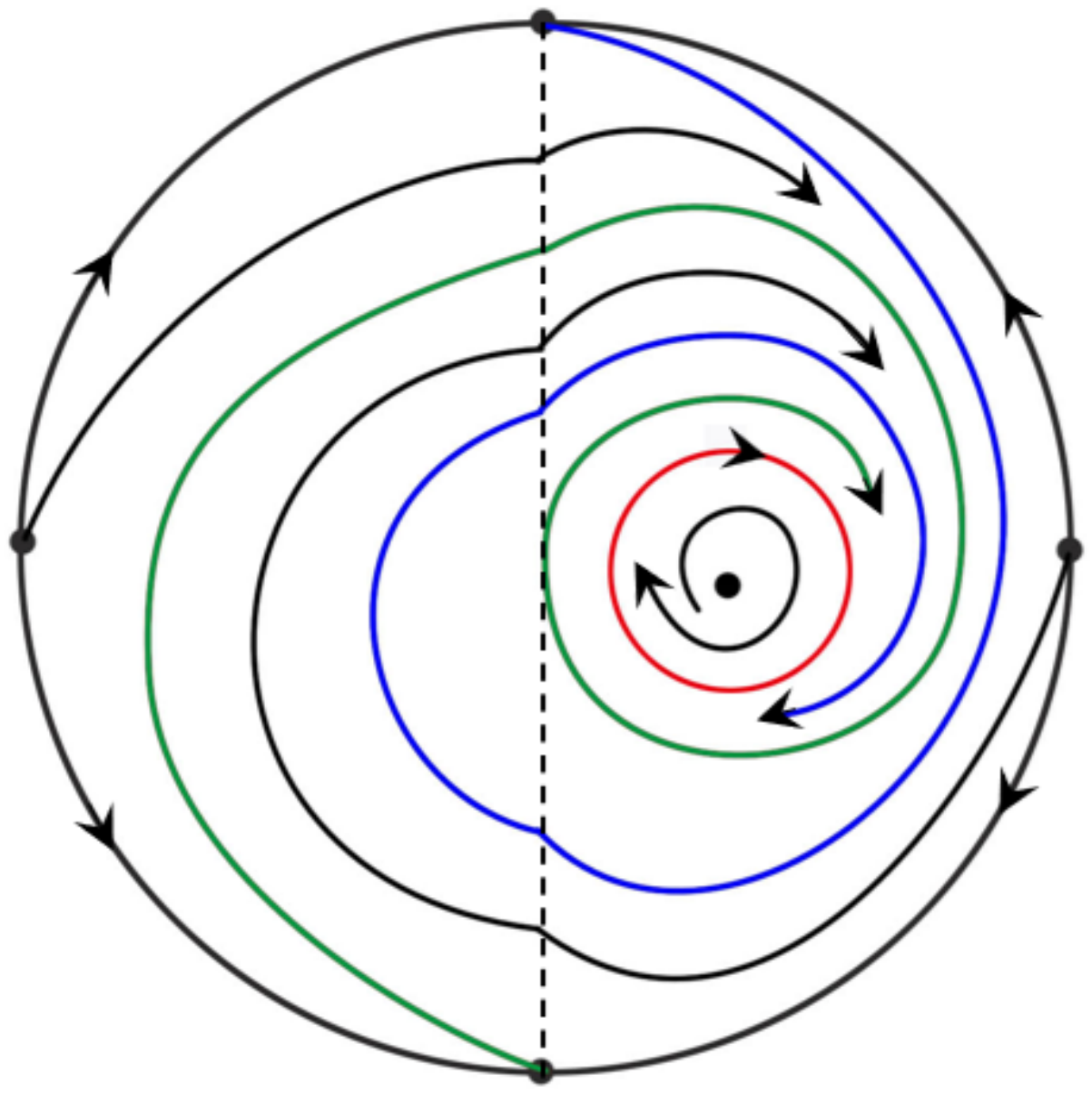}}
		\subfloat[${\cal II}_2$]
		{\includegraphics[scale=0.138]{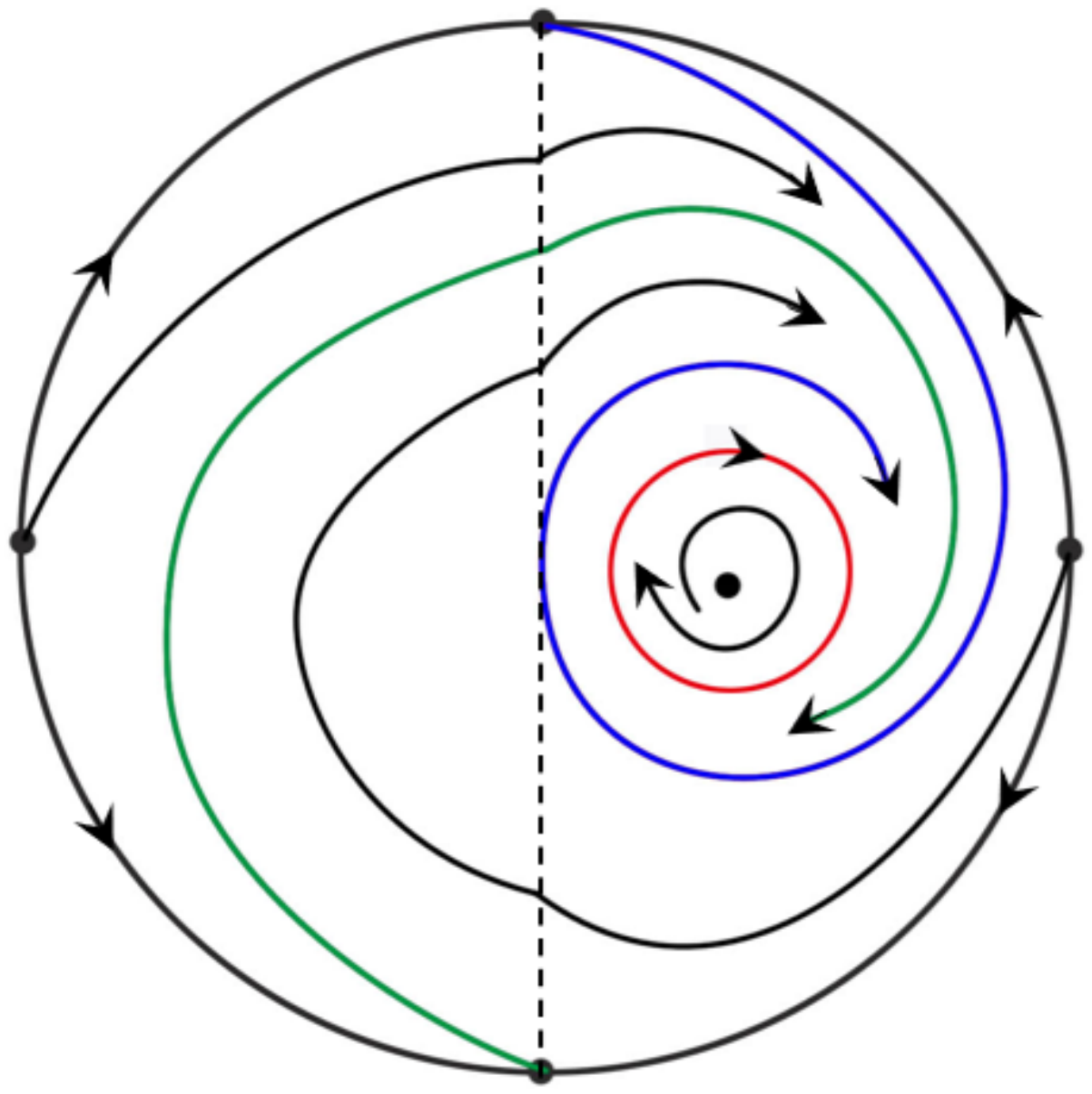}}
		\caption{{\footnotesize  Possible connections of the grazing orbit and the saddles at infinity.}}
		\label{fig-IG}
	\end{figure}

This paper is organised as follows. We give the qualitative properties of equilibria and  local bifurcations of system \eqref{SD} in section 2.
The limit cycles and  global bifurcations of system \eqref{SD} are shown in section 3.
The proof of main results 
is given in section 4.
The numerical simulations illustrate our analytical results in section 5.
Some conclusions are presented in section 6.

\section{Equilibria and  local bifurcations of system \eqref{SD}}

  Solving  $\dot x=\dot y=0$ for system \eqref{SD}, we can get that
  system  \eqref{SD} exhibits a unique equilibrium  $E_r:(a+1, F(a+1))$.
  Moreover, qualitative properties of $E_r$ are shown  in the following lemma.

\begin{lm}
    $E_r$ is a source when
    $b<-(a+1)^2$   and
    it is a sink when
    $b\geq-(a+1)^2$.
    \label{lm-fi-e}
\end{lm}

\begin{proof}
  The Jacobian matrix for system \eqref{SD}
  at $E_r$ is
	\begin{eqnarray*}
    J_{r}=
	\left(
	\begin {array}{cc}
	-\delta((a+1)^2+b) &  1 \\
	\noalign{\medskip}
	-1 & 0
	\end {array}
	\right).
	\end{eqnarray*}
  Then, ${\rm det}J_r=1$ and ${\rm tr}J_r=-\delta((a+1)^2+b)$.
  Moreover, $E_r$ is a source (resp. sink) if $b<-(a+1)^2$ (resp. $b>-(a+1)^2$).
  In the case $b=-(a+1)^2$,
  system \eqref{SD} can be rewritten as
    \begin{eqnarray}
	\dot{x}=-y+\delta(a+1)x^2+\frac{\delta}{3} x^3,\hskip 0.5cm \dot{y}=x-{\rm sgn}(x+a+1)+1,
	\label{Er-O}
    \end{eqnarray}
  by  the transformation  $(x,y,t)\rightarrow(x+(a+1), y+F(a+1),-t)$.
  Notice that ${\rm sgn}(x+a+1)=1$ near $x=0$ for any $a>1$.
  Then, system \eqref{Er-O} has the same topological structure as
	\begin{eqnarray}
	\dot{x}=-y+\delta(a+1)x^2+\frac{\delta}{3} x^3,\hskip 0.5cm \dot{y}=x
	\label{Err-O}
    \end{eqnarray}
  near the origin.
  One can compute that the first focal value of system \eqref{Err-O} is $g_{3}=\pi\delta/4>0$,
  identifying that the origin of system \eqref{Err-O} is an unstable weak focus of order 1.
  Then, $E_r$ is  a stable weak focus of order 1 when $b=-(a+1)^2$,
  which completes the proof.
\end{proof}

By Lemma \ref{lm-fi-e}, the
bifurcations from the unique equilibrium can be obtained in the following proposition.

\begin{pro}
	There is a unique limit cycle of system \eqref{SD}, which is stable, occurring in a small neighborhood of  $E_r$
    if $b$ varies from
	$b=-(a+1)^2$ to $b=-(a+1)^2-\epsilon$
	and no limit cycles exist in any small neighborhood of  $E_r$
	if $-(a+1)^2\le b<-(a+1)^2+\epsilon$, where $\epsilon>0$ is sufficiently small.
	\label{hopfbif}
\end{pro}

\begin{proof}
    Due to Lemma \ref{lm-fi-e}, $E_r$ is a stable weak focus of order $1$ when $b=-(a+1)^2$,
	and $E_r$ is unstable when  $b<-(a+1)^2$.
	By the classical Hopf bifurcation theorem, there is a unique limit cycle, which is stable, bifurcated from  $E_r$
	when  $b$ varies from $b=-(a+1)^2$ to $b=-(a+1)^2-\epsilon$ and no limit cycles exist in any small neighborhood of  $E_r$
	if $-(a+1)^2\le b<-(a+1)^2+\epsilon$.
\end{proof}

For the equilibria of system \eqref{SD} with $a>1$ at infinity, the analysis is totally the same as the case $0<a<1$ (see \cite{CTW}).
System \eqref{SD} exhibits the equilibria $I_x^+$, $I_x^-$ lying on the $x$-axis and $I_y^+$, $I_y^-$  lying on the $y$-axis.
Moreover, the properties of $I_x^+$, $I_x^-$, $I_y^+$ and $I_y^-$ are shown in the following lemma.

\begin{lm}
For system \eqref{SD}, all orbits are positively bounded.
Moreover, the small neighborhoods of both $I_x^+$ and $I_x^-$ each consist of a single parabolic sector,
and those of both $I_y^+$ and $I_y^-$ each consist of two hyperbolic sectors, as shown in {\rm Fig. \ref{fig-inf-d}}.
	\begin{figure}[!htb]
	\centering
	\includegraphics[scale=0.28]{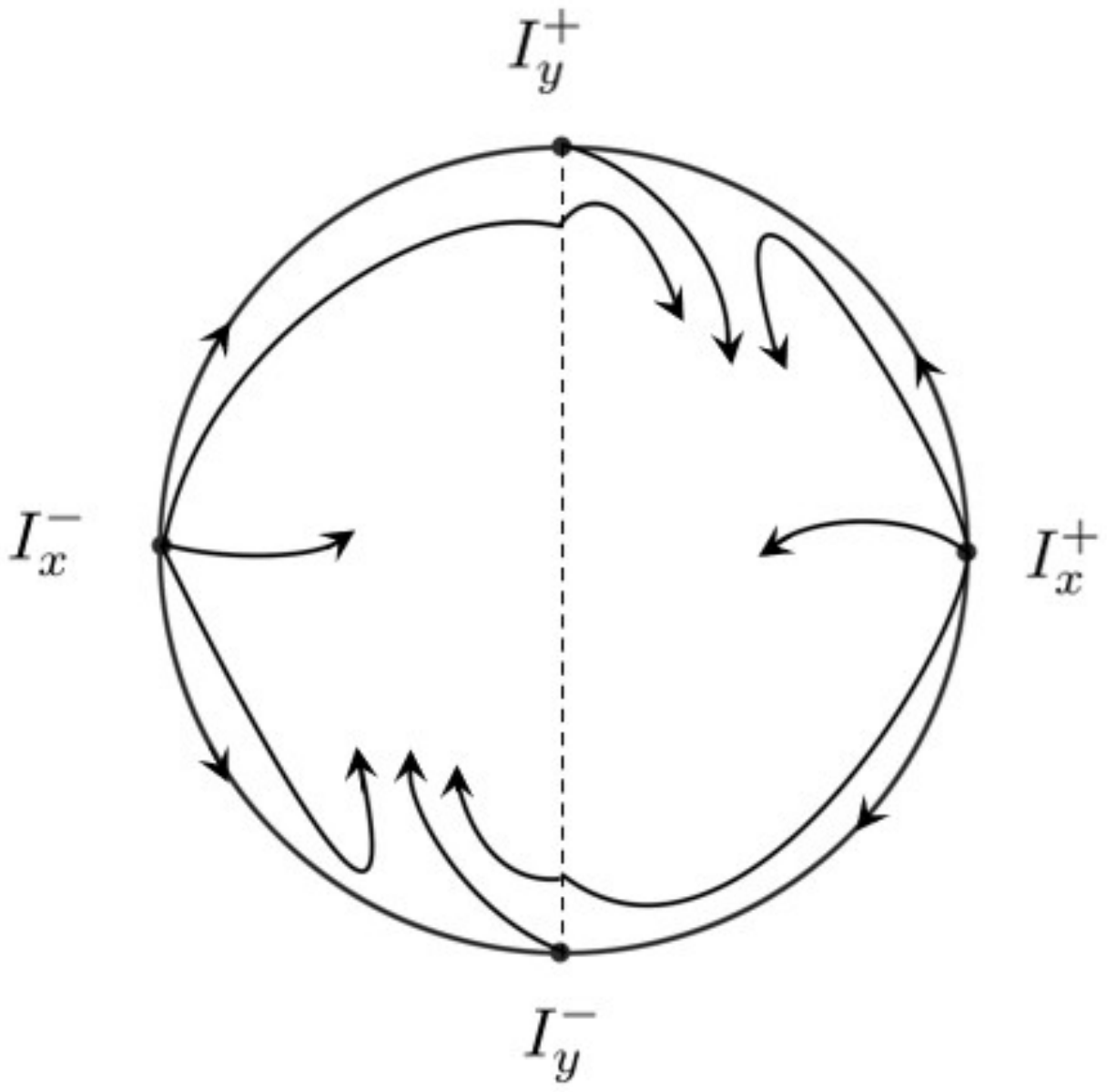}
	\caption{\footnotesize{ Dynamical behaviors near $I_x$ and $I_y$.}}
	\label{fig-inf-d}
    \end{figure}
\label{lm-infi}
\end{lm}

The proof is just like the proof of \cite[Lemma 2.3]{CTW}, so we omit it.


\section{Limit cycles and and  global  bifurcations of system \eqref{SD}}

Consider Li\'enard system
\begin{eqnarray}
\dot{x}=y-\hat F(x), \quad \dot{y}=-\hat g(x),
\label{hat-Fg}
\end{eqnarray}
where  $\hat F(x)\in C^1(\alpha,\beta)$ with $\hat F(0)=0$,
$\hat g(x)$ has $m$ jump discontinuities at $a_1, a_2, \cdots, a_m\in (\alpha,\beta)$
and is continuous at any  $x\in (\alpha, \beta)\backslash \{a_1, a_2, \cdots, a_m\}$,
$\alpha\in\mathbb{R}\cup \{-\infty\}$, $\beta\in\mathbb{R}\cup \{+\infty\}$ and $m\ge0$.

 \begin{lm}
   There is at most one crossing limit cycle in the strip $\alpha<x<\beta$ for system \eqref{hat-Fg} if the following conditions are satisfied,
    \begin{description}
		\item{\bf (i)} $\hat{f}(x)=\hat{F}'(x)$ has a unique zero $x_0<0$;
              $\hat{f}(x)<0$ (resp. $>0$) as $\alpha<x<x_0$ (resp. $x_0 < x < \beta$);

		\item{\bf (ii)} $\hat{F}(0) = 0$, $\hat{F}(\xi_0) = 0$ for $\alpha < \xi_0 < x_0$;

        \item{\bf (iii)} $\hat{g}(x)$ has a unique jump discontinuity point at $x=a_1\in(0, \beta)$ and
              $\lim_{x\rightarrow a_1^+}\hat{g}(x)<\lim_{x\rightarrow a_1^-}\hat{g}(x)$,
              $x\hat{g}(x)>0$ for $x\in (\alpha, \beta)\backslash\{0,a_1\}$;

        \item{\bf (iv)} the simultaneous equations
               \begin{eqnarray}
               \hat{F}(x_1) = \hat{F}(x_2), \quad \lambda(x_1) = \lambda(x_2)
               \label{Flam}
               \end{eqnarray}
             have no solutions such that $x_1 < x_0 < 0 < x_2<a_1$
             and have at most one solution such that $x_1 < x_0 < 0<a_1< x_2$, where $\lambda(x)$ is defined as
               $\lambda(x)= \hat{g}(x)/\hat{f}(x)$;

        \item{\bf (v)}  the function $\hat{F}(x)\hat{f}(x)/\hat{g}(x)$ is decreasing on the interval $(\alpha, x_0)$.
    \end{description}
    Moreover, the crossing limit cycle is unique, unstable and hyperbolic if it exists.
	\label{uni-lc-F}
  \end{lm}

  \begin{proof}
    Due to condition  {\bf (i)},
    \[
	{\rm div}(y-\hat{F}(x), -\hat{g}(x))=-\hat{f}(x)<0
	\]
    when $x> x_0$.
    By Bendixson-Dulac Criterion, system \eqref{hat-Fg} exhibits no limit cycles in the strip $x>x_0$.
    According to  conditions  {\bf (ii)} and  {\bf (iii)}, the origin $O$ is the unique equilibrium for system  \eqref{hat-Fg}.
    Assume that there is a crossing limit cycle $\Gamma$ for system \eqref{hat-Fg}.
    Then, $\Gamma$ crosses the $y$-axis, the vertical line $x=a_1$, the graph of $y=\hat{F}(x)$, the vertical line $x=x_0$ successively.
    Denote the intersections by $A$, $B$, $C$, $D$, $E$, $F$, $G$, $H$ respectively, as shown in Fig. \ref{fig-wF} (a).

     \begin{figure}[!htb]
		\centering
		\subfloat[in the $xy$-plane]
		{\includegraphics[scale=0.25]{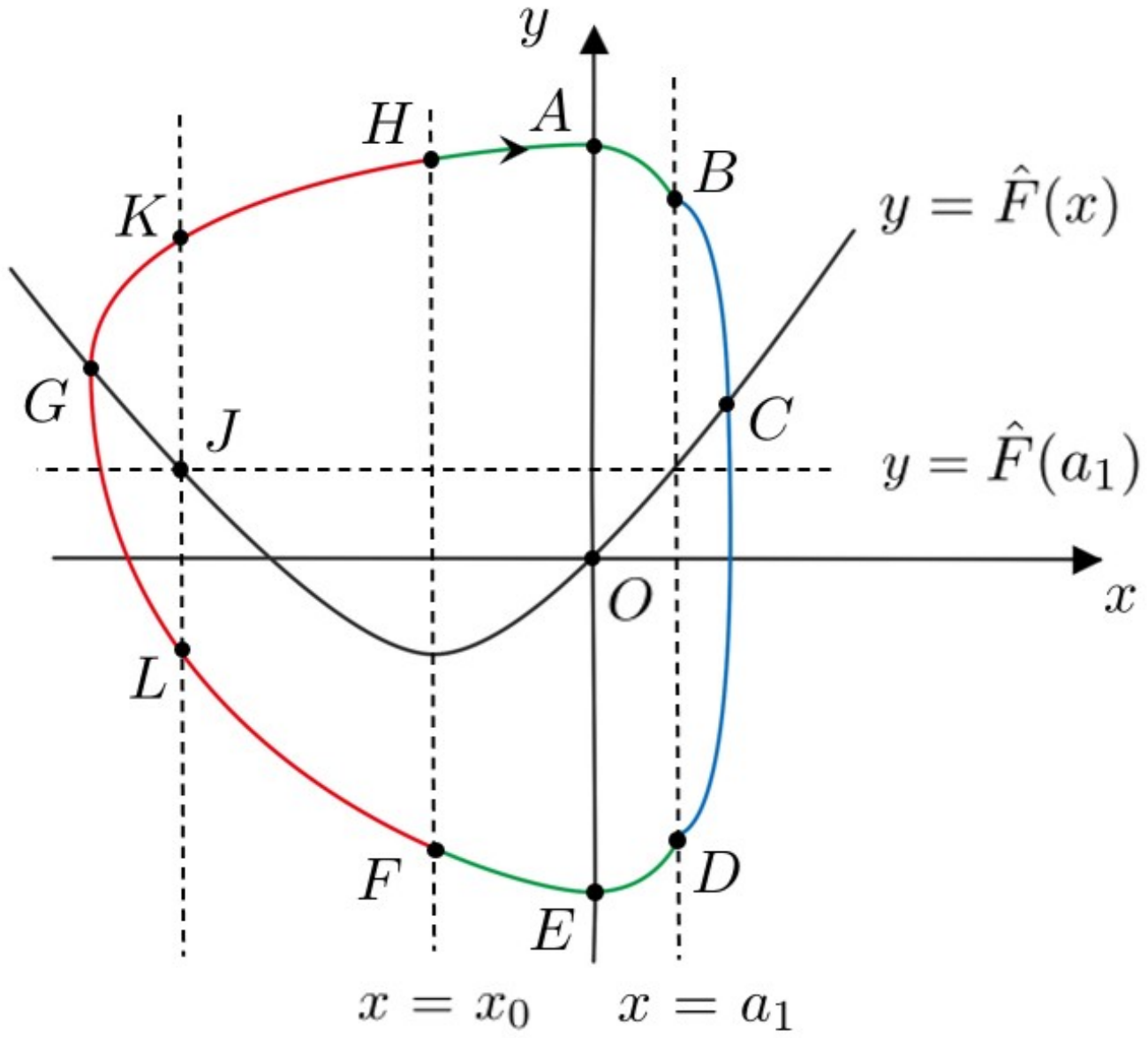}}\hspace{10pt}
        \hskip 0.2cm
		\subfloat[in the $wy$-plane]
		{\includegraphics[scale=0.25]{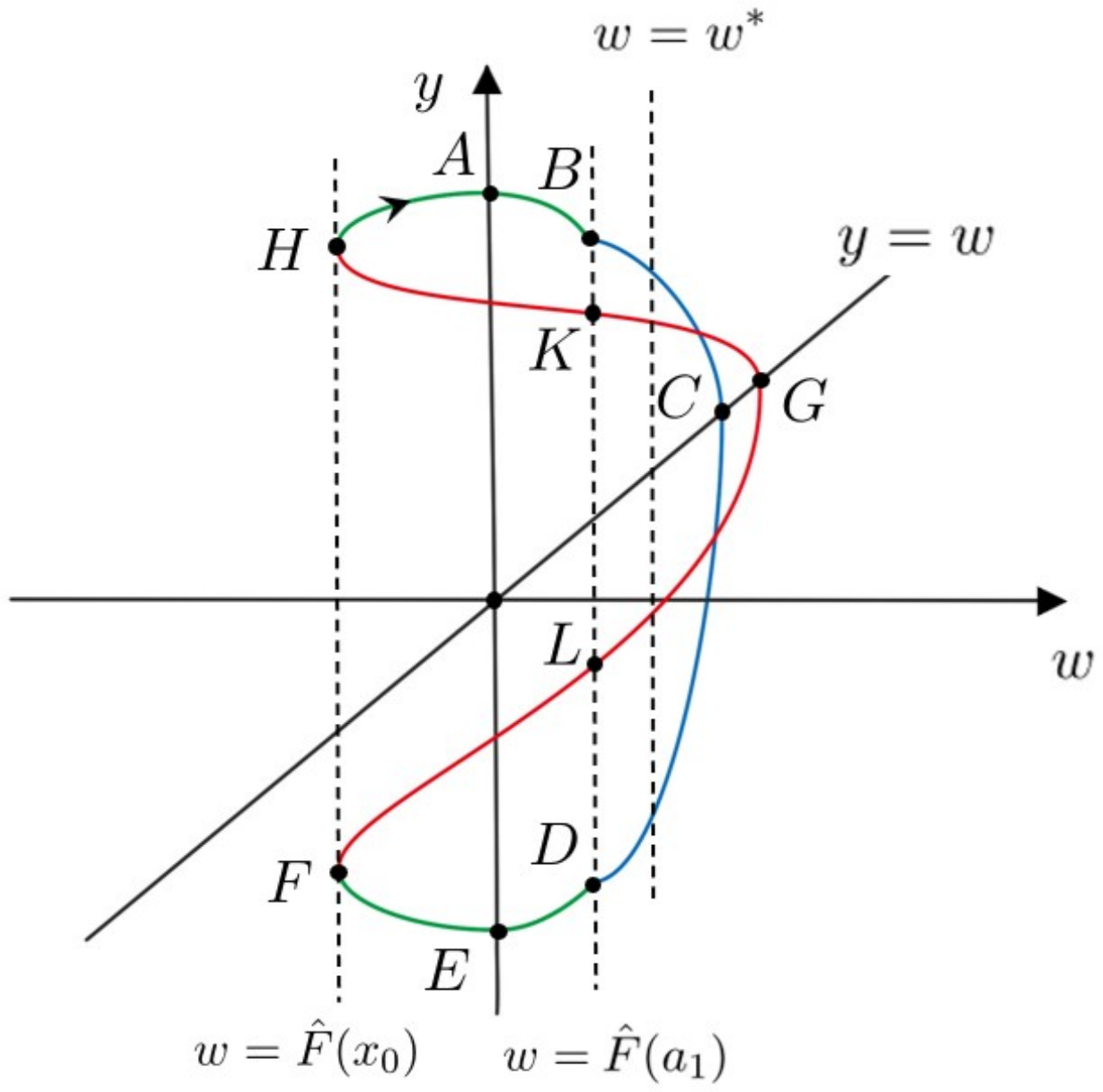}}
		\caption{{\footnotesize A limit cycle of system \eqref{hat-Fg}.}}
		\label{fig-wF}
	\end{figure}

    Let $w=\hat{F}(x)$ for $x\in (\alpha,\beta)$. By condition {\bf (i)}, $w=\hat{F}(x)$ has two inverse functions
    $x=F_1(w)$, $ x_0\le x< \beta$ and $x=F_2(w)$, $\alpha < x\le x_0$.
    System \eqref{SD} can be rewritten as
	\begin{eqnarray}
	\frac{\mathrm{d}y}{\mathrm{d}w}=\frac{\lambda_i(w)}{-y+w}, \hskip 0.5cm {\rm where}\hskip 0.2cm\lambda_i(w)=\frac{\hat{g}(F_i(w))}{\hat{f}(F_i(w))},~~i=1,2.
	\label{wy-cop}
	\end{eqnarray}
    Thus, the limit cycle $\Gamma$ in the $xy$-plane will be changed into two sub-orbits $\widehat{HBCDF}$ and $\widehat{FGH}$ in the $wy$-plane.

        \begin{figure}[!htb]
		\centering
		\subfloat[in the $wy$-plane]
		{\includegraphics[scale=0.25]{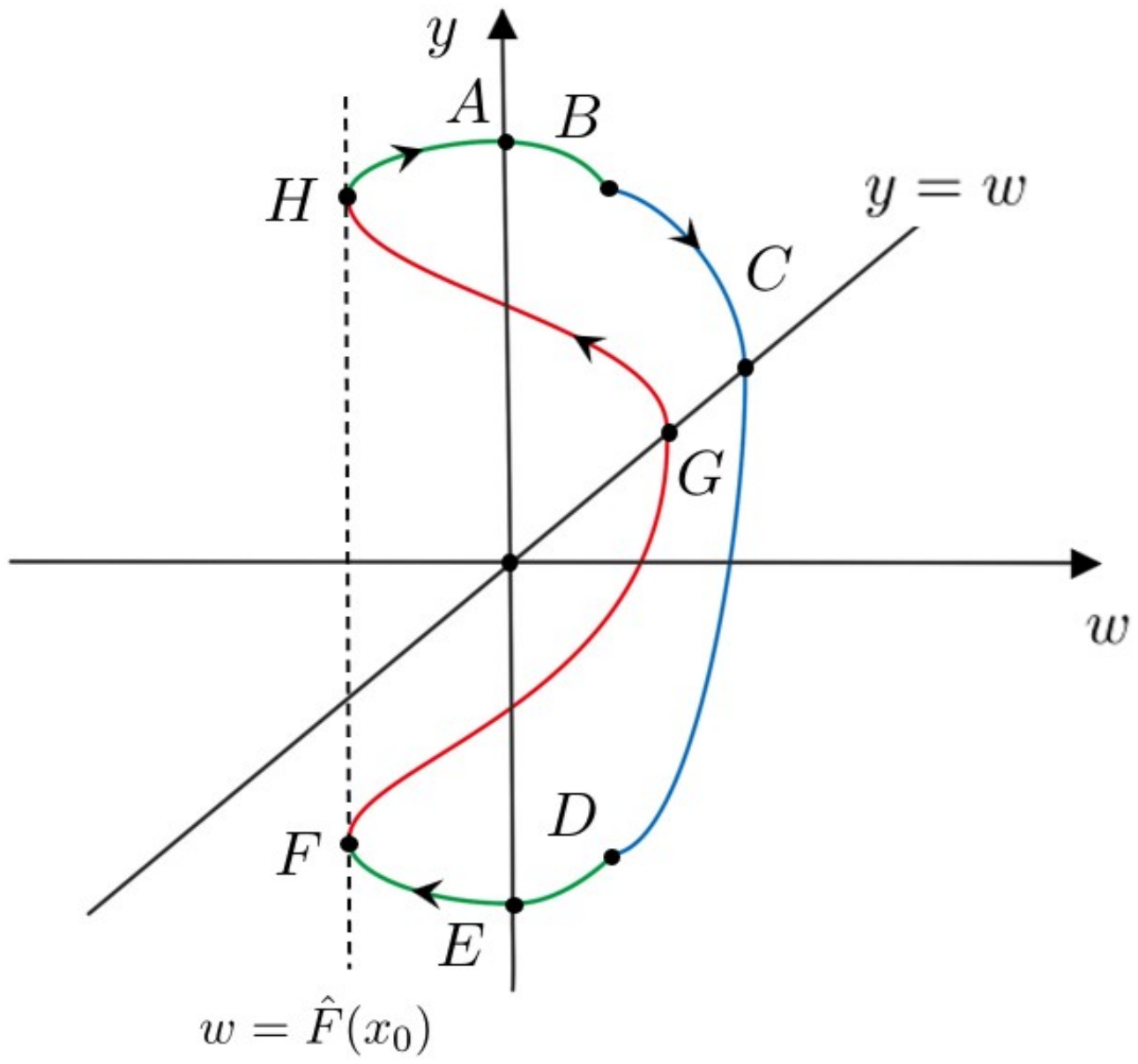}}\hspace{10pt}
        \hskip 0.2cm
		\subfloat[in the $wy$-plane]
		{\includegraphics[scale=0.25]{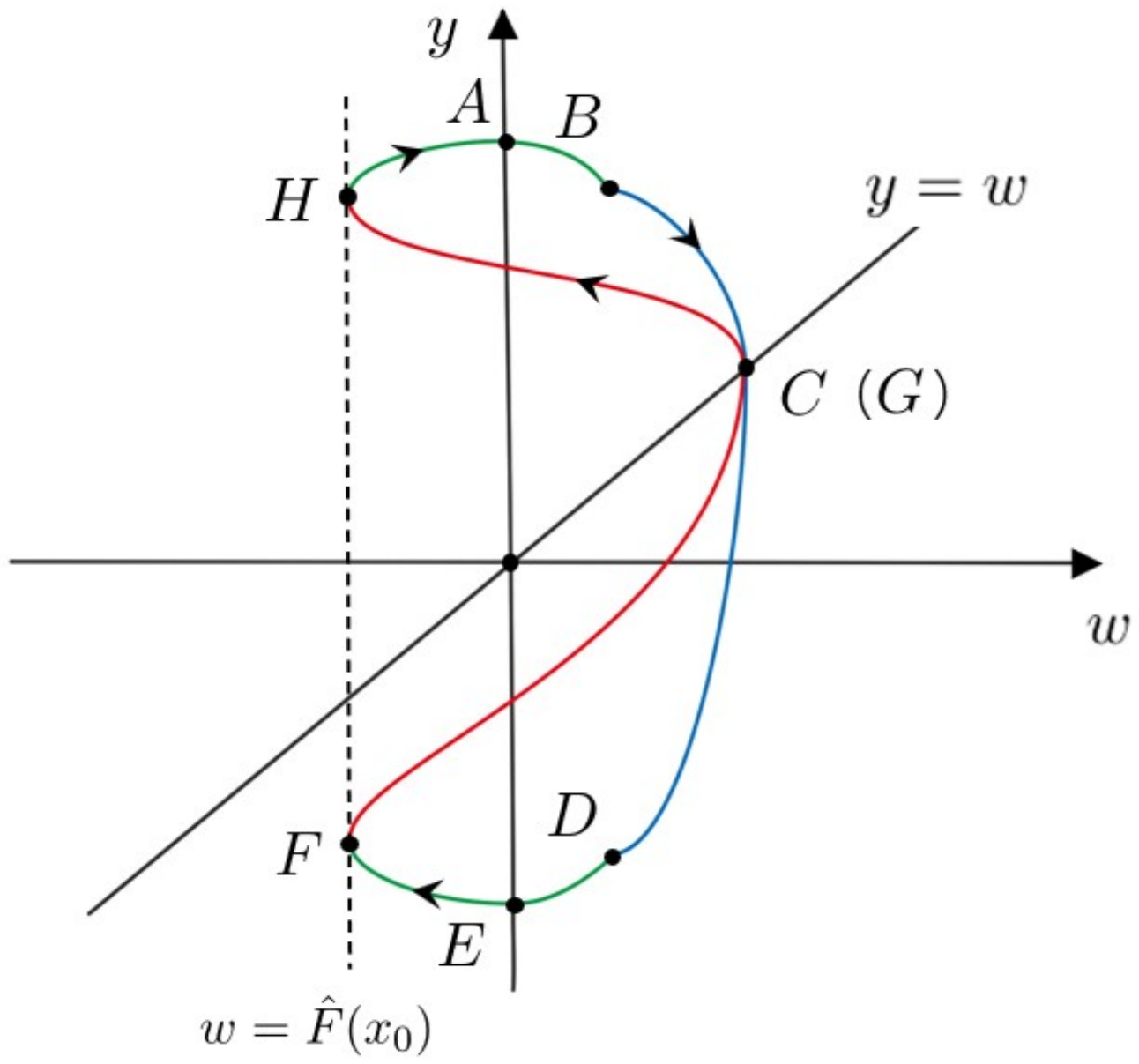}}
		\caption{{\footnotesize A limit cycle of system \eqref{hat-Fg}.}}
		\label{fig-wFN}
	    \end{figure}

    Firstly, we claim that in the $wy$-plane except at the points $H$ (resp. $F$ ) and $C$, the sub-orbit $\widehat{GH}$ (resp. $\widehat{FG}$) must
    intersect the sub-orbit $\widehat{HBC}$ (resp. $\widehat{CDF}$ )
    at least once.
    If not, the sub-orbit $\widehat{FGH}$ is on the left of $\widehat{HBCDF}$ in the $wy$-plane
    because
    \begin{eqnarray}
    \lim_{x\rightarrow x_0^+}\lambda(x)=-\infty,~~~~~~\lim_{x\rightarrow x_0^-}\lambda(x)=+\infty.\label{lamx0}
    \end{eqnarray}
    Then, $\widehat{HBCDF}$ and $\widehat{FGH}$  clockwise bound a region or a union of two regions,
    as shown in Fig. \ref{fig-wFN} (a) or (b).
    Denote the interior region
    of $\Gamma$ in $xy$-plane by $\Omega_{xy}$ and the one of $\widehat{HBCDF}\cup\widehat{FGH}$ in the $wy$-plane by $\Omega_{wy}$.
    By Green's formula,
	\begin{eqnarray*}
		0=\oint_{\Gamma}-\hat{g}(x)\mathrm{d}x-(y-\hat{F}(x))\mathrm{d}y=-\iint_{\Omega_{xy}}\hat{f}(x)\mathrm{d}x\mathrm{d}y
		=-\iint_{\Omega_{wy}}\mathrm{d}w\mathrm{d}y<0,
	\end{eqnarray*}
	which is a contradiction.

    Secondly, we will prove that the sub-orbit $\widehat{GH}$ (resp. $\widehat{FG}$) is below (resp. above)
    the sub-orbit $\widehat{HBC}$ (resp. $\widehat{CDF}$) on the left of $w=w^*>\hat{F}(a_1)$ and they cross each other
    exactly once on the right of $w=w^*$,
    where $w^*=\hat{F}(x_1^*)=\hat{F}(x_2^*)$ and $(x_1^*, \,x_2^*)$ is the unique solution of simultaneous equations \eqref{Flam}
    such that $x_1^* < x_0 < 0<a_1< x_2^*$.

    By condition  {\bf (iv)}, the simultaneous equations \eqref{Flam}  have no solutions such that $x_1 < x_0 < 0 < x_2<a_1$.
    It follows from  \eqref{wy-cop} and \eqref{lamx0}
    that the sub-orbit $\widehat{GH}$ (resp. $\widehat{FG}$) is below (resp. above)
    the sub-orbit $\widehat{HBC}$ (resp. $\widehat{CDF}$) on the left of $w=\hat{F}(a_1)$ in the $wy$-plane.
    The horizontal line $y=\hat{F}(a_1)$ intersects $y=\hat{F}(x)$ at two points in the $xy$-plane, denoted the left one by $J$.
    If $J$ is on $\Gamma$ or outside of the region enclosed by $\Gamma$,
    the sub-orbit $\widehat{GH}$ (resp. $\widehat{FG}$) does not
    intersect the sub-orbit $\widehat{HBC}$ (resp. $\widehat{CDF}$) except at the point $H$ (resp. $F$).
    That conflicts the existence of $\Gamma$.
    Passing through $J$ the vertical line crosses
    $\Gamma$ at $K$ and $L$, from top to bottom.
    Clearly,
    \begin{eqnarray}
    y_K<y_B,~~~~~ y_L>y_D,
    \label{a0-ws}
    \end{eqnarray}
    where $y_B$, $y_D$, $y_K$ and $y_L$ are ordinates of $B$, $D$, $K$ and $L$, respectively.
    On account of condition {\bf (iv)} and \eqref{lamx0}, we have
    $
    \lim_{x\rightarrow a_1^-}\lambda(x)\le\lambda(x_J),
    $
    where $x_J$ is the abscissa of $J$.
    By condition {\bf (iii)} and $\hat{f}(a_1)>0$,
    one can get that
    $
    \lim_{x\rightarrow a_1^+}\lambda(x)<\lim_{x\rightarrow a_1^-}\lambda(x)\le\lambda(x_J)
    $.
    Then, 
    $\lambda_1(w)<\lambda_2(w)$ for $w<w^*$.
    It follows from \eqref{wy-cop} and \eqref{a0-ws}
    that the sub-orbit $\widehat{GH}$ (resp. $\widehat{FG}$) is below (resp. above)
    the sub-orbit $\widehat{HBC}$ (resp. $\widehat{CDF}$) on the left of $w=w^*$.

    Assume that
    the sub-orbit $\widehat{GH}$  intersects
    the sub-orbit $\widehat{HBC}$  more than once on the right of $w=w^*$.
    Denote the abscissas of the two leftmost intersections by $w_1$ and $w_2$ successively.
    Then,  $\lambda_1(w_*)<\lambda_2(w_*)$,  $\lambda_1(w_1)>\lambda_2(w_1)$ and  $\lambda_1(w_2)<\lambda_2(w_2)$,
    indicating that there are at least two solutions of the simultaneous equations \eqref{Flam} such that  $x_1 < x_0 < 0<a_1< x_2$.
    Then,  the sub-orbits $\widehat{GH}$ and $\widehat{HBC}$  intersect
    exactly once on the right of $w=w^*$.
    So do $\widehat{FG}$ and  $\widehat{CDF}$.

    Thirdly, we compare the integrals of divergence along $\widehat{HBCDF}$ and $\widehat{FGH}$.
    Let $y=y_1(w)$ and $y=y_2(w)$ be the functions of  $\widehat{HBC}$ and $\widehat{GH}$ respectively.
    \begin{eqnarray}
    \hskip -2cm
    &&\int_{\widehat{HBC}|_{w<w^*}}{\rm div}(y-\hat{F}(x), -\hat{g}(x))\mathrm{d}t+\int_{\widehat{GH}|_{w<w^*}}{\rm div}(y-\hat{F}(x), -\hat{g}(x))\mathrm{d}t\nonumber
    \\
    &&=
    \int_{\widehat{HBC}|_{w<w^*}}-\hat{f}(x)\mathrm{d}t+\int_{\widehat{GH}|_{w<w^*}}-\hat{f}(x)\mathrm{d}t\nonumber
    \nonumber\\&&
    =
	\int_{\hat{F}(x_0)}^{w^*}
	\left(
	\frac{1}{w-y_1(w)}-\frac{1}{w-y_2(w)}
	\right)\mathrm{d}w
	\nonumber\\&&=
	\int_{\hat{F}(x_0)}^{w^*}
	\frac{y_1(w)-y_2(w)}{(w-y_1(w))(w-y_2(w))}\mathrm{d}w>0.\label{WL1}
	\end{eqnarray}
    Similarly,
    \begin{eqnarray}
    \int_{\widehat{CDF}|_{w<w^*}}{\rm div}(y-\hat{F}(x), -\hat{g}(x))\mathrm{d}t+\int_{\widehat{FG}|_{w<w^*}}{\rm div}(y-\hat{F}(x), -\hat{g}(x))\mathrm{d}t>0.
    \label{WL2}
	\end{eqnarray}
    Since the sub-orbits  $\widehat{GH}$ and  $\widehat{HBC}$  intersect
    exactly once, we get $w_G>w_C>w_*$,
    where $w_C$, $w_G$ are abscissas of $C$, $G$, respectively.
    Then,
    \begin{eqnarray}
    \hskip -2cm
    &&\int_{\widehat{HBC}|_{w>w^*}}{\rm div}(y-\hat{F}(x), -\hat{g}(x))\mathrm{d}t+\int_{\widehat{GH}|_{w>w^*}}{\rm div}(y-\hat{F}(x), -\hat{g}(x))\mathrm{d}t
    \nonumber
    \\&&
    =\int_{w^*}^{w_C}
	\frac{1}{w-y_1(w)}\mathrm{d}w-\int_{w^*}^{w_G}
	\frac{1}{w-y_2(w)}\mathrm{d}w
    \nonumber
    \\&&
    =\int_{w^*}^{w_G}
	\frac{1}{\bar{w}-\bar{y}_1(\bar{w})}\mathrm{d}\bar{w}-\int_{w^*}^{w_G}
	\frac{1}{w-y_2(w)}\mathrm{d}w
	\nonumber\\&&=
	\int_{w^*}^{w_G}
	\frac{\bar{y}_1({w})-y_2(w)}{(w-\bar y_1(w))(w-y_2(w))}\mathrm{d}w,\label{WR}
	\end{eqnarray}
    where
    $\bar{y}_1(\bar{w})$  and $y_2(w)$ satisfy
    \begin{eqnarray}
	\frac{\mathrm{d}\bar{y}_1}{\mathrm{d}\bar{w}}=\frac{\mu\lambda_1((\bar{w}-w_*)/\mu+w_*)}{-\bar{y}_1+\bar{w}},~~~~~~~\bar{y}_1(w_G)=w_G
	\label{wy-bar}
	\end{eqnarray}
    and
   \begin{eqnarray}
	\frac{\mathrm{d}y_2}{\mathrm{d}w}=\frac{\lambda_2(w)}{-y_2+w},~~~~~~~{y_2}(w_G)=w_G
	\label{wy-cop2}
	\end{eqnarray}
   respectively,  $\bar{w}=\mu(w-w^*)+w^*$, $\bar{y}_1=\mu(y_1-w^*)+w^*$
   and $\mu=(w_G-w_*)/(w_C-w_*)>1$.
   Rewrite \eqref{wy-bar} and \eqref{wy-cop2} as
   \begin{eqnarray}
	\frac{\mathrm{d}{y}}{\mathrm{d}{w}}=\frac{\mu\lambda_1(({w}-w_*)/\mu+w_*)}{-{y}+{w}}= \frac{\lambda_1(({w}-w_*)/\mu+w_*)}{({w}-w_*)/\mu+w_*}\cdot \frac{{w}+(\mu-1) w_*}{-{y}+{w}}
   \label{wy-c1}
	\end{eqnarray}
   and
   \begin{eqnarray}
	\frac{\mathrm{d}y}{\mathrm{d}w}=\frac{\lambda_2(w)}{-y+w}=\frac{\lambda_2(w)}{w}\cdot\frac{w}{-y+w}.
    \label{wy-c2}	
   \end{eqnarray}
   Since $\mu>1$ and $y>w$ along $\widehat{GH}$ and $\widehat{HBC}$,
   \[
   \frac{{w}+(\mu-1) w_*}{-{y}+{w}}<\frac{w}{-y+w}.
   \]
   Notice that $\lambda_1(w)>\lambda_2(w)$ for $w>w^*$.
   By condition {\bf (v)},
   \[
    \frac{\lambda_1(({w}-w_*)/\mu+w_*)}{({w}-w_*)/\mu+w_*}> \frac{\lambda_2(({w}-w_*)/\mu+w_*)}{({w}-w_*)/\mu+w_*}>\frac{\lambda_2(w)}{w}.
   \]
   From \eqref{WR}, \eqref{wy-c1} and \eqref{wy-c2}, we get that $\bar{y}_1(\bar{w})>y_2(\bar{w})$
   and
   \begin{eqnarray}
    \int_{\widehat{HBC}|_{w>w^*}}{\rm div}(y-\hat{F}(x), -\hat{g}(x))\mathrm{d}t+\int_{\widehat{GH}|_{w>w^*}}{\rm div}(y-\hat{F}(x), -\hat{g}(x))\mathrm{d}t
    >0.
    \label{WR1}
	\end{eqnarray}
    Similarly,
    \begin{eqnarray}
    \int_{\widehat{CDF}|_{w>w^*}}{\rm div}(y-\hat{F}(x), -\hat{g}(x))\mathrm{d}t+\int_{\widehat{FG}|_{w>w^*}}{\rm div}(y-\hat{F}(x), -\hat{g}(x))\mathrm{d}t>0.
    \label{WR2}
	\end{eqnarray}
    From \eqref{WL1}, \eqref{WL2}, \eqref{WR1} and \eqref{WR2}
    \begin{eqnarray*}
    \int_{\widehat{CDF}}{\rm div}(y-\hat{F}(x), -\hat{g}(x))\mathrm{d}t+\int_{\widehat{FG}}{\rm div}(y-\hat{F}(x), -\hat{g}(x))\mathrm{d}t>0.
	\end{eqnarray*}
    Therefore, $\Gamma$ is unstable and hyperbolic.
	Furthermore, the two limit cycles are both unstable if system \eqref{hat-Fg} exhibits two limit cycles.
	By the Poincar\'e-Bendixson Theorem, there is a stable limit cycle which lies in the annulus of the two unstable limit cycles.
	This is a contradiction.
	The uniqueness of limit cycles is proven.
  \end{proof}

The nonexistence of closed orbits  for system \eqref{SD} can be obtained by Bendixson-Dulac Criterion
when $b\ge 0$.

\begin{lm}
	There are no closed orbits  for system \eqref{SD} when $b\ge 0$.
	\label{lm-bg0}
\end{lm}

\begin{proof}
  The divergence of system (\ref{SD}) is given by
	\[
	{\rm div}(y-F(x), -g(x))=-f(x)=-\delta( x^2+b )\le0
	\]
  when $b\ge0$.
  By Bendixson-Dulac Criterion, system (\ref{SD}) exhibits no closed orbits.
\end{proof}

It is similar to the proof of the nonexistence of closed orbits for \eqref{SD} with $a=1$ and $-4\le b<0$ in \cite{CTW}.
We can use \cite[Lemma 3.1]{CTW} to get that
system \eqref{SD}  exhibits no closed orbits  when $-(a+1)^2\le b<0$.

\begin{lm}
	There are no closed orbits  for system \eqref{SD} when $-(a+1)^2\le b<0$.
	\label{lm-b0H}
\end{lm}

\begin{proof}
  By the transformation
	\begin{eqnarray*}
	(x,y)\rightarrow(x+(a+1), y+F(a+1)),
	\end{eqnarray*}
  system \eqref{SD} can be reduced to
	\begin{eqnarray}
	\dot{x}=y-\bar{F}(x),~~~~~~ \dot{y}=-\bar{g}(x),
    \label{F-g-x-z}
	\end{eqnarray}
  where
  $\bar F(x)=\delta\left(x^3/3+(a+1)x^2+((a+1)^2+b)x\right)$
  and   $\bar g(x)=x+1-{\rm sgn}(x+a+1)$.
  The nonexistence of limit cycles  for system \eqref{SD} can be obtained by verifying the
  conditions in \cite[Lemma 3.1]{CTW} with
  $(\alpha,\beta)=(-\infty,+\infty)$.
  Clearly, $\bar g(x)$ has a unique jump discontinuity at $x=-(a+1)$ and a  zero $x=0$.
  One can check that $x\bar g(x) \ge 0$ for all $x\in (-\infty,+\infty)$
  and
  \begin{eqnarray*}
  z(x)=\int_{0}^{x}\bar g(s)\mathrm{d}s=
  \left\{
  \begin{array}{ll}
  {x^2}/{2},~~~~~ &{\rm if} ~~  x\ge-(a+1),\\
  {x^2}/{2}+2x+2(a+1),~~~~~&{\rm if} ~~ x<-(a+1).
  \end{array}
  \right.
  \end{eqnarray*}
  The inverse functions of $z(x)$ are
  $
  x_1(z)= \sqrt{2z}
  $
  and
  \begin{eqnarray*}
  x_2(z)=
  \left\{\begin{array}{ll}
  -\sqrt{2z}, ~~~~ &{\rm if} ~~ 0\le z\le(a+1)^2/2,
  \\
  -2-\sqrt{2z-4a}, ~~~~ &{\rm if} ~~ z>(a+1)^2/2.
  \end{array}
  \right.
  \end{eqnarray*}
  We claim that
  \[
  \bar F(x_1)=\bar F(x_2), ~~~~~~  z(x_1)=z(x_2)
  \]
  have no solutions satisfying $-(a+1)\le x_2<0<x_1$. In fact, from
  $
  z(x_1)-z(x_2)=(x_1-x_2)(x_1+x_2)/2=0
  $
  we get
  $x_2=-x_1$.
  Substituting it into $\bar F(x_1)-\bar F(x_2)$, one can calculate
  \begin{eqnarray}
  \bar F(x_1)-\bar F(x_2)
  =\frac{2\delta x_1}{3}\left(x_1^2+3(a+1)^2+3b\right)>0,
  \label{Fx12}
  \end{eqnarray}
  meaning that $\bar F(x_1(z))\ne\bar F(x_2(z))$ for all $0<z\le z(-(a+1))=(a+1)^2/2$.
  Then, condition {(i)} of \cite[Lemma 3.1]{CTW} holds.

  On the one hand, it follows from \eqref{Fx12} that
  $\bar F(x_1(z))>\bar F(x_2(z))$ for all $0<z\le (a+1)^2/2$.
  On the other hand, one can check that $x_1((a+1)^2/2)=a+1$ and $x_2((a+1)^2/2)=-(a+1)$.
  Since $y=\bar F(x)$ is increasing
  on $x>(a+1)$ and $x_1(z)$ is increasing on $z>0$, we get
  $\bar F(x_1(z))>\bar F(a+1)$ when $z>(a+1)^2/2$.
  Notice that $y=\bar F(x)$ has an absolute maximum value at $x=-\sqrt{-b}-a-1$
  on the interval $(-\infty, 0]$ and  $x_2(z)$ is decreasing on $z>0$.
  Then,
  $\bar F(x_2(z))<\bar F(-\sqrt{-b}-a-1)$ when $z>(a+1)^2/2$.
  By direct calculation,
  \[
  \bar F(a+1)-\bar F(-\sqrt{-b}-a-1)=\frac{2\delta}{3}\left(2(a+1)+\sqrt{-b}\right)^2\left(a+1-\sqrt{-b}\right)\ge0,
  \]
  yielding that $\bar F(x_1(z))\ge\bar F(x_2(z))$ for all $z> (a+1)^2/2$.
  Thus, $\bar F(x_1(z))\ge\bar F(x_2(z))$  for all $z> 0$,
  indicating that condition {(ii)} of \cite[Lemma 3.1]{CTW} holds.
  Therefore, system \eqref{F-g-x-z} exhibits no limit cycles when $b\ge -(a+1)^2$ by \cite[Lemma 3.1]{CTW}.
  Consequently, there are no limit cycles  for system \eqref{SD} when
  $-(a+1)^2\le b<0$.
  \end{proof}

The closed orbit bifurcations may occur when $b<-(a+1)^2$.
Notice that there is a  switching line $x=0$ for system  \eqref{SD}.
For simplicity,
let {\it crossing limit cycles} be the limit cycles crossing the line $x=0$,
{\it grazing limit cycles} be the ones touching the line $x=0$
and {\it small limit cycles} be the ones that only lie on the right hand side of $x=0$.

The following lemma tells us an upper bound on the number of small limit cycles  when $b< -(a+1)^2$.

\begin{lm}
	 For system  \eqref{SD},
   \begin{description}
		\item{\bf (a)} there is at most one small limit cycle  when  $-4(a+1)^2/3 <b<-(a+1)^2$,
   	and the small limit cycle is stable and hyperbolic if it exists;
        \item{\bf (b)} there are no small limit cycles  when $b \leq-4(a+1)^2/3$.
	\end{description}
	\label{one-sl}
\end{lm}

\begin{proof}
  By the transformation
	\begin{eqnarray*}
	(x,y)\rightarrow(-x+(a+1), y+F(a+1)),
	\end{eqnarray*}
  the equilibrium $E_r$ is  translated to the origin,
  the orbits are reflected about the $y$-axis
  and
  system \eqref{SD} can be rewritten as
	\begin{eqnarray}
	\dot{x}=\tilde{F}(x)-y,~~~~~~ \dot{y}=\tilde{g}(x),
    \label{F-g-x}
	\end{eqnarray}
  where $\tilde{F}(x)=-\delta({x^3}/{3}-(a+1)x^2+((a+1)^2+b) x)$
  and $\tilde{g}(x)=x+{\rm sgn}(a+1-x)-1$.

  Firstly, we prove that system \eqref{F-g-x}
  exhibits  at most one limit cycle in the strip $x\in(-\infty, a+1)$ and it  is stable and hyperbolic if it exists.
  Clearly, $\tilde{g}(x)=x$ for  $x\in(-\infty, a+1)$.
  Notice that both $\tilde{F}(x)$ and $\tilde{g}(x)$ are analytic on $(-\infty, a+1)$.
  We claim that system \eqref{F-g-x}
  satisfies all conditions in \cite[Theorem 2.1]{DR}.
  Note that $\tilde{f}(x)=\tilde{F}'(x)=-\delta\left(x^2-2(a+1)x+(a+1)^2+b\right)$.
  One can check that $\tilde{f}(x)$ has a unique zero $x_0=a-\sqrt{-b}+1$ on the interval $(-\infty,a+1)$
  and $\tilde f(x)<0$ (resp. $>0$) as $x<x_0$  (resp. $x_0<x<a+1$).
  Then, condition (i) of \cite[Theorem 2.1]{DR} holds.
  By calculation,  $\tilde{F}(0)=\tilde{F}(\xi_0)=0$, where
  \begin{eqnarray*}
  \xi_0=\frac{3(a+1)}{2}-\frac{\sqrt{-3(a+1)^2-12b}}{2}
  <a+1-\left(-\frac{\sqrt{-b}}{2}+\frac{\sqrt{3b-12b}}{2}\right)=x_0
  \end{eqnarray*}
  because $b< -(a+1)^2$.
  Then, condition (ii) of \cite[Theorem 2.1]{DR} holds.
  Clearly, $x\tilde{g}(x)=x^2>0$ for $x\in(-\infty,0)\cup(0,a+1)$.
  Then, condition (iii) of \cite[Theorem 2.1]{DR} holds.
  From  $\tilde g(x_1)/\tilde f(x_1)=\tilde g(x_2)/\tilde f(x_2)$, we get
  $x_2=\left((a+1)^2+b\right)/x_1$. Substituting it to $\tilde F(x_1)-\tilde F(x_2)$, one can calculate
  $\tilde F(x_1)-\tilde F(x_2)=-\delta(x_1-x_2)s(x_1)/(3x_1^2)$,
  where
  \[
  s(x):=x^4-3(a+1)x^3+4\left((a+1)^2+b\right)x^2-3(a+1)\left((a+1)^2+b\right)x+\left((a+1)^2+b\right)^2.
  \]
  Since there are at most four zeros for $s(x)$ and
  $s(x_0)=s(a-\sqrt{-b}+1)=2b(a-\sqrt{-b}+1)^2<0$,
  equations
	\begin{eqnarray}
	\tilde F(x_1)=\tilde F(x_2),\hskip 0.5cm \frac{\tilde g(x_1)}{\tilde f(x_1)}=\frac{\tilde g(x_2)}{\tilde f(x_2)}
    \label{Fx1x2}
	\end{eqnarray}
  with $x_1<x_0<0<x_2$ have at most one solution.
  Then, condition (iv) of \cite[Theorem 2.1]{DR} holds.
  One can compute that ${\mathrm{d}\left( {\tilde F(x)\tilde f(x)}/{\tilde g(x)}\right)}/{\mathrm{d}x}={\delta^2}\kappa(x)/3$,
  where
	\[
    \kappa(x):=4x^3-15(a+1)x^2+\left(20(a+1)^2+8b\right)x-9(a+1)\left((a+1)^2+b\right).
    \]
  It is easy to compute that
	\[
	\kappa(\xi_0)=\frac{1}{2}\left((a+1)^2+4b\right)\left(-3(a+1)+\sqrt{-3(a+1)^2-12b}\right)<0
	\]
  because $b< -(a+1)^2$.
  Moreover, $\mathrm{d} \kappa(x)/\mathrm{d}x=0$ has two roots
  \[
  \eta_1=\frac{5(a+1)}{4}-\frac{\sqrt{-15(a+1)^2-96b}}{12},~~~~~~~~
  \eta_2=\frac{5(a+1)}{4}+\frac{\sqrt{-15(a+1)^2-96b}}{12}
  \]
  and $\xi_0<\eta_1<\eta_2$.
  Then, $\kappa(x)<0$ for $x<\xi_0$,
  implying that ${\tilde F(x)\tilde f(x)}/{\tilde g(x)}$ is decreasing on $(-\infty,\xi_0)$.
  Thus, condition (v) of \cite[Theorem 2.1]{DR} holds.
  Therefore,  system \eqref{F-g-x} exhibits at most one limit cycle in the region $x\in(-\infty,a+1)$ by \cite[Theorem 2.1]{DR}.
  Moreover, the limit cycle is stable and hyperbolic if it exists.

  Secondly, we prove that  system \eqref{F-g-x} exhibits no limit cycles in the region $x\in(-\infty, a+1)$  when $b \leq-4(a+1)^2/3$.
  Let $G(x):=\int_{0}^{x} \tilde{g}(x) \mathrm{d}x=x^2/2$.
  From $G(x_1)=G(x_2)$ we can get $x_2=-x_1$.
  Substituting it into $\tilde F(x_1)-\tilde F(x_2)$,
  we have
  \[
  \tilde F(x_1)-\tilde F(x_2)=-2x_1\delta(x_1^2/3+(a+1)^2+b)\ge0
  \]
  for $x_2<0<x_1<a+1$ as $b \leq -4(a+1)^2/3$.
  By \cite[Proposition 2.1]{CC15},
  system \eqref{F-g-x} exhibits no limit cycles in the region  $x\in(-\infty,a+1)$ when $b \leq-4(a+1)^2/3$.

  In summary, system \eqref{F-g-x} exhibits at most one limit cycle in the region $x\in(-\infty,a+1)$ when  $-4(a+1)^2/3 <b<-(a+1)^2$
  and no limit cycles in the region $x\in(-\infty,a+1)$ when $b \leq-4(a+1)^2/3$.
  Moreover, the limit cycle is stable and hyperbolic if it exists.
  Furthermore, there is at most one small limit cycle of system \eqref{SD} when  $-4(a+1)^2/3 <b<-(a+1)^2$, which is stable and hyperbolic if it exists,
  and no small limit cycles  when $b \leq-4(a+1)^2/3$.
\end{proof}

To investigate the number and stability of crossing limit cycles for system \eqref{SD}, we give a necessary condition for the existence of crossing limit cycles.
The following lemma can be regarded as the extension of \cite[Lemma 3.4]{CTW}, which considers system  \eqref{SD} with $a=1$. Although the proof of the following lemma is similar to the proof of \cite[Lemma 3.4]{CTW},
we restate it for completeness.

\begin{lm}
	If there is a crossing limit cycle for system \eqref{SD} when $b< -(a+1)^2$,
    it must surround $(\sqrt{-3b},0)$.
    \label{lm-c-p}
\end{lm}
\begin{proof}
    Assume that there is a crossing limit cycle $\Gamma$ for system \eqref{SD}.
    Then, $\Gamma$ intersects the $y$-axis at two points, denoted by $A$, $B$ from top to bottom
    and  intersects  $y=F(x)$ at two points, denoted by $C$, $D$ from left to right, as shown in Fig. \ref{fig-posin} (a).
    Let $x_C$ and $x_D$ be the abscissas of $C$ and $D$, respectively. Obviously, $x_C<0<x_D$.
    In order to prove that $\Gamma$ surrounds $(\sqrt{-3b},0)$,
    we assume $0<x_D\le\sqrt{-3b}$ and then derive a contradiction.

    Suppose that $0<x_D\le\sqrt{-3b}$.
    Passing through $A$ and $B$ respectively, the lines perpendicular to the $y$-axis cross $\Gamma$ at two points, denoted by $M$ and $N$.
    Then,
    \begin{eqnarray}
    F(x_M)<0~~~~{\rm and}~~~~  F(x_N)<0,
    \label{FxMN}
    \end{eqnarray}
    where  $x_M$ and $x_N$ are the abscissas of $M$ and $N$, respectively.
    Define
    \begin{eqnarray}
    {\mathcal E}(x,y)=\int_{0}^{x}g(s)\mathrm{d}s+\frac{1}{2}y^2=\frac{1}{2}x^2-({\rm sgn}(x)+a)x+\frac{1}{2}y^2.\label{Energy}
    \end{eqnarray}
    Then,
    \begin{eqnarray*}
    \frac{\mathrm{d}{\mathcal E}(x,y)}{\mathrm{d}t}\Big|_{\eqref{SD}}=-\delta\left(\frac{x^3}{3}+b x\right)(x-{\rm sgn}(x)-a)=-F(x)g(x).
    \end{eqnarray*}
    Let $w=F(x)$ for $x\ge0$. The inverse functions of $w=F(x)$ are
    $x=F_1(x)$, $x\ge\sqrt{-b}$ and $x=F_2(x)$, $0<x\le\sqrt{-b}$.
    System \eqref{SD} can be rewritten as
	\begin{eqnarray*}
	\frac{\mathrm{d}y}{\mathrm{d}w}=\frac{\lambda_i(w)}{-y+w}, \hskip 0.5cm {\rm where}\hskip 0.2cm\lambda_i(w)=\frac{g(F_i(w))}{f(F_i(w))},~~i=1,2.
	\end{eqnarray*}
	Thus, the sub-orbits $\widehat{AM}$ and $\widehat{NB}$ in the $xy$-plane will be changed into the ones in the $wy$-plane.
    Similar to \eqref{Fx1x2},
    \begin{eqnarray*}
	 F(x_1)= F(x_2),\hskip 0.5cm \frac{ g(x_1)}{ f(x_1)}=\frac{ g(x_2)}{ f(x_2)}
	\end{eqnarray*}
    with $0<x_1<a+1<\sqrt{-b}<x_2$ have at most one solution.
    It follows from \eqref{FxMN}
    that both sub-orbits $\widehat{AM}$ and $\widehat{NB}$ in the $wy$-plane do not cross itself,
    implying that $\widehat{AM}\cup\overline{MA}$ and $\widehat{NB}\cup\overline{BN}$ are both simple closed curves in the $wy$-plane.
    Notice that $\overline{MA}$ is a horizontal segment in both $xy$-plane and $wy$-plane.
    By Green's formula,
	\begin{eqnarray}
	\int_{\widehat{AM}}\mathrm{d}{\mathcal E}(x,y)=\int_{\widehat{AM}}F(x)\mathrm{d}y=\int_{\widehat{AM}\cup\overline{MA}}F(x)\mathrm{d}y=\int_{\widehat{AM}\cup\overline{MA}}w\mathrm{d}y=\iint_{S_{AM}}\mathrm{d}w\mathrm{d}y>0,
    \label{E-AM}
	\end{eqnarray}
    where  $S_{AM}$ is the region enclosed by $\widehat{AM}\cup\overline{MA}$ in the $wy$-plane.
    Similarly,
    \begin{eqnarray}
	\int_{\widehat{NB}}\mathrm{d}{\mathcal E}(x,y)>0. \label{E-NB}
	\end{eqnarray}

   \begin{figure}[!htb]
		\centering
		\subfloat[ $-\sqrt{-3b}\le x_C<0$]
		{\includegraphics[scale=0.2]{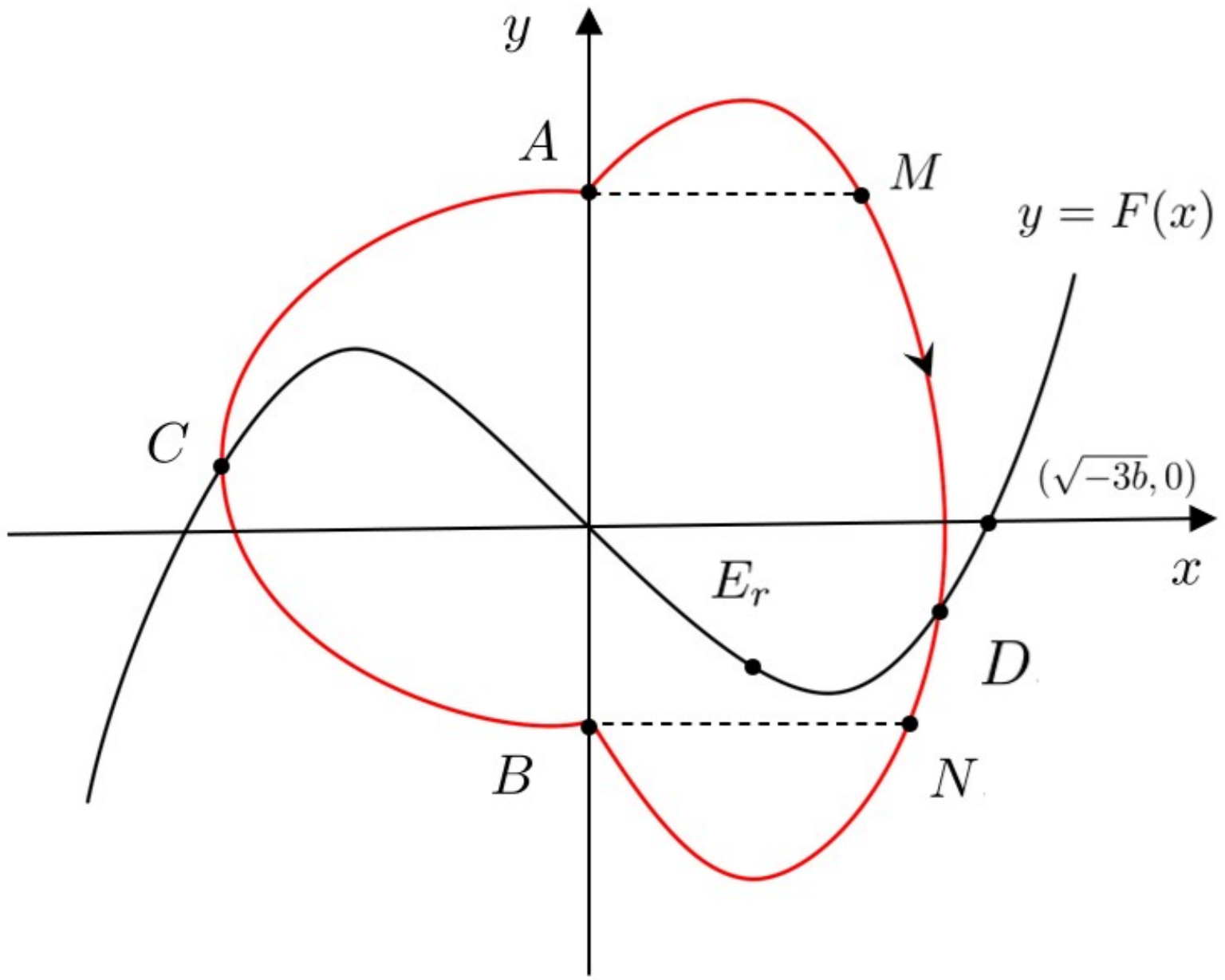}}\hspace{10pt}
        \hskip 1cm
		\subfloat[$x_C<-\sqrt{-3b}$]
		{\includegraphics[scale=0.2]{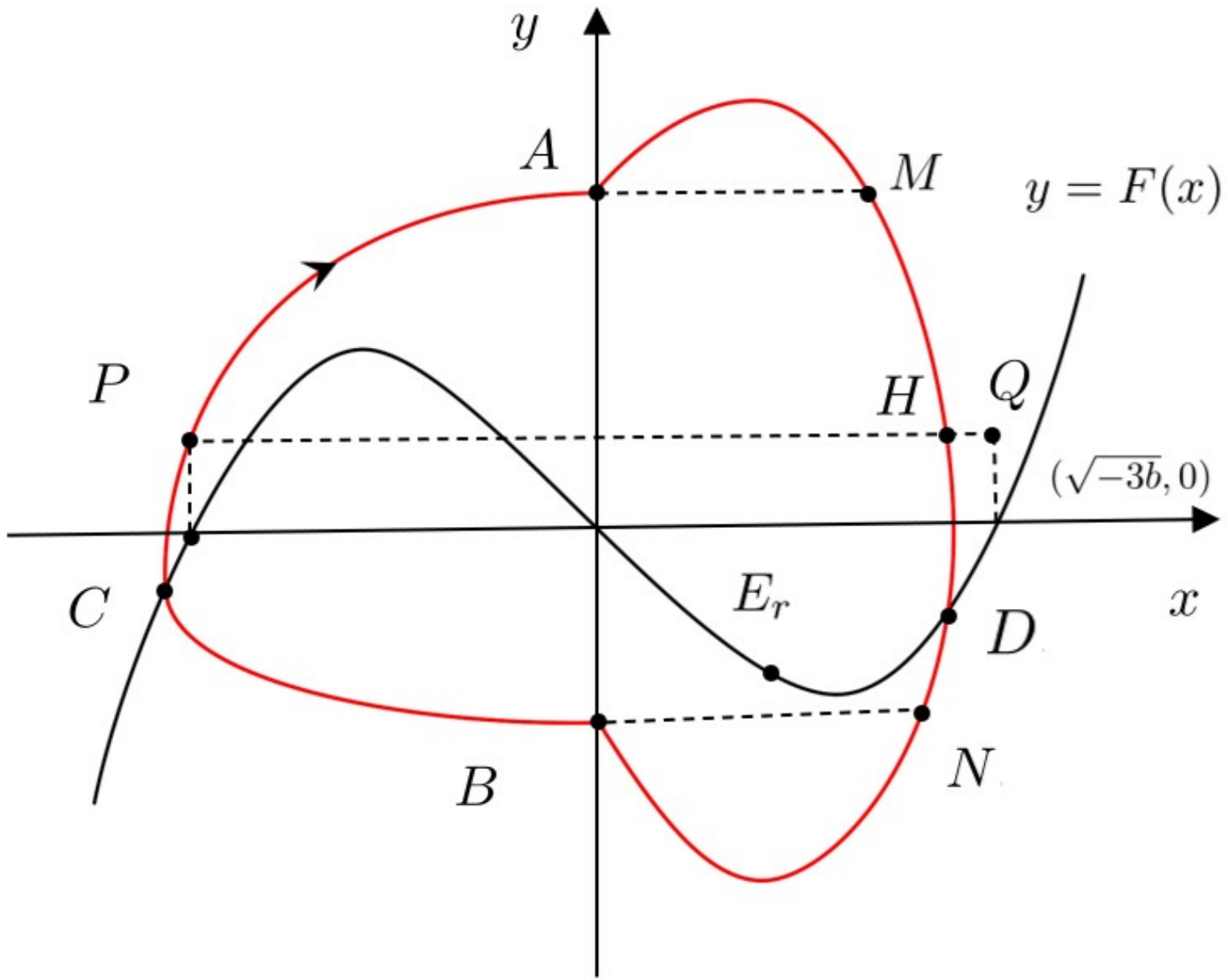}}
		\caption{{\footnotesize Limit cycles in hypothesis.}}
		\label{fig-posin}
	\end{figure}

    In what follows, we discuss two cases $-\sqrt{-3b}\le x_C<0$ and $x_C<-\sqrt{-3b}$.
    Under the assumption $0<x_D\le\sqrt{-3b}$, the contradictions will be deduced, no matter $-\sqrt{-3b}\le x_C<0$ or $x_C<-\sqrt{-3b}$.

    In the case $-\sqrt{-3b}\le x_C<0$, one can check that $F(x)\ge0$ and $g(x)\le0$ along the sub-orbit $\widehat{BCA}$.
    Then,
    \begin{eqnarray}
	\int_{\widehat{BCA}}\mathrm{d}{\mathcal E}(x,y)>0.
    \label{E-BCA}
	\end{eqnarray}
    Since  $0<x_D\le\sqrt{-3b}$, we have $F(x)\le0$ and $g(x)\ge0$ along the sub-orbit $\widehat{MDN}$, implying that
    \begin{eqnarray}
	\int_{\widehat{MDN}}\mathrm{d}{\mathcal E}(x,y)>0.
    \label{E-MND}
	\end{eqnarray}
    It follows from (\ref{E-AM})-(\ref{E-MND}) that
    \begin{eqnarray*}
	\oint_{\Gamma}\mathrm{d}{\mathcal E}(x,y)=\int_{\widehat{AM}\cup\widehat{MDN}\cup\widehat{NB}\cup\widehat{BCA}}\mathrm{d}{\mathcal E}(x,y)>0,
	\end{eqnarray*}
    which is a contradiction.

    In the case $x_C<-\sqrt{-3b}$, passing through $(-\sqrt{-3b},0)$, the line perpendicular to the $x$-axis crosses $\Gamma$ above $x$-axis at a point, denoted by $P$.
    Let $y_{P}$ be the  ordinate of $P$ and denote the point $(\sqrt{-3b}, y_P)$ by $Q$.
    Connect $P$ and $Q$ by a straight segment, which crosses $\Gamma$ at a point, denoted by $H$, as shown in Fig. \ref{fig-posin} (b).
    On the one hand,
    \begin{eqnarray*}
	\int_{\widehat{PA}}\mathrm{d}{\mathcal E}(x,y)>0~~~{\rm and}~~~ \int_{\widehat{MH}}\mathrm{d}{\mathcal E}(x,y)>0
	\end{eqnarray*}
    because  $F(x)g(x)\le0$ along the sub-orbits $\widehat{PA}$ and $\widehat{MH}$.
    From \eqref{E-AM} that
    \begin{eqnarray*}
	\int_{\widehat{PAH}}\mathrm{d}{\mathcal E}(x,y)=\int_{\widehat{PA}\cup\widehat{AM}\cup\widehat{MH}}\mathrm{d}{\mathcal E}(x,y)>0,
	\end{eqnarray*}
    indicating that ${\mathcal E}(H)>{\mathcal E}(P)$.
    On the other hand,
    by \eqref{Energy} we can compute
    ${\mathcal E}(P)=y^2_P/2+(a-1)\sqrt{-3b}-3b/2>y^2_P/2-(a+1)\sqrt{-3b}-3b/2={\mathcal E}(Q)$.
    Since
    \[
    \frac{\partial {\mathcal E}}{\partial x}=g(x)=x-({\rm sgn}(x)+a) >0, ~~~{\rm as}~~~x>a+1,
    \]
    we get ${\mathcal E}(Q)>{\mathcal E}(H)$,
    yielding ${\mathcal E}(P)>{\mathcal E}(H)$.
    This is a contradiction.
    Thus,  $x_D>\sqrt{-3b}$ and $\Gamma$ is  surrounding $(\sqrt{-3b},0)$.
\end{proof}

\begin{lm}
	System \eqref{SD} exhibits at most one crossing limit cycle in the strip $x\ge-\sqrt{-b}$ when $b\le -3(a+1)^2$. 
    Moreover, the limit cycle is stable and hyperbolic if it exists.
	\label{lm-ag1clu}
\end{lm}

\begin{proof}
 By the transformation
	\begin{eqnarray*}
	(x,y,t)\rightarrow(-x+(a+1), y+F(a+1),-t),
	\end{eqnarray*}
  system \eqref{SD} can be rewritten as
	\begin{eqnarray}
	\dot{x}=y-\tilde{F}(x),~~~~~~ \dot{y}=-\tilde{g}(x),
    \label{F-g-x-1}
	\end{eqnarray}
  where $\tilde{F}(x)$ and $\tilde{g}(x)$ are defined in \eqref{F-g-x}.
   We will show that system \eqref{F-g-x-1}
   exhibits  at most one crossing limit cycle in the strip $x\in(-\infty, a+1+\sqrt{-b})$ when $b\le -3(a+1)^2$, and it is unstable and hyperbolic if it exists.
   We will get the conclusions by verifying all the conditions in Lemma \ref {uni-lc-F} for system \eqref{F-g-x-1}.

   It is easy to check that $\tilde{f}(x)$ has a unique zero $x_0=a+1-\sqrt{-b}$ on the interval $(-\infty,a+1+\sqrt{-b})$
   and $\tilde f(x)<0$ (resp. $>0$) as $x<x_0$  (resp. $x_0<x<a+1+\sqrt{-b}$), where $\tilde{f}(x)=\tilde{F}'(x)=-\delta\left(x^2-2(a+1)x+(a+1)^2+b\right)$.
   Then, condition {\bf (i)}  in Lemma \ref {uni-lc-F} holds.
   Notice that all the conditions in \cite[Theorem 2.1]{DR} with $(\alpha,\beta)=(-\infty, a+1)$
   hold for \eqref{F-g-x}, as proven in Lemma \ref{one-sl},
   and conditions {\bf (ii)} and  {\bf (v)} in Lemma \ref {uni-lc-F} with $(\alpha,\beta)=(-\infty, a+1+\sqrt{-b})$
   are totally same with the ones in \cite[Theorem 2.1]{DR} with $(\alpha,\beta)=(-\infty, a+1)$.
   Then, conditions {\bf (ii)} and  {\bf (v)} in Lemma \ref {uni-lc-F} hold.
   In what follows,
   we check conditions {\bf (iii)} and  {\bf (iv)} in Lemma \ref {uni-lc-F} with  $(\alpha,\beta)=(-\infty, a+1+\sqrt{-b})$ for system \eqref{F-g-x-1}.

   Since $a>1$ and
   \begin{eqnarray*}
   \tilde{g}(x)=x+{\rm sgn}(a+1-x)-1=
   \left\{
   \begin{array}{ll}
   x,~~~~~~& {\rm if} ~~x<a+1,\\
   x-1,  & {\rm if} ~~x=a+1,\\
   x-2,  & {\rm if} ~~x>a+1,
   \end{array}
   \right.
   \end{eqnarray*}
   one can calculate $\tilde{g}(x)$ has a unique jump discontinuity point $a_1=a+1\in (0,a+1+\sqrt{-b})$,
   $\lim_{x\rightarrow a_1^+}\tilde{g}(x)=a-1<\lim_{x\rightarrow a_1^-}\tilde{g}(x)=a+1$ and
   $x\tilde{g}(x)>0$ for $x\in (\alpha, \beta)\backslash\{0,a_1\}$.
   Then, condition {\bf (iii)}  in Lemma \ref{uni-lc-F} holds.

   By the proof of Lemma \ref{one-sl},  equations \eqref{Fx1x2}  with $x_1<x_0<0<x_2$ have
   at most one solution on the interval $(\alpha,\beta)=(-\infty, a+1)$.
   Notice that
   $\tilde{F}(a+1)=\tilde{F}(a+1-\sqrt{-3b})$
   and
   \[
   \lambda(a+1)-\lambda(a+1-\sqrt{-3b})=\frac{-3(a+1)+\sqrt{-3b}}{2\delta b}\le0
   \]
   because $b\le-3(a+1)^2$.
   It follows from $\lim_{x\rightarrow x_0^+}\lambda(x)<\lim_{x\rightarrow x_0^-}\lambda(x)$
   that equations \eqref{Flam} in Lemma \ref{uni-lc-F} have no solutions such that $x_1 < x_0 < 0 < x_2<a_1$.
   By the direct calculation, for any $x_1 < x_0 < 0<a_1 < x_2$ we have
   \begin{eqnarray*}
   \tilde F(x_2)-\tilde F(x_1)=\frac{1}{3}\delta (x_1-x_2)p_1,
    ~~~~~~~
  \lambda(x_2)-\lambda(x_1)=\frac{\delta p_2}{\hat{f}(x_1)\hat{f}(x_2)},
   \end{eqnarray*}
   where
   \begin{eqnarray*}
   p_1&=& x_1^2+x_1x_2+x_2^2-3(a+1)x_1-3(a+1)x_2+3(a+1)^2+3b,
   \\
   p_2&=&x_1x_2^2-x_1^2x_2-2x_2^2+(a^2+2a+b+1)x_1-(a^2-2a+b-3)x_2-2(a+1)^2-2b.
   \end{eqnarray*}
   By the method given in \cite[pp.368-369]{Knuth},
   $p_1=p_2=0$ if and
   only if $p_1=0$ and ${\rm prem}(p_2, p_1, x_2)=0$,
   where
   ${\rm prem}(p_2,p_1, x_2)$ is called the pseudo-remainder of $p_2$ divided
   by $p_1$ and
   \begin{eqnarray*}
   {\rm prem}(p_2,p_1, x_2)&=&s_1x_2+s_2,
   \\
   s_1&=&-2x_1^2+(3a+5)x_1-a^2-4a-b-3,
   \\
   s_2&=&-x_1^3+(3a+5)x_1^2-2(a^2+5a+b+4)x_1+4(a+1)^2+4b.
   \end{eqnarray*}
   A necessary condition for $s_1=s_2=0$ is given by the vanishing of  resultants (see \cite{Gelfand}) of $s_1$ and $s_2$.
   Since
   \[
   {\rm res}(s_1,s_2,x_1)=b(9a^2+9b-1)((a-1)^2+b)\ne0,
   \]
   we can calculate $x_2=-s_2/s_1$ from ${\rm prem}(p_2, p_1, x_2)=0$ and substitute it into $p_1$.
   Then,
   \[
   p_1=\frac{3((x_1-(a+1))^2+b)}{s_1^2}h(x_1),
   \]
   where
  \begin{eqnarray*}
   h(x)&=&x^4-(3a+5)x^3+\frac{4}{3}(3a^2+9a+3b+7)x^2-\frac{1}{3}(9a^3+33a^2+9ab+47a+33b+23)x
   \\
   & &+a^4+4a^3+2a^2b+\frac{22}{3}a^2+4ab+b^2+\frac{20}{3}a+\frac{22}{3}b+\frac{7}{3}.
  \end{eqnarray*}
   Since when $b\leq -3(a+1)^2$ and $a>1$,
   \begin{eqnarray*}
   \lim_{x\rightarrow-\infty} h(x)&=&+\infty,\\
   h(a+1-\sqrt{-b})&=&2b(a-\sqrt{-b})(a-1-\sqrt{-b})<0,
   \\
   h(a+1)&=&\frac{1}{3}b((3a-1)^2+3b)>0,
   \\
   h(a+1+\sqrt{-b})&=&2b(a+\sqrt{-b})(a-1+\sqrt{-b})<0,\\
    \lim_{x\rightarrow+\infty} h(x)&=&+\infty,\\
   \end{eqnarray*}
   there is exactly one $x_1$ such that $-\infty<x_1<a+1-\sqrt{-b}$.
   Then,  condition {\bf (iv)}  in Lemma \ref {uni-lc-F} holds.
   By Lemma \ref {uni-lc-F}, system \eqref{F-g-x-1}
   exhibits  at most one crossing limit cycle in the strip $x\in(-\infty, a+1+\sqrt{-b})$ when $b\le -3(a+1)^2$, and it is unstable and hyperbolic if it exists.

    If system  \eqref{F-g-x-1} exhibits a closed orbit which passes through $(a+1+\sqrt{-b}, \tilde F(a+1+\sqrt{-b}))$,
    then the integral of the divergence of \eqref{F-g-x-1} along this closed orbit is also greater than $0$, which can be proven by the similar method of Lemma \ref {uni-lc-F}.
   Therefore, system \eqref{SD} exhibits at most one crossing limit cycle in the strip $x\ge-\sqrt{-b}$. 
    Moreover, the limit cycle is stable and hyperbolic if it exists.
\end{proof}

\begin{lm}
    If there are two crossing limit cycles $L_1$ and $L_2$ of system \eqref{SD} intersecting with $x=-\sqrt{-b}$
    when $b< -(a+1)^2$, then the following assertion is true:
    \begin{eqnarray*}
	\oint_{L_1}{\rm div} (y-F(x), -g(x))\mathrm{d}t>\oint_{L_2}{\rm div} (y-F(x), -g(x))\mathrm{d}t,
	\end{eqnarray*}
    where $L_1$ lies in the region enclosed by $L_2$.
	\label{lm-div-a>1}
\end{lm}

\begin{proof}
    Denote the intersection points of $L_i$ and the $y$-axis by $A_i$, $B_i$ from top to bottom.
    Let the intersection points of $L_i$ and $y=F(x)$ be $C_i$, $D_i$ from left to right, $i=1,2$.
    By Lemma \ref{lm-c-p}, $L_i$ crosses $x=\sqrt{-b}$ at two points, denoted by $M_i$ and $N_i$ from top to bottom, $i=1,2$.
    Moreover,
    $L_i$ crosses $x=-\sqrt{-b}$ at two points, denoted by $P_i$ and $Q_i$, $i=1,2$,
    as shown in Fig. \ref{fig-div}.

	\begin{figure}[!htb]
	\centering
	\includegraphics[scale=0.27]{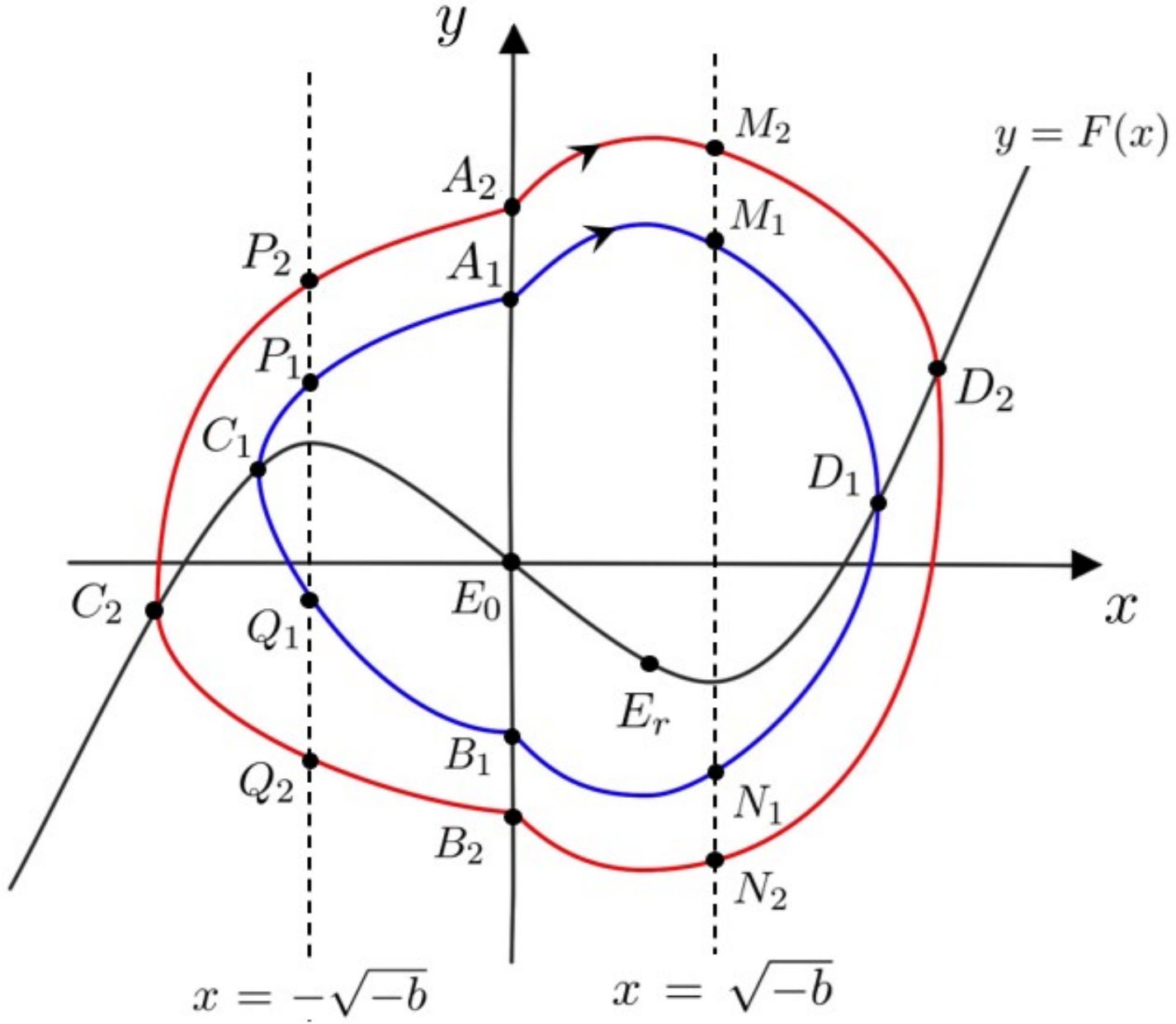}
	\caption{\footnotesize{ Two crossing limit cycles $L_1$ and $L_2$. }}
	\label{fig-div}
    \end{figure}

    Notice that $y-F(x)>0$ along the orbit segments $\widehat{A_1M_1}$ and $\widehat{A_2M_2}$.
    The arc $\widehat{A_iM_i}$ can be regarded as the graph of the function $y=y_i(x)$, where $0\le x\le \sqrt{-b}$ and $i=1,2$.
    Thus,
    \begin{eqnarray}
	\int_{\widehat{A_2M_2}}f(x)\mathrm{d}t-\oint_{\widehat{A_1M_1}} f(x)\mathrm{d}t
    &=&\int_{0}^{\sqrt{-b}}\frac{f(x)}{y_2(x)-F(x)}\mathrm{d}x-\int_{0}^{\sqrt{-b}}\frac{f(x)}{y_1(x)-F(x)}\mathrm{d}x
    \nonumber\\
    &=&\int_{0}^{\sqrt{-b}}\frac{f(x)(y_1(x)-y_2(x))}{(y_1(x)-F(x))(y_2(x)-F(x))}\mathrm{d}x>0.
    \label{fAM}
    \end{eqnarray}
    Similar to \eqref{fAM}, we obtain that
    \begin{eqnarray}
	\int_{\widehat{N_2B_2}}f(x)\mathrm{d}t-\int_{\widehat{N_1B_1}} f(x)\mathrm{d}t>0.
    \label{fNB}
    \end{eqnarray}
    Note that $f(x)=0$ along $\overline{M_1M_2}$ and $\overline{N_2N_1}$.
    Let $\gamma$ be the simple closed curve $\widehat{M_2D_2N_2}\cup\overline{N_2N_1}\cup\widehat{N_1D_1M_1}\cup\overline{M_1 M_2}$
    and $\Omega$ be the region enclosed by $\gamma$.
    Since
    \[
    \frac{\mathrm{d}\left(f(x)/g(x)\right)}{\mathrm{d}x}=\frac{\delta((x-a-1)^2-b-(a+1)^2)}{(x-a-1)^2} >0,
    \]
    by Green's formula,
	\begin{eqnarray}
	\int_{\widehat{M_2D_2N_2}}f(x)\mathrm{d}t-\int_{\widehat{M_1D_1N_1}}f(x)\mathrm{d}t
    =\oint_{\gamma}-\frac{f(x)}{g(x)}\mathrm{d}y
    =\iint_{\Omega}\frac{\mathrm{d}\left(f(x)/g(x)\right)}{\mathrm{d}x}\mathrm{d}x\mathrm{d}y>0.
    \label{fMN}
	\end{eqnarray}
    It follows from (\ref{fAM})--(\ref{fMN}) that
    \begin{eqnarray}
	\int_{\widehat{A_2D_2B_2}}f(x)\mathrm{d}t-\int_{\widehat{A_1D_1B_1}}f(x)\mathrm{d}t>0.\label{divADB}
	\end{eqnarray}

    On the one hand, the arc $\widehat{P_i A_i}$ can be considered as the graph of the function $y=\hat{y}_i(x)$, where $-\sqrt{-b}\le x\le0$ and $i=1,2$.
    Thus,
    \begin{eqnarray}
	\int_{\widehat{P_2 A_2}}f(x)\mathrm{d}t-\int_{\widehat{P_1 A_1}} f(x)\mathrm{d}t
    &=&\int^{0}_{-\sqrt{-b}}\frac{f(x)}{\hat{y}_2(x)-F(x)}\mathrm{d}x-\int^{0}_{-\sqrt{-b}}\frac{f(x)}{\hat{y}_1(x)-F(x)}\mathrm{d}x
    \nonumber\\
    &=&\int^{0}_{-\sqrt{-b}}\frac{f(x)(\hat{y}_1(x)-\hat{y}_2(x))}{(\hat{y}_1(x)-F(x))(\hat{y}_2(x)-F(x))}\mathrm{d}x>0.
    \label{fPA}
    \end{eqnarray}
    Similarly,
    \begin{eqnarray}
	\int_{\widehat{B_2 Q_2}}f(x)\mathrm{d}t-\int_{\widehat{B_1 Q_1}} f(x)\mathrm{d}t>0.
    \label{fBQ}
    \end{eqnarray}
    On the other hand, by Green's formula,
	\begin{eqnarray}
	\int_{\widehat{Q_2C_2P_2}}f(x)\mathrm{d}t-\int_{\widehat{Q_1C_1P_1}}f(x)\mathrm{d}t
    =\oint_{\hat{\gamma}}-\frac{f(x)}{g(x)}\mathrm{d}y
    =\iint_{\hat{\Omega}}\frac{\mathrm{d}\left(f(x)/g(x)\right)}{\mathrm{d}x}\mathrm{d}x\mathrm{d}y>0,
    \label{fQP}
	\end{eqnarray}
    because  $f(x)=0$ along $\overline{P_2P_1}$ and $\overline{Q_1Q_2}$ and
    \[
    \frac{\mathrm{d}\left(f(x)/g(x)\right)}{\mathrm{d}x}=\frac{\delta((x-a+1)^2-b-(a-1)^2)}{(x-a+1)^2} >0,
    \]
    where  $\hat{\gamma}$ is the simple closed curve $\widehat{Q_2C_2P_2}\cup\overline{P_2P_1}\cup\widehat{P_1C_1Q_1}\cup\overline{Q_1 Q_2}$
    and $\hat{\Omega}$ is the region enclosed by $\hat{\gamma}$.
    It follows from 
    (\ref{fPA})--(\ref{fQP}) that
    \begin{eqnarray}
    \int_{\widehat{B_2C_2A_2}} f(x)\mathrm{d}t-\int_{\widehat{B_1C_1A_1}} f(x)\mathrm{d}t>0.
    \label{divBCA}
	\end{eqnarray}
   By \eqref{divADB} and \eqref{divBCA},
   {\small
   \begin{eqnarray*}
   & &\oint_{L_2}{\rm div} (y-F(x), -g(x))\mathrm{d}t-\oint_{L_1}{\rm div} (y-F(x), -g(x))\mathrm{d}t
   \\
   & &=-\oint_{\widehat{A_2D_2B_2}\cup\widehat{B_2C_2A_2}}f(x)\mathrm{d}t+\oint_{\widehat{A_1D_1B_1}\cup\widehat{B_1C_1A_1}}f(x)\mathrm{d}t
   \\
   & &<0.
   \end{eqnarray*}}
\end{proof}

By the transformation $(x,y)\rightarrow(x,y+F(x))$,
system \eqref{SD} is rewritten as
\begin{eqnarray}
\dot{x}=y,~~~~\dot{y}
=-(x-{\rm sgn}(x)-a)-\delta x^2y-\tilde{b}y,
\label{ygf}
\end{eqnarray}
where $\tilde{b}=\delta b$.
Clearly, system \eqref{ygf} has the same topological structure as system \eqref{SD}.

Consider the orbit starting at the point $C:(x_c,0)$
and that crosses the positive $x$-axis
regardless of whether time $t\rightarrow \infty$ or $t\rightarrow -\infty$, where $x_c\le0$.
Let $\gamma_{\delta,\tilde{b}}^+(c)$ and $\gamma_{\delta,\tilde{b}}^-(c)$ be the positive and negative orbits of system
\eqref{ygf} that start at $C$, respectively.
Denote the first intersection point of $\gamma_{\delta,\tilde{b}}^+(c)$ (resp. $\gamma_{\delta,\tilde{b}}^-(c)$)
and the positive $x$-axis by $A^+: (x_{A^+}^c(\delta,\tilde{b}), 0)$ (resp. $A^-: (x_{A^-}^c(\delta,\tilde{b}), 0)$).
The following lemma tells us that how $x_{A^+}^c$ and $x_{A^-}^c$ continuously depend on $\delta$ and $\tilde{b}$.

\begin{lm}
	For a fixed $\tilde{b}$,  $x_{A^+}^c$ decreases and $x_{A^-}^c$ increases continuously as $\delta$ increases.
    For a fixed $\delta$, $x_{A^+}^c$ decreases and $x_{A^-}^c$ increases continuously as $\tilde{b}$ increases.
    \label{lm-mono}
\end{lm}
\begin{proof}

    Let $y=y_{\delta,\tilde{b}}^\pm(x)$ be the function of $\gamma_{\delta,\tilde{b}}^\pm(c)$, where  $x_c<x<x_{A^\pm}^c$.
    Then, $y_{\delta,\tilde{b}}^\pm(x_c)=0$ and $y_{\delta,\tilde{b}}^\pm(x_{A^\pm})=0$.
	Let $z_{\delta,\tilde{b}}^\pm(x)=y_{\delta+\epsilon,\tilde{b}}^\pm(x)-y_{\delta,\tilde{b}}^\pm(x)$,
    $0\le x\le \min\{x_{A^\pm}^c(\delta,\tilde{b}), x_{A^\pm}^c(\delta+\epsilon,\tilde{b})\}$,
    be the vertical distance between
	$\gamma_{\delta+\epsilon,\tilde{b}}^\pm(c)$ and $\gamma_{\delta,\tilde{b}}^\pm(c)$, where $(x_{A^\pm}^c(\delta+\epsilon,\tilde{b}),0)$
    is the first intersection point of
	$\gamma_{\delta+\epsilon,\tilde{b}}^\pm(c)$ and the positive
	$x$-axis and $|\epsilon|$ is sufficiently small.
	One can check that
	\begin{eqnarray}
	z_{\delta,\tilde{b}}^\pm(x)&=&y_{\delta+\epsilon,\tilde{b}}^\pm(x)-y_{\delta,\tilde{b}}^\pm(x)\nonumber\\
	&=&
	\int_{x_c}^{x}
	\left\{
	\left(\frac{-(s-{\rm sgn}(s)-a)}{y_{\delta+\epsilon,\tilde{b}}^\pm(s)}-(\delta+\epsilon)s^2-\tilde{b}\right)
	-\left(\frac{-(s-{\rm sgn}(s)-a)}{y_{\delta,\tilde{b}}^\pm(s)}-\delta s^2-\tilde{b}\right)
	\right\}
	\mathrm{d}s\nonumber
	\\
	&=&H_1(x)+H_2(x) ,\label{mo-zdb}
	\end{eqnarray}
	where
	\[
	H_1(x)=-\epsilon \frac{x^3}{3}+\epsilon \frac{x_c^3}{3},\hskip 0.3cm H_2(x)=\int_{x_c}^{x}z_{\delta,\tilde{b}}^\pm(s)H_3(s)\mathrm{d}s  ~~{\rm and}~~~
	H_3(x)=\frac{x-{\rm sgn}(x)-a}{y_{\delta+\epsilon,\tilde{b}}^\pm(x)y_{\delta,\tilde{b}}^\pm(x)}.
	\]
	From \eqref{mo-zdb}, we get
	\begin{eqnarray}
	z_{\delta,\tilde{b}}^\pm(x)H_3(x)=H_1(x)H_3(x)+H_2(x)H_3(x).
	\label{zH123}
	\end{eqnarray}
	It follows from \eqref{zH123} that
	\[
	\frac{\mathrm{d}H_2(x)}{\mathrm{d}x}-H_2(x)H_3(x)=H_1(x)H_3(x),
	\]
	which is a first order linear differential equation.
	With the initial condition $H_2(x_c)=0$, we get
	\begin{eqnarray}
	H_2(x)=\int_{x_c}^{x}H_1(\tau)H_3(\tau)\exp\left\{\int_{\tau}^{x}H_3(\eta)\mathrm{d}\eta
	\right\}\mathrm{d}\tau.\label{z-H2}
	\end{eqnarray}
	From \eqref{mo-zdb} and \eqref{z-H2},
		\begin{eqnarray}
		z_{\delta,\tilde{b}}^\pm(x)
		&=&H_1(x)+\int_{x_c}^{x}H_1(\tau)H_3(\tau)\exp\left\{\int_{\tau}^{x}H_3(\eta)\mathrm{d}\eta
		\right\}\mathrm{d}\tau\nonumber\\
		&=&H_1(x)-\left[H_1(\tau)\exp\left\{\int_{\tau}^{x}H_3(\eta)\mathrm{d}\eta
		\right\}\right]_{x_c}^{x}
		+\int_{x_c}^{x}H'_1(\tau)\exp\left\{\int_{\tau}^{x}H_3(\eta)\mathrm{d}\eta
		\right\}\mathrm{d}\tau~~~~
		\nonumber\\
		&=&H_1(x_c)\exp\left\{\int_{x_c}^{x}H_3(\eta)\mathrm{d}\eta
		\right\}+\int_{x_c}^{x}H'_1(\tau)\exp\left\{\int_{\tau}^{x}H_3(\eta)\mathrm{d}\eta
		\right\}\mathrm{d}\tau\nonumber\\
		&=&-\epsilon\int_{x_c}^{x}\tau^2\exp\left\{\int_{\tau}^{x}H_3(\eta)\mathrm{d}\eta
		\right\}\mathrm{d}\tau
		<0~({\rm resp.} >0), ~{\rm if}~\epsilon>0 ~({\rm resp.} <0).
        \label{zdb}
		\end{eqnarray}
	Then, $0=y_{\delta+\epsilon,\tilde{b}}^+(x_{A^+}^c(\delta+\epsilon,\tilde{b}))<y_{\delta,\tilde{b}}^+(x_{A^+}^c(\delta+\epsilon,\tilde{b}))$
    and $y_{\delta+\epsilon,\tilde{b}}^-(x_{A^-}^c(\delta,\tilde{b}))<y_{\delta,\tilde{b}}^-(x_{A^-}^c(\delta,\tilde{b}))=0$
    when $\epsilon>0$,
    implying that $x_{A^+}^c(\delta+\epsilon,\tilde{b})<x_{A^+}^c(\delta,\tilde{b})$
    and $x_{A^-}^c(\delta+\epsilon,\tilde{b})>x_{A^-}^c(\delta,\tilde{b})$.
	Moreover, it follows from \eqref{zdb} that $\lim_{\epsilon\rightarrow 0}z_{\delta,\tilde{b}}^\pm(x)=0$.
	That means  $\lim_{\epsilon\rightarrow 0}x_{A^\pm}^c(\delta+\epsilon,\tilde{b})=x_{A^\pm}^c(\delta,\tilde{b})$.
	Thus, $x_{A^+}^c$ decreases and $x_{A^-}^c$ increases continuously as $\delta$ increases.
	
	Let $d_{\delta,\tilde{b}}^\pm(x)=y_{\delta,\tilde{b}+\epsilon}^\pm(x)-y_{\delta,\tilde{b}}^\pm(x)$,
    $0\le x\le \min\{x_{A^\pm}^c(\delta,\tilde{b}), x_{A^\pm}^c(\delta,\tilde{b}+\epsilon)\}$,
    where $(x_{A^\pm}^c(\delta,\tilde{b}+\epsilon),0)$
    is the first intersection point of
	$\gamma_{\delta,\tilde{b}+\epsilon}^\pm(c)$ and the positive
	$x$-axis and $|\epsilon|$ is sufficiently small.
	Then,
	\begin{eqnarray*}
	d_{\delta,\tilde{b}}^\pm(x)&=&y_{\delta,\tilde{b}+\epsilon}^\pm(x)-y_{\delta,\tilde{b}}^\pm(x)\nonumber\\
	&=&
	\int_{x_c}^{x}
	\left\{
	\left(\frac{-(s-{\rm sgn}(s)-a)}{y_{\delta,\tilde{b}+\epsilon}^\pm(s)}-\delta s^2-(\tilde{b}+\epsilon)\right)
	-\left(\frac{-(s-{\rm sgn}(s)-a)}{y_{\delta,\tilde{b}}^\pm(s)}-\delta s^2-\tilde{b}\right)
	\right\}
	\mathrm{d}s\nonumber
	\\
	&=&\hat{H}_1(x)+{H}_2(x) ,
	\end{eqnarray*}
	where $\hat H_1(x)=-\epsilon x+\epsilon x_c$ and $H_2(x)$ is defined as in \eqref{mo-zdb}.
    Similar to the calculation of $z_{\delta,\tilde{b}}^\pm(x)$,
	we get
		\begin{eqnarray}
		d_{\delta,\tilde{b}}^\pm(x)
		=-\epsilon\int_{x_c}^{x}\exp\left\{\int_{\tau}^{x}H_3(\eta)\mathrm{d}\eta
		\right\}\mathrm{d}\tau
		<0~({\rm resp.} >0), ~{\rm if}~\epsilon>0 ~({\rm resp.} <0).
        \label{ddb}
		\end{eqnarray}
	Then, $0=y_{\delta,\tilde{b}+\epsilon}^+(x_{A^+}^c(\delta,\tilde{b}+\epsilon))<y_{\delta,\tilde{b}}^+(x_{A^+}^c(\delta,\tilde{b}+\epsilon))$
    and $y_{\delta,\tilde{b}+\epsilon}^-(x_{A^-}^c(\delta,\tilde{b}))<y_{\delta,\tilde{b}}^-(x_{A^-}^c(\delta,\tilde{b}))=0$
    when $\epsilon>0$,
    implying that $x_{A^+}^c(\delta,\tilde{b}+\epsilon)<x_{A^+}^c(\delta,\tilde{b})$
    and $x_{A^-}^c(\delta,\tilde{b}+\epsilon)>x_{A^-}^c(\delta,\tilde{b})$.
	Moreover, it follows from \eqref{ddb} that $\lim_{\epsilon\rightarrow 0}d_{\delta,\tilde{b}}^\pm(x)=0$.
	That means  $\lim_{\epsilon\rightarrow 0}x_{A^\pm}^c(\delta,\tilde{b}+\epsilon)=x_{A^\pm}^c(\delta,\tilde{b})$.
	Thus, $x_{A^+}^c$ decreases and $x_{A^-}^c$ increases continuously as $\tilde{b}$ increases.
\end{proof}

In the following, we consider the number of crossing limit cycles of system \eqref{SD} for $\delta>0$ sufficiently small and $b<-(a+1)^2$. Thus, we can view system \eqref{SD} as  the perturbed system  of the following system
\begin{eqnarray}
\dot{x}=y, ~~~~~~~\dot{y}=-x+{\rm sgn}(x)+a.
\label{SD0}
\end{eqnarray}	
Obviously, system \eqref{SD0} has a family of piecewise periodic orbits given by
 \begin{eqnarray*}
\Gamma_h&=&\Gamma_h^-\cup\Gamma_h^+\\
&=&\{(x,y)\mid H^-(x,y)=h,\ x<0\}\cup
\{(x,y)\mid H^+(x,y)=h,\ x>0\},\quad 0<h<+\infty,
\end{eqnarray*}	
where
 \begin{eqnarray*}
H^-(x,y)=\frac{y^2}{2}+\frac{x^2}{2}-(a-1)x, \quad \quad  H^+(x,y)=\frac{y^2}{2}+\frac{x^2}{2}-(a+1)x.
\end{eqnarray*}	

Since $H^+(0,y)\equiv H^-(0,y)$ for $y\in(-\infty,+\infty)$, we obtain the expression of the first order Melnikov function of system \eqref{SD} as
follows by \cite[Theorem 1.1]{LH}
\begin{eqnarray}
M(h)=\int_{\Gamma_h^-}\left(\frac{x^3}{3}+bx\right)\mathrm{d}y+\int_{\Gamma_h^+}\left(\frac{x^3}{3}+bx\right)\mathrm{d}y,\quad 0<h<+\infty.
\label{M10}
\end{eqnarray}
Note that the orientation of the orbit $\Gamma_h$ is clockwise. We compute the first order Melnikov function \eqref{M10} directly as follows
\begin{equation}\begin{split}\label{M}
M(h)=&\displaystyle\int_{-\sqrt{2h}}^{\sqrt{2h}}\left(\frac{x^3}{3}+bx\right)\Big|_{x=a-1-\sqrt{2 h+(a-1)^2-y^2}}\mathrm{d}y\\
&+\displaystyle\int_{\sqrt{2h}}^{\sqrt{2 h+(a+1)^2}}\left(\frac{x^3}{3}+bx\right)\Big|_{x=a+1-\sqrt{2 h+(a+1)^2-y^2}}\mathrm{d}y\\
&+\displaystyle\int_{\sqrt{2 h+(a+1)^2}}^{-\sqrt{2 h+(a+1)^2}}\left(\frac{x^3}{3}+bx\right)\Big|_{x=a+1+\sqrt{2 h+(a+1)^2-y^2}}\mathrm{d}y\\
&+\displaystyle\int_{-\sqrt{2 h+(a+1)^2}}^{-\sqrt{2h}}\left(\frac{x^3}{3}+bx\right)\Big|_{x=a+1-\sqrt{2 h+(a+1)^2-y^2}}\mathrm{d}y\\
=&-\frac{1}{4} \pi  (a+1)^2 \left(5 a^2+10 a+4 b+5\right)-\frac{\left(15 a^2+4 b+5\right)\sqrt{h} }{\sqrt{2}}\\
&-\left(3 a^2+6 a+2b+3\right)\pi h -\frac{13}{3}\sqrt{2} h^{3/2}-\pi  h^2\\
&-\frac{1}{4} \left(a^2-2 a+2 h+1\right) \left(5 a^2-10 a+4 b+2 h+5\right) \arctan
   \left(\frac{\sqrt{2h}}{a-1}\right)\\
&+\frac{1}{4} \left(a^2+2 a+2 h+1\right)
   \left(5 a^2+10 a+4 b+2 h+5\right) \arctan\left(\frac{\sqrt{2h}}{a+1}\right).
\end{split}\end{equation}
By \cite[Theorem 1.1]{LH}, the number of simple zeros of $M(h)$ implies the number of crossing limit cycles of system \eqref{SD} for $\delta>0$ sufficiently small.
If $M(h)$ has a zero $h_0$ with multiplicity $k$, then system \eqref{SD} has at most $k$ limit cycles bifurcating from $\Gamma_{h_0}$ for $\delta>0$ sufficiently small. Moreover, if $k$ is odd, then system \eqref{SD} has at least one limit cycle.
Thus, we begin to study the number of zeros of $M(h)$ in \eqref{M}.
A simple calculation of the second order derivative of $M(h)$ shows that
\begin{equation*}
M''(h)=M_{20}(h)+M_{21}(h),
\end{equation*}
where
\begin{equation*}\begin{split}
M_{20}(h)=&\frac{2 \sqrt{2} b \left(2 h-a^2+1\right)}{\sqrt{h} \left(2h+a^2-2 a+1\right) \left(2h+a^2+2
   a+1\right)},\\
M_{21}(h)=&-\frac{2 \left(\left(a^2-\sqrt{2} \sqrt{h}-1\right)^2+\left(4 \pi  \left(a^2+1\right)-2\right) h+(\pi -1) \left(a^2-1\right)^2+4 \sqrt{2} h^{3/2}+4 \pi
   h^2\right)}{\left(a^2-2 a+2 h+1\right) \left(a^2+2 a+2 h+1\right)}\\
   &-2 \arctan\left(\frac{\sqrt{2h}}{a-1}\right)+2 \arctan\left(\frac{\sqrt{2h}}{a+1}\right).
\end{split}\end{equation*}
Obviously, $M_{21}(h)<0$ for $0<h<+\infty$ according to $a>1$ and the monotonicity of $\arctan h$. Note that
$M_{20}(h)\leq0$ when $h\geq\frac{a^2-1}{2}$ by $b<0$, thus we derive that
\begin{equation}\label{M22}
M''(h)<0,\quad \mbox{for}\ \frac{a^2-1}{2}\leq h<+\infty.
\end{equation}
When $0<h<\frac{a^2-1}{2}$, we consider the third order derivative of $M(h)$
\begin{equation}\label{M3}
M'''(h)=M_{30}(h)+M_{31}(h),
\end{equation}
where
\begin{equation*}\begin{split}
M_{30}(h)=&\frac{8 \sqrt{2} \sqrt{h} \left(\left(-a^2+2 h+1\right)^2-4 a^2 \left(a^2-1\right)\right)}{\left(a^2-2 a+2 h+1\right)^2 \left(a^2+2 a+2 h+1\right)^2},\\
M_{31}(h)=&-\frac{\sqrt{2} b
   \left(-\left(a^2-1\right)^3+\left(-14 a^4+4 a^2+10\right) h+\left(28-12 a^2\right) h^2+24 h^3\right)}{h^{3/2} \left(a^2-2 a+2 h+1\right)^2 \left(a^2+2 a+2 h+1\right)^2}.
\end{split}\end{equation*}
It is easy to verify that
\begin{equation}\label{M30}
M_{30}(h)<0, \quad \mbox{for}\ 0<h<\frac{a^2-1}{2},
\end{equation}
owing to $\left(-a^2+2 h+1\right)^2-4 a^2 \left(a^2-1\right)<-(a^2-1)(1 + 3 a^2)<0$. Let
$$f(h)=-\left(a^2-1\right)^3+\left(-14 a^4+4 a^2+10\right) h+\left(28-12 a^2\right) h^2+24 h^3.$$
We can get that
$$f'(h)=10+4a^2-14 a^4+2(28-12 a^2)h + 72h^2.$$
Since the coefficient of $h^2$ is positive, and
\begin{equation*}\begin{split}
f'(0)=&-2 (-1 + a) (1 + a) (5 + 7 a^2)<0,\\
f'\left(\frac{a^2-1}{2}\right)=&-8 (-1 + a)(1 + a)a^2<0,
\end{split}\end{equation*}
we obtain that $f'(h)<0$ for  $0<h<\frac{a^2-1}{2}$, which means that $f(h)$ is monotonically decreasing on the interval $(0,\frac{a^2-1}{2})$.
It follows from $f(0)=-(a^2-1)^3<0$ that $f(h)<0$ on the interval $(0,\frac{a^2-1}{2})$. Further, using $b<0$, we have
$$M_{31}(h)=-\frac{\sqrt{2} bf(h)}{h^{3/2} \left(a^2-2 a+2 h+1\right)^2 \left(a^2+2 a+2 h+1\right)^2}<0,$$
and  $M'''(h)<0$ on the interval $(0,\frac{a^2-1}{2})$ combing with \eqref{M3} and \eqref{M30}. So, $M''(h)$ is monotonically decreasing on the interval $(0,\frac{a^2-1}{2})$.
Note that
$$\lim_{h\rightarrow 0^+}M''(h)=+\infty\quad \mbox{and} \quad M''\left(\frac{a^2-1}{2}\right)=M_{21}\left(\frac{a^2-1}{2}\right)<0,$$
we find that $M''(h)$ has a unique simple zero $h_2\triangleq h_2(a,b)$ on the interval $(0,+\infty)$ combing with \eqref{M22},  more specifically, $h_2\in (0,\frac{a^2-1}{2})$.
Hence, as the first order derivative of $M(h)$, $M'(h)$ is monotonically  increasing on the interval $(0,h_2)$, and monotonically  decreasing on the interval
$(h_2,+\infty)$.

In the following, we will split two cases to study the number of zeros of $M(h)$ given in \eqref{M}.

A direct computation gives that
\begin{equation}\begin{split}\label{M1M0}
M'(0)=&-(3+6a+3a^2+2b)\pi,\\
M(0)=&-\frac{1}{4}(a+1)^2 (5+10a+5a^2+4b)\pi.
\end{split}\end{equation}
Then, we have the first result.
\begin{lm}\label{lm3-9}
If $b\leq-\frac{3}{2}(a+1)^2$, then $M(h)$ has a unique simple zero on the interval $(0,+\infty)$.
\end{lm}
\begin{proof}
Recall that $M'(h)$  increases monotonically on $(0,h_2)$, attains its maximum at $h_2$, and then decreases monotonically on $(h_2,+\infty)$.
When $b\leq-\frac{3}{2}(a+1)^2$, $M'(0)\geq0$ by \eqref{M1M0}, thus
$$M'(h_2)>M'(h)>M'(0)\geq0,\quad \mbox{for} \ 0<h<h_2.$$
And since
$$\lim_{h\rightarrow+\infty}M'(h)=-\infty,$$
we find that $M'(h)$ has a unique simple zero $h_1$, located in the interval $(h_2,+\infty)$. It follows that $M(h)$  is monotonically increasing on the interval $(0,h_1)$, and monotonically decreasing on the interval $(h_1,+\infty)$. In this case, by \eqref{M1M0}, $M(0)>0$, therefore $M(h)$ increases until it reaches its maximum $M(h_1)>M(0)>0$,  then decreases, and finally tends to
$-\infty$ as $h\rightarrow+\infty$. Thus, $M(h)$ has a unique simple zero on the interval $(0,+\infty)$.
\end{proof}

If $-\frac{3}{2}(a+1)^2< b<-(a+1)^2$, then $M'(0)<0$.
Since $M'(h)$ achieves its maximum at $h_2$ and  decreases monotonically to $-\infty$ for $h \in (h_2, +\infty)$, the sign of $M'(h_2) \triangleq M'(h_2(a,b), a, b)$ determines the number of zeros of $M'(h)$, which in turn affects the number of zeros of $M(h)$.
To establish the sign of $M'(h_2)$, we analyze the cases when $b = -\frac{3}{2}(a+1)^2$ and $b = -(a+1)^2$. The bifurcation result is presented as follows.

\begin{lm}\label{lm3-10}
There exists a unique function $b=b_1(a)\in\left(-\frac{3}{2}(a+1)^2,-(a+1)^2\right)$ such that $M'(h_2(a,b),a,b)|_{b=b_1(a)}=0$.
Moreover, the function $b=b_1(a)$ is decreasing on $a$.
\end{lm}
\begin{proof}

Obviously, according to the discussion
in Lemma \ref{lm3-9}, if $b=-\frac{3}{2}(a+1)^2$, then $M'(h_2)>M'(0)=0$. If $b=-(a+1)^2$, then
\begin{equation*}\begin{split}\label{M1h}
M'(h)=&-(a+1)^2\pi-6\sqrt{2h}-2h\pi+(-1+10a-a^2-2h)\arctan\left(\frac{\sqrt{2h}}{a-1}\right)\\
&+(1+2a+a^2+2h)\arctan\left(\frac{\sqrt{2h}}{a+1}\right).
\end{split}\end{equation*}
We will prove that $M'(h)<0$ for all $h>0$ when $b=-(a+1)^2$, consequently we have  $M'(h_2)<0$ when $b=-(a+1)^2$.
When $a\geq 5+2\sqrt{6}$, one has  $-1+10a-a^2\leq0$, then
\begin{equation}\begin{split}\label{M1hq}
M'(h)<&-(a+1)^2\pi-6\sqrt{2h}-2h\pi+(-1+10a-a^2-2h)\arctan\left(\frac{\sqrt{2h}}{a-1}\right)\\
&+(1+2a+a^2+2h)\frac{\pi}{2}\\
= &-\frac{1}{2}(1 + a)^2\pi-6\sqrt{2h}-h\pi+(-1+10a-a^2-2h)\arctan\left(\frac{\sqrt{2h}}{a-1}\right)\\
<&0.
\end{split}\end{equation}
When $1<a<5+2\sqrt{6}$ and $h\geq\frac{-1 + 10 a - a^2}{2}$, one has $-1+10a-a^2-2h\leq0$, then the inequality \eqref{M1hq} also holds.
When $1<a<5+2\sqrt{6}$ and $0<h<\frac{-1 + 10 a - a^2}{2}$, we consider an auxiliary function
\begin{equation*}\begin{split}\label{A}
A(h)=&-(a+1)^2\pi-6\sqrt{2h}-2h\pi+(-1+10a-a^2-2h)\min\left(\frac{\sqrt{2h}}{a-1},\frac{\pi}{2}\right)\\
&+(1+2a+a^2+2h)\min\left(\frac{\sqrt{2h}}{a+1},\frac{\pi}{2}\right),
\end{split}\end{equation*}
which satisfies $M'(h)<A(h)$.
More specific, when $1<a\leq\bar a=\frac{20+\pi ^2+4 \sqrt{2 \left(12+\pi ^2\right)}}{4+\pi ^2}\approx4.06096$,
\begin{equation}\label{A1}
A(h)=\left\{\begin{array}{ll}
\begin{array}{ll}-(a+1)^2\pi-6\sqrt{2h}-2h\pi+(-1+10a-a^2-2h)\frac{\sqrt{2h}}{a-1}\\
+(1+2a+a^2+2h)\frac{\sqrt{2h}}{a+1},\end{array} &\mbox{ $0<h\leq\frac{(a-1)^2\pi^2}{8}$,}
\\[2ex]
\begin{array}{ll}-(a+1)^2\pi-6\sqrt{2h}-2h\pi+\frac{\pi}{2}(-1+10a-a^2-2h)\\
+(1+2a+a^2+2h)\frac{\sqrt{2h}}{a+1},\end{array} &\mbox{ $\frac{(a-1)^2\pi^2}{8}<h<\frac{-1 + 10 a - a^2}{2}$,}
\end{array} \right.
\end{equation}
and when $\bar a<a<5+2\sqrt{6}$,
\begin{equation}\begin{split}\label{A2}
A(h)=&-(a+1)^2\pi-6\sqrt{2h}-2h\pi+(-1+10a-a^2-2h)\frac{\sqrt{2h}}{a-1}\\
&+(1+2a+a^2+2h)\frac{\sqrt{2h}}{a+1}, \quad 0<h<\frac{-1 + 10 a - a^2}{2}.
\end{split}\end{equation}
We now prove that $A(h)<0$ for  $A(h)$ defined in \eqref{A1} and \eqref{A2}.
First, consider
\begin{equation}\label{A00}
A(h)=-(a+1)^2\pi-6\sqrt{2h}-2h\pi+(-1+10a-a^2-2h)\frac{\sqrt{2h}}{a-1}+(1+2a+a^2+2h)\frac{\sqrt{2h}}{a+1}.
\end{equation}
It is easy to obtain that
\begin{equation*}
A'(h)=\frac{2 \left(\sqrt{2} (a+1)^2-\pi  (a-1) (a+1) \sqrt{h}-3 \sqrt{2} h\right)}{\left(a^2-1\right) \sqrt{h}},
\end{equation*}
and it has a unique simple zero  at $h=\hat h=\frac{1}{72} (a+1)^2 \left(\pi-\pi  a+\sqrt{\pi ^2 (a-1)^2+24}\right)^2$.
A direct computation shows that
\begin{equation*}\begin{split}
A_1=&A'\left(\frac{(a-1)^2\pi^2}{8}\right)=\frac{-2 \pi ^2 a^3+\left(8-\pi ^2\right) a^2+8 \left(2+\pi ^2\right) a-5 \pi ^2+8}{\pi  (a-1)^2 (a+1)},\\
A_2=&A'\left(\frac{-1 + 10 a - a^2}{2}\right)=\frac{2 \left((a-5) (5 a-1)-\pi  (a-1) (a+1) \sqrt{-a^2+10 a-1}\right)}{\sqrt{-a^2+10 a-1}
   \left(a^2-1\right)},
\end{split}\end{equation*}
and $A_1$ has a unique zero at $a=a_1$  when $a\in(1,\bar a]$ and $A_2$ has a unique zero at $a=a_2$  when $a\in(\bar a,5+2\sqrt{6})$, where $a_1\in(1.8,1.9)$ and $a_2\in(9.83,9.84)$. Thus, we have $\hat h\geq\frac{(a-1)^2\pi^2}{8}$ when $1<a\leq a_1$, and
$\hat h<\frac{(a-1)^2\pi^2}{8}$ when $a_1<a\leq \bar a$, and $\hat h\leq\frac{-1 + 10 a - a^2}{2}$ when $\bar a <a\leq a_2$, and
$\hat h>\frac{-1 + 10 a - a^2}{2}$ when $a_2<a<5+2\sqrt{6}$.
It follows that when $1<a\leq a_1$, $A(h)$ achieves its maximum at $h=\frac{(a-1)^2\pi^2}{8}$ for $h\in(0,\frac{(a-1)^2\pi^2}{8}]$,
when $a_1<a\leq\bar a$, $A(h)$ achieves its maximum at $h=\hat h$  for $h\in(0,\frac{(a-1)^2\pi^2}{8}]$,
when $\bar a<a\leq a_2$,  $A(h)$ achieves its maximum at $h=\hat h$  for $h\in(0,\frac{-1 + 10 a - a^2}{2})$, and
when $a_2<a<5+2\sqrt{6}$, $A(h)$ achieves its maximum at $h=\frac{-1 + 10 a - a^2}{2}$  for $h\in(0,\frac{-1 + 10 a - a^2}{2})$. By a simple calculation,
\begin{equation*}\begin{split}
A\left(\frac{(a-1)^2\pi^2}{8}\right)=&-\frac{\pi  (a-1)\left( \left(4+\pi ^2\right) a^2+\left(8+\pi ^2\right) a-2 \left(\pi ^2-2\right)\right)}{4 (a+1)}<0, \\
A\left(\frac{-1 + 10 a - a^2}{2}\right)
=&-\frac{6 \left(\left(1+4 \pi ^2\right) a^4+4 \left(2 \pi ^2-3\right) a^3+2 \left(11+2 \pi ^2\right) a^2-12 a+1\right)}{(a+1)\left(2 \pi  a^2+2 \pi  a+(a-1)\sqrt{-a^2+10 a-1}\right)}<0,
\end{split}
\end{equation*}
and when $a_1<a\leq a_2$,
\begin{equation*}
A\left(\hat h\right)=\frac{(a+1)^2 }{54 (a-1)}\left(-\pi ^3 (a-1)^3-90 \pi (a-1)+\left(\pi^2(a-1)^2+24\right)^{\frac{3}{2}}\right)<0.
\end{equation*}
Thus, $A(h)$ defined in \eqref{A00} is negative when $1<a\leq\bar a$ and $0<h\leq\frac{(a-1)^2\pi^2}{8}$, and $\bar a<a<5+2\sqrt{6}$ and $0<h<\frac{-1 + 10 a - a^2}{2}$.
Second, consider the function
\begin{equation}\label{A11}
A(h)=-(a+1)^2\pi-6\sqrt{2h}-2h\pi+\frac{\pi}{2}(-1+10a-a^2-2h)+(1+2a+a^2+2h)\frac{\sqrt{2h}}{a+1}.
\end{equation}
We get that
\begin{equation*}
A'(h)=\frac{6h+(a-4) a-5}{\sqrt{2} (a+1) \sqrt{h}}-3 \pi,
\end{equation*}
which has a unique simple zero at
$$h=\bar h=\frac{1}{12}(a+1) \left(\left(3 \pi ^2-2\right) a+\sqrt{3} \pi  \sqrt{(a+1) \left(3 \pi ^2 a-4 a+3 \pi ^2+20\right)}+3 \pi ^2+10\right).$$
It is easy to verify that when $1<a\leq\bar a$,
\begin{equation*}\begin{split}
A_3=&A'\left(\frac{(a-1)^2\pi^2}{8}\right)=\frac{2 (a-5)}{\pi  (a-1)}-\frac{3 \pi  (a+3)}{2 (a+1)}<0,\\
A_4=&A'\left(\frac{-1 + 10 a - a^2}{2}\right)=-\frac{2 (a^2-13a+4)}{(a+1) \sqrt{-a^2+10a-1}}-3 \pi<0.
\end{split}\end{equation*}
Thus, $\frac{(a-1)^2\pi^2}{8}<\frac{-1 + 10 a - a^2}{2}<\bar h$ and $A(h)$ achieves its maximum at $h=\frac{(a-1)^2\pi^2}{8}$. Note that
$$A\left(\frac{(a-1)^2\pi^2}{8}\right)=-\frac{\pi  (a-1) \left(\left(4+\pi ^2\right) a^2+\left(8+\pi ^2\right) a-2 \pi ^2+4\right)}{4 (a+1)}<0,$$
we have $A(h)$ given in \eqref{A11} is also negative when $1<a\leq\bar a$ and $\frac{(a-1)^2\pi^2}{8}<h<\frac{-1 + 10 a - a^2}{2}$. Hence, we have
$M'(h)<A(h)<0$ for $1<a<5+2\sqrt{6}$ and $0<h<\frac{-1 + 10 a - a^2}{2}$. Based on the previous discussions, $M'(h)<0$ when $b=-(a+1)^2$. Further we obtain that $M'(h_2)<0$ when $b=-(a+1)^2$. Combining the fact that $M'(h_2)>0$ when $b=-\frac{3}{2}(a+1)^2$ and
\begin{equation}\begin{split}\label{Mb}
\frac{\mathrm{\partial}M'(h_2)}{\mathrm{\partial}b}\triangleq&\frac{\mathrm{\partial}M'(h_2(a,b),a,b)}{\mathrm{\partial}b}\\
=&\frac{\mathrm{\partial}M'(h,a,b)}{\mathrm{\partial}h}\Big|_{h=h_2}\frac{\mathrm{\partial}h_2}{\mathrm{\partial}b}
+\frac{\mathrm{\partial}M'(h,a,b)}{\mathrm{\partial}b}\Big|_{h=h_2}\\
=&\frac{\mathrm{\partial}M'(h,a,b)}{\mathrm{\partial}b}\Big|_{h=h_2}\\
=&-2 \left(\arctan\left(\frac{\sqrt{2h_2} }{a-1}\right)-\arctan\left(\frac{\sqrt{2h_2}}{a+1}\right)+\pi \right)<0,
\end{split}\end{equation}
there exists a unique function $b=b_1(a)\in\left(-\frac{3}{2}(a+1)^2,-(a+1)^2\right)$ such that $M'(h_2)|_{b=b_1(a)}=0$.
Moreover, it follows from the fact that $b_1(a)$ satisfies $$M'(h_2(a,b),a,b)=0$$ that
\begin{equation*}
\left(\frac{\mathrm{\partial}M'(h,a,b)}{\mathrm{\partial}h}\frac{\mathrm{\partial}h_2}{\mathrm{\partial}b}
+\frac{\mathrm{\partial}M'(h,a,b)}{\mathrm{\partial}b}\right)\Big|_{h=h_2}\frac{\mathrm{d}b_1}{\mathrm{d}a}+
\left(\frac{\mathrm{\partial}M'(h,a,b)}{\mathrm{\partial}h}\frac{\mathrm{\partial}h_2}{\mathrm{\partial}a}
+\frac{\mathrm{\partial}M'(h,a,b)}{\mathrm{\partial}a}\right)\Big|_{h=h_2}=0.
\end{equation*}
It can be simplified to
\begin{equation}\label{ba}
\frac{\mathrm{\partial}M'(h,a,b)}{\mathrm{\partial}b}\Big|_{h=h_2}\frac{\mathrm{d}b_1}{\mathrm{d}a}+\frac{\mathrm{\partial}M'(h,a,b)}{\mathrm{\partial}a}\Big|_{h=h_2}=0.
\end{equation}
Observe that
\begin{equation*}\begin{split}
\frac{\mathrm{\partial}M'(h,a,b)}{\mathrm{\partial}a}=&\frac{8 \sqrt{2h}a(b-2 h)}{\left(a^2-2 a+2 h+1\right) \left(a^2+2 a+2 h+1\right)}+6 (a+1) \arctan\left(\frac{\sqrt{2h} }{a+1}\right)\\
&+6(1-a) \arctan\left(\frac{\sqrt{2h}}{a-1}\right)-6 \pi  (a+1)\\
\leq &\frac{8 \sqrt{2h}a(b-2 h)}{\left(a^2-2 a+2 h+1\right) \left(a^2+2 a+2 h+1\right)}+6(1-a) \arctan\left(\frac{\sqrt{2h}}{a-1}\right)-3 \pi  (a+1)\\
<&0,
\end{split}\end{equation*}
owing to $a>1$ and $b<0$, one has
\begin{equation}\label{b1a}
\frac{\mathrm{d}b_1(a)}{\mathrm{d}a}=-\left(\frac{\mathrm{\partial}M'(h,a,b)}{\mathrm{\partial}a}\Big|_{h=h_2}\right)\Big{/}
\left(\frac{\mathrm{\partial}M'(h,a,b)}{\mathrm{\partial}b}\Big|_{h=h_2}\right)<0
\end{equation}
from \eqref{ba} and \eqref{Mb}. This means that the function $b=b_1(a)$ is  decreasing on $a$, see the red curve of Fig. \ref{fig-bif}.
\end{proof}

To determine the relative position of $b=-\frac{5}{4}(a+1)^2$ and $b=b_1(a)$,  now, we check the properties of  $M(h)$ and $M'(h)$ when $b=-\frac{5}{4}(a+1)^2$. In this case,
\begin{equation}\label{MM01}
\begin{split}
M(0)&=0,\\
M'(0)&=-\frac{1}{2}(a+1)^2\pi,\\
\lim_{h\rightarrow+\infty} M(h)&=-\infty,\\
\lim_{h\rightarrow+\infty} M'(h)&=-\infty.
\end{split}\end{equation}
Recall that  $M''(h)$ has a unique zero $h_2$ on the interval $(0,+\infty)$, it follows from \eqref{MM01} that $M(h)$ has at most two zeros on the interval $(0,+\infty)$.
To be more precise, we analyze the existence and number of intersection points between the curves $b=-\frac{5}{4}(a+1)^2$ and $b=b_1(a)$.
\begin{lm}\label{lm3-11}
For $b=-\frac{5}{4}(a+1)^2$, there exists a unique zero $a=a_*$ such that $M'(h_2(a_*))=0$, which implies that there exists a unique intersection point $A=(a_*,-\frac{5}{4}(a_*+1)^2)$ between the curves $b=-\frac{5}{4}(a+1)^2$ and $b=b_1(a)$. Moreover, $a_*\in(\frac{7}{5},2)$.
\end{lm}
\begin{proof}
A direct computation shows that
\begin{equation*}\begin{split}
M'(h)\triangleq M'(h,a)=&-\frac{1}{2} \pi  (a+1)^2-6 \sqrt{2h}-2 \pi  h+\frac{1}{2} \left(-a^2+22 a-1-4 h\right) \arctan\left(\frac{\sqrt{2h}}{a-1}\right)\\
&+\frac{1}{2} \left(a^2+2 a+1+4 h\right) \arctan\left(\frac{\sqrt{2h}}{a+1}\right),
\end{split}\end{equation*}
and
\begin{equation*}\begin{split}
\frac{\mathrm{\partial}M'(h)}{\mathrm{\partial}a}=&\frac{\sqrt{h} \left(a^2-22 a+4 h+1\right)}{\sqrt{2} \left(a^2-2 a+2 h+1\right)}-\frac{\sqrt{h} \left(a^2+2 a+4 h+1\right)}{\sqrt{2} \left(a^2+2 a+2 h+1\right)}-\pi  (a+1)\\
   &+(11-a) \arctan\left(\frac{\sqrt{2h}}{a-1}\right)+(a+1) \arctan\left(\frac{\sqrt{2h}}{a+1}\right).
\end{split}\end{equation*}
When $a\geq11$, one has
\begin{equation}\begin{split}\label{dma}
\frac{\mathrm{\partial}M'(h)}{\mathrm{\partial}a}\leq&\frac{\sqrt{h} \left(a^2-22 a+4 h+1\right)}{\sqrt{2} \left(a^2-2 a+2 h+1\right)}-\frac{\sqrt{h} \left(a^2+2 a+4 h+1\right)}{\sqrt{2} \left(a^2+2 a+2 h+1\right)}-\pi  (a+1)+(a+1)\frac{\pi}{2}\\
   =&-\frac{2 \sqrt{2} a \sqrt{h} \left(5 a^2+10 a+8 h+5\right)}{\left(a^2-2 a+2 h+1\right) \left(a^2+2 a+2 h+1\right)}-\frac{1}{2} \pi  (a+1)<0.
\end{split}\end{equation}
When $1<a<11$, we consider an  auxiliary function
\begin{equation*}\begin{split}
B(h)=&\frac{\sqrt{h} \left(a^2-22 a+4 h+1\right)}{\sqrt{2} \left(a^2-2 a+2 h+1\right)}-\frac{\sqrt{h} \left(a^2+2 a+4 h+1\right)}{\sqrt{2} \left(a^2+2 a+2 h+1\right)}-\pi  (a+1)\\
   &+(11-a) \min\left(\frac{\sqrt{2h}}{a-1},\frac{\pi}{2}\right)+(a+1) \min\left(\frac{\sqrt{2h}}{a+1},\frac{\pi}{2}\right)\\
   =&B_1(h)+(11-a) \min\left(\frac{\sqrt{2h}}{a-1},\frac{\pi}{2}\right)+(a+1) \min\left(\frac{\sqrt{2h}}{a+1},\frac{\pi}{2}\right),
\end{split}\end{equation*}
satisfying
\begin{equation}\label{MaB}
\frac{\mathrm{\partial}M'(h)}{\mathrm{\partial}a}\leq B(h),
\end{equation}
where
$$B_1(h)=\frac{\sqrt{h} \left(a^2-22 a+4 h+1\right)}{\sqrt{2} \left(a^2-2 a+2 h+1\right)}-\frac{\sqrt{h} \left(a^2+2 a+4 h+1\right)}{\sqrt{2} \left(a^2+2 a+2 h+1\right)}-\pi  (a+1).$$
We will prove that $B(h)<0$ for $0<h<\frac{a^2-1}{2}$.
When $1<a\leq\frac{\pi^2+4}{\pi^2-4}\approx2.36295$,
\begin{equation}\label{B1}
B(h)=\left\{\begin{array}{ll}
\begin{array}{ll}B_1(h)+(11-a)\frac{\sqrt{2h}}{a-1}+\sqrt{2h}, \end{array} &\mbox{ $0<h\leq\frac{(a-1)^2\pi^2}{8}$,}
\\[2ex]
\begin{array}{ll}B_1(h)+(11-a)\frac{\pi}{2}+\sqrt{2h},\end{array} &\mbox{ $\frac{(a-1)^2\pi^2}{8}<h<\frac{a^2-1}{2}$,}
\end{array} \right.
\end{equation}
and $\frac{\pi^2+4}{\pi^2-4}<a<11$,
\begin{equation}\label{B2}
B(h)=B_1(h)+(11-a)\frac{\sqrt{2h}}{a-1}+\sqrt{2h}, \quad 0<h<\frac{a^2-1}{2}.
\end{equation}
First, consider the function
\begin{equation}\begin{split}\label{Bh1}
B(h)=&B_1(h)+(11-a)\frac{\sqrt{2h}}{a-1}+\sqrt{2h}\\
=&\frac{1}{(a-1) \left((a-1)^2+2 h\right) \left((a+1)^2+2 h\right)}U(h,a)
\end{split}\end{equation}
defined by \eqref{B1} and \eqref{B2} in the regions $\mathcal{D}_1=\{(h,a)\mid 0<h\leq\frac{(a-1)^2\pi^2}{8}, 1<a\leq\frac{\pi^2+4}{\pi^2-4}\}$ and $\mathcal D_2=\{(h,a)\mid 0<h<\frac{a^2-1}{2}, \frac{\pi^2+4}{\pi^2-4}<a<11\}$, where
\begin{equation}\begin{split}\label{Ha}
U(h,a)=&-\pi  \left(a^2-1\right)^3-10 \sqrt{2} (a-1) (a+1)^2 \sqrt{h}-4 \pi  \left(a^4-1\right) h+8 \sqrt{2} \left(3 a^2+2 a+5\right) h^{3/2}\\
&-4 \pi  \left(a^2-1\right) h^2+40\sqrt{2} h^{5/2}.
\end{split}\end{equation}
Let $\sqrt{h}=t$, $U(h,a)$ becomes
\begin{equation*}\begin{split}
\tilde U(t,a)=&-\pi  \left(a^2-1\right)^3-10 \sqrt{2} (a-1) (a+1)^2t-4 \pi  \left(a^4-1\right)t^2+8 \sqrt{2} \left(3 a^2+2 a+5\right) t^3\\
&-4 \pi  \left(a^2-1\right) t^4+40\sqrt{2} t^5,
\end{split}\end{equation*}
and the two regions $\mathcal D_1$ and $\mathcal D_2$ become $\tilde{\mathcal{D}}_1=\{(t,a)\mid 0<t\leq\frac{\sqrt{2}(a-1)\pi}{4}, 1<a\leq\frac{\pi^2+4}{\pi^2-4}\}$ and $\tilde{\mathcal{D}}_2=\{(t,a)\mid 0<t<\frac{\sqrt{2(a^2-1)}}{2}, \frac{\pi^2+4}{\pi^2-4}<a<11\}$.
By a direct computation,
\begin{equation*}
\begin{split}
  \frac{\partial\tilde U(t,a)}{\partial t} =& -10 \sqrt{2} (a-1) (a+1)^2-8 \pi  \left(a^4-1\right) t+24 \sqrt{2} \left(3 a^2+2 a+5\right) t^2\\
  &-16 \pi  (a^2-1)t^3+200 \sqrt{2} t^4,\\
  \frac{\partial\tilde U(t,a)}{\partial a} =&-6 \pi a (a^2-1)^2 -10 \sqrt{2} (a+1) (3 a-1) t -16 \pi  a^3 t^2+16 \sqrt{2} (3 a+1) t^3-8 \pi  a t^4,
\end{split}
\end{equation*}
and the resultant of $\frac{\partial\tilde U(t,a)}{\partial t}$ and  $ \frac{\partial\tilde U(t,a)}{\partial a}$ with respect to $a$, denoted by $\Phi(t)$ is
\begin{equation}\label{Phi}
\Phi(t)={\rm res}\left(\frac{\partial\tilde U(t,a)}{\partial t},\frac{\partial\tilde U(t,a)}{\partial a},a\right)=131072\pi t^5R(t),
\end{equation}
where
\begin{eqnarray*}
R(t)&=&45000 (2 + \pi^2) (375 + 16 \pi^2) +
 3000 \sqrt{2} \pi (56450 + 11747 \pi^2 + 336 \pi^4) t\\
 && +
 300 (1473125 + 2255190 \pi^2 + 207432 \pi^4 +3456 \pi^6) t^2\\
 && +
 300 \sqrt{2} \pi (7666400 + 3011799 \pi^2 + 134688 \pi^4 +
    768 \pi^6) t^3 \\
    &&+
 6 (192643750 + 1399761625 \pi^2 + 195157400 \pi^4 +
    4223488 \pi^6 + 6144 \pi^8) t^4\\
    && +
 2 \sqrt{2} \pi (4680943500 + 4445786815 \pi^2 +
    239484496 \pi^4 + 1287168 \pi^6) t^5 \\
    &&+ (1108535625 +
    51111020840 \pi^2 + 8104828344 \pi^4 + 145305984 \pi^6 +
    208896 \pi^8) t^6 \\
    &&+
 \sqrt{2} \pi (32677962750 + 35979978671 \pi^2 +
    1871600736 \pi^4 + 9398016 \pi^6) t^7\\
    && +
 4 (141750000 + 30756096470 \pi^2 + 5450594003 \pi^4 +
    87435136 \pi^6 + 110848 \pi^8) t^8 \\
    &&+
 2 \sqrt{2} \pi (32406672915 + 30091038442 \pi^2 +
    1562509712 \pi^4 + 7630080 \pi^6) t^9\\
    && +
 8 \pi^2 (12869010693 + 2941897493 \pi^2 + 49102624 \pi^4 +
    54528 \pi^6) t^{10}\\
    && +
 8 \sqrt{2} \pi (7086052044 + 4065513609 \pi^2 +
    277371704 \pi^4 + 1452672 \pi^6) t^{11}\\
    && +
 24 \pi^2 (563317857 + 310229714 \pi^2 + 7966304 \pi^4 +
    8192 \pi^6) t^{12}\\
    && +
 96 \sqrt{2} \pi (283435200 - 33983411 \pi^2 + 4755340 \pi^4 +
    39104 \pi^6) t^{13}\\
    && +
 32 \pi^2 (280810800 - 27940071 \pi^2 + 882864 \pi^4 +
    1024 \pi^6) t^{14} \\
    &&+
 240 \sqrt{2} \pi^3 (2323215 - 152624 \pi^2 + 1536 \pi^4) t^{15}\\
 && +
 19200 \pi^4 (1605 - 52 \pi^2) t^{16} + 320000 \sqrt{2} \pi^5 t^{17}.
\end{eqnarray*}
Note that all the coefficients of $R(t)$ is positive, it follows that $R(t)$ has no zeros on $(0,+\infty)$, further $\Phi(t)$ has no zeros on $(0,+\infty)$ by \eqref{Phi}.
This implies that there does not exist the extreme point of $\tilde U(t,a)$ in the interior of $\tilde{\mathcal{D}}_1$ and $\tilde{\mathcal{D}}_2$. Therefore, $\tilde U(t,a)$ achieves its
maximum at the boundary of  $\tilde{\mathcal{D}}_1$ and $\tilde{\mathcal{D}}_2$. It is easy to verify that
\begin{equation*}
\begin{split}
   \tilde U(0,a)= &-\left(a^2-1\right)^3\pi<0,\\
   \tilde U\left(\frac{\sqrt{2}(a-1)\pi}{4},a\right)=&-\frac{1}{16} \pi  (a-1)^2 \left(\left(4+\pi ^2\right)^2 a^4+\left(32-24 \pi ^2-7 \pi ^4\right) a^3+\left(80+8 \pi ^2+15 \pi ^4\right) a^2\right.\\
   &\left.+\left(128-24 \pi ^2-13 \pi
   ^4\right) a+4 \left(4+\pi ^2\right)^2\right)<0,\quad 1<a\leq\frac{\pi^2+4}{\pi^2-4},\\
   \tilde U\left(t,\frac{\pi^2+4}{\pi^2-4}\right)=&40 \sqrt{2} t^5-\frac{64 \pi ^3 t^4}{(\pi -2)^2 (2+\pi )^2}+\frac{16 \sqrt{2} \left(48-8 \pi ^2+5 \pi ^4\right) t^3}{(\pi -2)^2 (2+\pi )^2}-\frac{128 \pi ^3 \left(16+\pi
   ^4\right) t^2}{(\pi -2)^4 (2+\pi )^4}\\
   &-\frac{320 \sqrt{2} \pi ^4 t}{(\pi -2)^3 (2+\pi )^3}-\frac{4096 \pi ^7}{(\pi -2)^6 (2+\pi )^6}<0, \quad 0<t\leq\frac{2 \sqrt{2} \pi }{\pi ^2-4},\\
   \tilde U\left(\frac{\sqrt{2(a^2-1)}}{2},a\right)=&-2 a \left(a^2-1\right) \left(\sqrt{a^2-1} (1-11 a)+2 \pi  a \left(a^2-1\right)\right)<0,\quad \frac{\pi^2+4}{\pi^2-4}<a<11,\\
   \tilde U(t,11)=&40 \sqrt{2} t^5-480 \pi  t^4+3120 \sqrt{2} t^3-58560 \pi  t^2-14400 \sqrt{2} t-1728000 \pi<0, \\
   &0<t<2 \sqrt{15},
\end{split}
\end{equation*}
thus we obtain that $\tilde U(t,a)<0$ in the regions $\tilde{\mathcal{D}}_1$ and $\tilde{\mathcal{D}}_2$. Further, $U(h,a)<0$ by \eqref{Ha} and $B(h)<0$ by \eqref{Bh1} in the regions $\mathcal D_1$ and $\mathcal D_2$.
Second, we consider the function
\begin{equation}\label{Bh2}
B(h)=B_1(h)+(11-a)\frac{\pi}{2}+\sqrt{2h}
\end{equation}
for $\frac{(a-1)^2\pi^2}{8}<h<\frac{a^2-1}{2}$ when $1<a\leq\frac{\pi^2+4}{\pi^2-4}$ given in \eqref{B1}.
A direct computation shows that
\begin{equation}\label{B21}
  B'(h)=\frac{1}{\sqrt{2h} (1-2a+a^2+2h)^2 (1+2a+a^2+2 h)^2}K(h),
\end{equation}
where
\begin{equation*}
\begin{split}
K(h)=&(a-1)^2 \left(a^2-12 a+1\right) (a+1)^4+8 \left(a^4-3 a^3+14 a^2-3 a+1\right) (a+1)^2 h\\
&+8 \left(3 a^4+7 a^3+32 a^2+7 a+3\right) h^2+32 (a+1)^2 h^3+16 h^4.
\end{split}
\end{equation*}
Observe that all the coefficients of $K'(h)$ are positive for $1<a\leq\frac{\pi^2+4}{\pi^2-4}$, thus we have $K'(h)>0$. Consequently,
\begin{equation*}\begin{split}
K(h)>&K\left(\frac{(a-1)^2\pi^2}{8}\right)\\
=&\frac{1}{256} (a-1)^2 \left(\left(4+\pi ^2\right)^4 a^6-2 \left(4+\pi ^2\right) \left(256-32 \pi ^2+4 \pi ^4+3 \pi ^6\right) a^5\right.\\
&\left.+\left(-10496+2304 \pi ^2+672 \pi ^4-16 \pi
   ^6+15 \pi ^8\right) a^4\right.\\
   &\left.-4 \left(4096-1408 \pi ^2+400 \pi ^4-16 \pi ^6+5 \pi ^8\right) a^3\right.\\
   &\left.+\left(-10496+2304 \pi ^2+672 \pi ^4-16 \pi ^6+15 \pi ^8\right) a^2\right.\\
   &\left.-2
   \left(4+\pi ^2\right) \left(256-32 \pi ^2+4 \pi ^4+3 \pi ^6\right) a+\left(4+\pi ^2\right)^4\right)>0
\end{split}\end{equation*}
for $\frac{(a-1)^2\pi^2}{8}<h<\frac{a^2-1}{2}$, and hence $B'(h)>0$ for $\frac{(a-1)^2\pi^2}{8}<h<\frac{a^2-1}{2}$ by \eqref{B21}. This means that $B(h)$ given in \eqref{Bh2} is an increasing function for $\frac{(a-1)^2\pi^2}{8}<h<\frac{a^2-1}{2}$. It follows from
$$B\left(\frac{a^2-1}{2}\right)=\frac{\sqrt{a^2-1} \left(2 a^2-11 a-1\right)-3 \pi  a \left(a^2-4 a+3\right)}{2 (a-1) a}<0$$
that $B(h)<0$ for $\frac{(a-1)^2\pi^2}{8}<h<\frac{a^2-1}{2}$ when $1<a\leq\frac{\pi^2+4}{\pi^2-4}$. Summarizing  these two results,
we have the auxiliary function $B(h)<0$ for $0<h<\frac{a^2-1}{2}$ when $1<a<11$. Thus we acquire $\frac{\mathrm{\partial}M'(h)}{\mathrm{\partial}a}<0$ for $0<h<\frac{a^2-1}{2}$ when $1<a<11$ from \eqref{MaB}. Combining the result in \eqref{dma}, we get that $\frac{\mathrm{\partial}M'(h)}{\mathrm{\partial}a}<0$ for $0<h<\frac{a^2-1}{2}$ when $a>1$.
Note that $h_2=h_2(a)\in\left(0,\frac{a^2-1}{2}\right)$, we obtain that
\begin{equation*}\begin{split}
\frac{\mathrm{d}M'(h_2)}{\mathrm{d}a}=\frac{\mathrm{d}M'(h_2(a),a)}{\mathrm{d}a}=&\frac{\mathrm{\partial}M'(h,a)}{\mathrm{\partial}h}\Big|_{h=h_2}\frac{\mathrm{d}h_2}{\mathrm{d}a}
+\frac{\mathrm{\partial}M'(h,a)}{\mathrm{\partial}a}\Big|_{h=h_2}
=\frac{\mathrm{\partial}M'(h,a)}{\mathrm{\partial}a}\Big|_{h=h_2}<0.
\end{split}\end{equation*}
When $a=\frac{7}{5}$,
$$M'\left(\frac{1}{5}\right)=\frac{1}{25} \left(-30 \sqrt{10}-82 \pi +82 \arctan\left(\frac{1}{6}\sqrt{\frac{5}{2}}\right)+338 \arctan\left(\sqrt{\frac{5}{2}}\right)\right)\approx0.358647>0,$$
thus we have $M'(h_2)\geq M'\left(\frac{1}{5}\right)>0$.
When $a=11+2\sqrt{30}$,
\begin{equation*}
\begin{split}
M'(h)=&-2 \left(\pi  h+3 \sqrt{2} \sqrt{h}+12 \sqrt{30} \pi +66 \pi \right)-2 h \arctan\left(\frac{\left(\sqrt{\frac{6}{5}}-1\right) \sqrt{h}}{\sqrt{2}}\right)\\
&+2 \left(h+12\sqrt{30}+66\right) \arctan\left(\frac{\sqrt{h}}{\sqrt{2} \left(6+\sqrt{30}\right)}\right)\\
\leq&-\pi  h-6 \sqrt{2} \sqrt{h}-6 \left(11+2 \sqrt{30}\right) \pi-2 h \arctan\left(\frac{\left(\sqrt{\frac{6}{5}}-1\right) \sqrt{h}}{\sqrt{2}}\right)\\
<&0,
\end{split}
\end{equation*}
it follows that $M'(h_2)<0$. Therefore, there exists a unique zero $a=a_*$ such that $M'(h_2(a_*))=0$, which implies that $A=(a_*,-\frac{5}{4}(a_*+1)^2)$ lies on the curve $b=b_1(a)$.

Moreover, using the method of auxiliary function as in Lemma \ref{lm3-10} and the previous proof, it is easy to verify that when $a=2$,
\begin{equation*}\begin{split}
M'(h)=&-\frac{9
   \pi }{2}-6 \sqrt{2h}-2 \pi  h+\frac{1}{2} (4 h+9) \arctan\left(\frac{\sqrt{2h}}{3}\right)+\frac{1}{2} (39-4 h) \arctan \left(\sqrt{2h}\right)\\
   \leq&\frac{1}{3}\pi  \left(9-\pi ^2\right)<0,
\end{split}\end{equation*}
for all $h>0$. We determine that $a_*\in(\frac{7}{5},2)$.
\end{proof}

With the help of Lemma \ref{lm3-11}, two results are obtained.

\begin{lm}\label{lm3-12}
If $\max\{-\frac{5}{4}(a+1)^2,b_1(a)\}\leq b<-(a+1)^2$, then $M(h)$ has no zero on the interval $(0,+\infty)$, see the region $D_1$ with its lower boundary in Fig.  \ref{fig-bif}.
\end{lm}
\begin{proof}
If $\max\{-\frac{5}{4}(a+1)^2,b_1(a)\}\leq b<-(a+1)^2$, then $M(0)\leq0$ and $M'(h)\leq M'(h_2)\leq0$ for $h\in(0,+\infty)$. One has that $M(h)$ is monotonically decreasing on the interval $(0,+\infty)$, further $M(h)<0$ for $h\in(0,+\infty)$. This lemma is finished.
\end{proof}

\begin{lm}\label{lm3-13}
If $a>a_*$ and $b_1(a)<b<-\frac{5}{4}(a+1)^2$, then $M(h)$ has a unique zero on $(0,+\infty)$, see the region $D_2$ in Fig. \ref{fig-bif}.
\end{lm}
\begin{proof}
When $a>a_*$ and $b_1(a)<b<-\frac{5}{4}(a+1)^2$, one has $M(0)>0$ and $M'(h)\leq M'(h_2)<0$ for $h\in(0,+\infty)$. Thus, $M(h)$ is monotonically  decreasing on the interval $(0,+\infty)$, and it follows from $M(0)>0$ and $\lim_{h\rightarrow+\infty}M(h)=-\infty$ that this lemma holds.
\end{proof}

When $-\frac{3}{2}(a+1)^2<b<b_1(a)$, one has $M'(h_2)>0$, which implies that $M'(h)$ has two different simple zeros on the interval $(0,+\infty)$, denoted by $h_{11}\triangleq h_{11}(a,b)$ and $h_{12}\triangleq h_{12}(a,b)$ with $h_{11}<h_{12}$.
Thus, $M(h)$ decreases on $(0,h_{11})$ and achieves its local minimum  at $h_{11}$ first, then increases on $(h_{11},h_{12})$ and reaches its local maximum at $h_{12}$, and at last it monotonically  decreases to $-\infty$ as $h\rightarrow+\infty$.
Note that when $b=b_1(a)$, $h_{11}=h_{12}=h_2$. We study the number of zeros of $M(h)$ case by case. First, an important bifurcation curve is obtained.
\begin{lm}\label{lm3-14}
There exists a unique function $b=b_2(a)$ such that $M(h_{11}(a,b_2(a)),a,b_2(a))=0$, and  $b=b_2(a)$ is monotonically decreasing and has a unique intersection point $B=(a_{**},b_1(a_{**}))$ with $b=b_1(a)$, where $a_{**}>a_{*}$, see  Fig. \ref{fig-bif}.
\end{lm}
\begin{proof}
Firstly, we show the existence and uniqueness of $b=b_2(a)$.
When $b=-\frac{5}{4}(a+1)^2$, $a\in(1,a_*]$,  one has $M(0)=0$, further we have $M(h_{11})<0$.
When $b=-\frac{3}{2}(a+1)^2$, one has  $M(0)=\frac{1}{4}(1 + a)^4\pi>0$ and  $M'(0)=0$ from \eqref{M1M0}. Hence  $M'(h_2)>0$ and if we define $h_{11}=0$, there exists $h_{12}>h_2$ such that $M'(h_{12})=0$ by using $\lim_{h\rightarrow+\infty} M'(h)=-\infty$. From the continuity of the functions $M(h)$ and $M'(h)$, we obtain that $M(h_{11})>0$ for $b=-\frac{3}{2}(a+1)^2+\varepsilon$ with $0<\varepsilon\ll1$. Therefore, there exists
$b=b_2(a)$ such that $M(h_{11}(a,b_2(a)),a,b_2(a))=0$. Notice that
\begin{equation}\begin{split}\label{h11}
\frac{\mathrm{\partial} M(h_{11}(a,b),a,b)}{\mathrm{\partial}b}=&\frac{\partial M(h,a,b)}{\partial h}\Big|_{h=h_{11}}\frac{\partial h_{11}(a,b)}{\partial b}+\frac{\partial M(h,a,b)}{\partial b}\Big|_{h=h_{11}}\\
=&\frac{\partial M(h,a,b)}{\partial b}\Big|_{h=h_{11}}\\
=&-\pi(a+1)^2-2 \sqrt{2} \sqrt{h_{11}}-2 \pi  h_{11}-\left(a^2-2 a+2 h_{11}+1\right) \arctan\left(\frac{\sqrt{2} \sqrt{h_{11}}}{a-1}\right)\\
&+\left(a^2+2 a+2 h_{11}+1\right) \arctan\left(\frac{\sqrt{2} \sqrt{h_{11}}}{a+1}\right)\\
<&-\frac{\pi}{2}(a+1)^2-2 \sqrt{2} \sqrt{h_{11}}-\pi  h_{11}-\left(a^2-2 a+2 h_{11}+1\right) \arctan\left(\frac{\sqrt{2} \sqrt{h_{11}}}{a-1}\right)\\
<&0,
\end{split}\end{equation}
this means that the function $b=b_2(a)$ is unique.

Secondly, we prove the monotonicity of $b_2(a)$. Similar as the computation of $\frac{\mathrm{d}b_1(a)}{\mathrm{d}a}$ given in \eqref{b1a},
we obtain that
\begin{equation*}\begin{split}
\frac{\mathrm{\partial} M(h_{11}(a,b),a,b)}{\mathrm{\partial}a}=&\frac{\partial M(h,a,b)}{\partial h}\Big|_{h=h_{11}}\frac{\partial h_{11}(a,b)}{\partial a}+\frac{\partial M(h,a,b)}{\partial a}\Big|_{h=h_{11}}\\
=&\frac{\partial M(h,a,b)}{\partial a}\Big|_{h=h_{11}}\\
=&-2 b w_1(h_{11})-5 \pi  a^3-15 \pi  a^2-a \left(6 \pi  h_{11}+20 \sqrt{2} \sqrt{h_{11}}+15 \pi \right)-\pi  (6 h_{11}+5)\\
&-(a-1) \left(5 a^2-10 a+6 h_{11}+5\right) \arctan\left(\frac{\sqrt{2} \sqrt{h_{11}}}{a-1}\right)\\
&+(a+1) \left(5 a^2+10 a+6 h_{11}+5\right) \arctan\left(\frac{\sqrt{2}
   \sqrt{h_{11}}}{a+1}\right)\\
<&-2 \left(\pi  a^3+3 \pi  a^2+a \left(3 \pi  h_{11}+10 \sqrt{2} \sqrt{h_{11}}+3 \pi \right)+3 \pi  h_{11}+\pi \right)\\
&-2 (a-1) \left(a^2-8 a+3 h_{11}+1\right) \arctan\left(\frac{\sqrt{2} \sqrt{h_{11}}}{a-1}\right)\\
&+2 (a+1) \left(a^2+2 a+3 h_{11}+1\right) \arctan\left(\frac{\sqrt{2}\sqrt{h_{11}}}{a+1}\right)\\
<&-\pi  a^3-3 \pi  a^2-a \left(3 \pi  h_{11}+20 \sqrt{2} \sqrt{h_{11}}+3 \pi \right)-\pi  (3h_{11}+1)\\
&-2 (a-1) \left(a^2-8 a+3 h_{11}+1\right) \arctan\left(\frac{\sqrt{2} \sqrt{h_{11}}}{a-1}\right)\\
\triangleq & L(h)\Big|_{h=h_{11}},
\end{split}\end{equation*}
where
\begin{equation}\label{w1}
w_1(h)=(a-1) \arctan\left(\frac{\sqrt{2} \sqrt{h}}{a-1}\right)-(a+1)\arctan\left(\frac{\sqrt{2} \sqrt{h}}{a+1}\right)+\pi(a+1)>0
\end{equation} and $-\frac{3}{2}(a+1)^2<b<0$.
If $a^2-8 a+3 h+1>0$, then $L(h)<0$, and if $a^2-8 a+3 h+1<0$, then
$L(h)<-2 a \left(\pi  \left(a^2-3 a+6\right)+10 \sqrt{2} \sqrt{h}+3 \pi  h\right)<0$.
Thus, we always have
\begin{equation}\label{M11a}
\frac{\mathrm{\partial} M(h_{11}(a,b),a,b)}{\mathrm{\partial}a}< L(h)\Big|_{h=h_{11}}<0.
\end{equation}
Consequently,
\begin{equation}\label{b2a}
\frac{\mathrm{d}b_2(a)}{\mathrm{d}a}=-\left(\frac{\mathrm{\partial}M(h,a,b)}{\mathrm{\partial}a}\Big|_{h=h_{11}}\right)\Big{/}
\left(\frac{\mathrm{\partial}M(h,a,b)}{\mathrm{\partial}b}\Big|_{h=h_{11}}\right)<0
\end{equation}
from \eqref{h11} and \eqref{M11a}. This means that the function $b=b_2(a)$ is  decreasing on $a$.

Finally, we prove the existence and uniqueness of the intersection point of $b_2(a)$ and $b_1(a)$, i.e. there exists a unique intersection point $B=(a_{**}, b_1(a_{**}))$ of the curves $b_1(a)$ and $b_2(a)$. Since $M(h_{11})<0$ at $A=(a_*,-\frac{5}{4}(a_*+1)^2)$, if $b=b_2(a)$ intersects the curve $b=b_1(a)$ at $B$, then $a_{**}>a_{*}$. To begin with, we claim that the simple zero of $M''(h)$ at $h_2$ satisfies $0<h_2<\frac{4}{5}$. When $1<a<\sqrt{\frac{13}{5}}$, by the previous discussion, $0<h_2<\frac{a^2-1}{2}<\frac{4}{5}$. When $a\geq\sqrt{\frac{13}{5}}$,
\begin{equation*}
\begin{split}
M''\left(\frac{4}{5}\right)=&\frac{5 \sqrt{10} \left(13-5 a^2\right) b}{25 a^4+30 a^2+169}-\frac{2 \left(25 \pi  a^4+\left(30 \pi -20 \sqrt{10}\right) a^2+169 \pi +52 \sqrt{10}\right)}{25 a^4+30
   a^2+169}\\
   &-2 \arctan\left(\frac{2 \sqrt{\frac{2}{5}}}{a-1}\right)+2 \arctan\left(\frac{2 \sqrt{\frac{2}{5}}}{a+1}\right).
\end{split}
\end{equation*}
Recall that $-\frac{3}{2}(a+1)^2<b<-\frac{5}{4}(a+1)^2,$ we have
\begin{equation*}
\begin{split}
M''\left(\frac{4}{5}\right)<&M''\left(\frac{4}{5}\right)\Big|_{b=-\frac{3}{2}(a+1)^2}\\
=&\frac{25 \left(3 \sqrt{10}-4 \pi \right) a^4+150 \sqrt{10} a^3-40 \left(\sqrt{10}+3 \pi \right) a^2-390 \sqrt{10} a-13 \left(31 \sqrt{10}+52 \pi \right)}{50 a^4+60
   a^2+338}\\
   &-2 \arctan\left(\frac{2
   \sqrt{\frac{2}{5}}}{a-1}\right)+2 \arctan\left(\frac{2 \sqrt{\frac{2}{5}}}{a+1}\right)\\
<&\frac{25 \left(3 \sqrt{10}-4 \pi \right) a^4+150 \sqrt{10} a^3-40 \left(\sqrt{10}+3 \pi \right) a^2-390 \sqrt{10} a-13 \left(31 \sqrt{10}+52 \pi \right)}{50 a^4+60
   a^2+338}\\
<&0,
\end{split}
\end{equation*}
which implies that $0<h_2<\frac{4}{5}$ when $a\geq\sqrt{\frac{13}{5}}$ due to the fact that $M''(h)$ is monotonically decreasing on the interval $(0,\frac{a^2-1}{2})$. Moreover, along $b=b_1(a)$, $h_{11}=h_{12}=h_2$, and from \eqref{b1a}
\begin{equation}\begin{split}\label{dMh2}
\frac{\mathrm{d} M(h_{2}(a,b_1(a)),a,b_1(a))}{\mathrm{d}a}=&\frac{\partial M(h,a,b_1)}{\partial h}\Big|_{h=h_{2}(a,b_1)}\left(\frac{\partial h_{2}(a,b_1)}{\partial a}+\frac{\partial h_{2}(a,b)}{\partial b}\Big|_{b=b_1}\frac{\mathrm{d}b_1}{\mathrm{d}a}\right)\\
&+\frac{\partial M(h_2,a,b_1)}{\partial a}+\frac{\partial M(h_2,a,b)}{\partial b}\Big|_{b=b_1}\frac{\mathrm{d}b_1}{\mathrm{d}a}\\
=&\frac{\partial M(h_2,a,b_1)}{\partial a}+\frac{\partial M(h_2,a,b)}{\partial b}\Big|_{b=b_1}\frac{\mathrm{d}b_1}{\mathrm{d}a}\\
=&\left(\frac{\partial M(h,a,b)}{\partial a}-\frac{\partial M(h,a,b)}{\partial b}\frac{\mathrm{\partial}M'(h,a,b)}{\mathrm{\partial}a}\Big{/}
\frac{\mathrm{\partial}M'(h,a,b)}{\mathrm{\partial}b}\right)\Big|_{h=h_2, b=b_1}\\
\triangleq&S(h,a,b)\big|_{h=h_2, b=b_1}.
\end{split}\end{equation}
Write
\begin{equation*}\begin{split}
S(h,a,b)=&\left(-2w_1(h)-\frac{8 \sqrt{2} a\sqrt{h} M_b}{M_{1b} \left(a^2-2 a+2 h+1\right)
   \left(a^2+2 a+2 h+1\right)}\right)b\\
   &-5 \pi  a^3-15 \pi  a^2-a \left(6 \pi  h+20 \sqrt{2} \sqrt{h}+15 \pi \right)-\pi  (6 h+5)\\
&-(a-1) \left(5 a^2-10 a+6 h+5\right) \arctan\left(\frac{\sqrt{2} \sqrt{h}}{a-1}\right)\\
&+(a+1) \left(5 a^2+10 a+6 h+5\right) \arctan\left(\frac{\sqrt{2}
   \sqrt{h}}{a+1}\right)\\
&-\frac{M_b}{M_{1b}}\left(-\frac{16 \sqrt{2} a
   h^{3/2}}{\left(a^2-2 a+2 h+1\right)
   \left(a^2+2 a+2 h+1\right)}-6 (a-1) \arctan\left(\frac{\sqrt{2}
   \sqrt{h}}{a-1}\right)\right.\\
   &\left.+6 (a+1) \arctan\left(\frac{\sqrt{2}
   \sqrt{h}}{a+1}\right)-6 \pi  a-6 \pi \right),
\end{split}\end{equation*}
where $M_b=\frac{\mathrm{\partial}M(h,a,b)}{\mathrm{\partial}b}$ and $M_{1b}=\frac{\mathrm{\partial}M'(h,a,b)}{\mathrm{\partial}b}$.
Notice that the coefficient of $b$ is negative by \eqref{Mb}, \eqref{h11} and \eqref{w1}, hence for $a>a_*>\frac{7}{5}$, $0<h<\frac{4}{5}$ and $-\frac{3}{2}(a+1)^2<b<-\frac{5}{4}(a+1)^2$,
\begin{equation}\begin{split}\label{Sh}
S(h,a,b)>S(h,a,b)|_{b=-\frac{5}{4}(a+1)^2}
=-\frac{1}{2} \left(\arctan\left(\frac{\sqrt{2} \sqrt{h}}{a-1}\right)-\arctan \left(\frac{\sqrt{2} \sqrt{h}}{a+1}\right)+\pi \right)^{-1}S_1(h,a),
\end{split}\end{equation}
where
$$S_1(h,a)=\frac{\partial M(h,a,b)}{\partial a}\frac{\mathrm{\partial}M'(h,a,b)}{\mathrm{\partial}b}-\frac{\partial M(h,a,b)}{\partial b}\frac{\mathrm{\partial}M'(h,a,b)}{\mathrm{\partial}a}\Big|_{b=-\frac{5}{4}(a+1)^2}.$$
For each fixed $a>a_*>\frac{7}{5}$, the function $S_1(h,a)$ exhibits one of two behaviors on the interval $(0,\frac{4}{5})$: it is either entirely decreasing or it decreases to a minimum before increasing, see the Appendix for the proof of this property. Therefore, $S_1(h,a)$ achieves its maximum at $h=0$ or $h=\frac{4}{5}$. A direct computation shows that
\begin{eqnarray*}
S_1(0,a)&=&-\pi ^2(a+1)^3<0,\\
S_1\left(\frac{4}{5},a\right)&=&-\frac{1}{5 \left(25 a^4+30 a^2+169\right)}\left(125 \pi ^2 a^7+375 \pi ^2 a^6-75 \left(12 \sqrt{10}-7 \pi \right)\pi  a^5\right.\\
&&\left.+25 \pi  \left(104 \sqrt{10}+23 \pi \right) a^4+5 \left(800+552 \sqrt{10} \pi +259 \pi ^2\right) a^3\right.\\
&&\left.+5 \left(1600+1120 \sqrt{10} \pi +537 \pi ^2\right)a^2+\left(9120-6500 \sqrt{10} \pi +2535 \pi ^2\right) a\right.\\
&&\left.+169 \pi  \left(24 \sqrt{10}+5 \pi \right)\right)\\
&&+2 a \left(a^2+10 a-31\right) \arctan\left(\frac{2 \sqrt{\frac{2}{5}}}{a-1}\right) \arctan\left(\frac{2 \sqrt{\frac{2}{5}}}{a+1}\right)\\
&&+\left(-a^3-17 a^2+17a+1\right) \arctan\left(\frac{2 \sqrt{\frac{2}{5}}}{a-1}\right)^2-(a+1)^3 \arctan\left(\frac{2 \sqrt{\frac{2}{5}}}{a+1}\right)^2\\
&&+\frac{2 }{25 a^2+50a+65}\left(-25 \pi  a^5-300 \pi  a^4+30 \left(3 \sqrt{10}+7 \pi \right) a^3+60 \left(4\sqrt{10}+15 \pi \right) a^2\right.\\
&&\left.+5 \left(74 \sqrt{10}+403 \pi \right) a+156 \sqrt{10}\right) \arctan\left(\frac{2 \sqrt{\frac{2}{5}}}{a-1}\right)\\
&&+\frac{2 }{25 a^2-50 a+65}\left(25 \pi  a^5+25 \pi  a^4-10 \left(9 \sqrt{10}+\pi \right) a^3+10 \left(44 \sqrt{10}+7 \pi \right) a^2\right.\\
&&\left.-5 \left(74 \sqrt{10}-29 \pi \right) a+65 \pi +156 \sqrt{10}\right) \arctan\left(\frac{2 \sqrt{\frac{2}{5}}}{a+1}\right).
\end{eqnarray*}
By careful analysis, $S_1\left(\frac{4}{5},a\right)<0$ for $a>\frac{7}{5}$. Hence, we have $S_1(h,a)<0$ for $a>\frac{7}{5}$ and $0<h<\frac{4}{5}$, further, $S(h,a,b)>0$ by \eqref{Sh}.
It follows from $0<h_2=h_{11}<\frac{4}{5}$ and $-\frac{3}{2}(a+1)^2<b_1(a)<-\frac{5}{4}(a+1)^2$ for $a>a_*>\frac{7}{5}$ that
$$\frac{\mathrm{d} M(h_{11}(a,b_1(a)),a,b_1(a))}{\mathrm{d}a}=\frac{\mathrm{d} M(h_{2}(a,b_1(a)),a,b_1(a))}{\mathrm{d}a}>0$$
by \eqref{dMh2}. Combining the fact that
\begin{equation*}\begin{split}
&M(h_{11}(a,b_1(a)),a,b_1(a))<0, \quad \mbox{when}\ a=a_*,\\
&M(h_{11}(a,b_1(a)),a,b_1(a))>M(h_{11}(3,b_1(3)),3,-22)>0, \quad \mbox{when}\ a=3,
\end{split}\end{equation*}
there exists a unique intersection point $B$ for the curves $b=b_1(a)$ and $b=b_2(a)$. Here we have used $\frac{\partial M(h,a,b)}{\partial b}<0$ from \eqref{h11} and verified  $b_1(3)<-22$ and $M(h,3,-22)>0$ for $0<h<\frac{4}{5}$ by lengthy calculations.
\end{proof}

 In view of this, one has

\begin{lm}\label{lm3-15}
\begin{itemize}
\item[(i)] When $-\frac{3}{2}(a+1)^2<b<b_2(a)$, $1<a\leq a_{**}$,  $M(h)$ has a unique simple zero on the interval $(0,+\infty)$, see the region $D_7$ in Fig. \ref{fig-bif}.

\item[(ii)] When $-\frac{3}{2}(a+1)^2<b<b_1(a)$, $a>a_{**}$, $M(h)$ has a unique simple zero on the interval $(0,+\infty)$, see the region $D_6$ in Fig. \ref{fig-bif}.

\item[(iii)] When $b=b_1(a)$, $a>a_{**}$,  $M(h)$ has a unique zero on the interval $(0,+\infty)$, see the upper boundary of the region $D_6$ in Fig. \ref{fig-bif}.

\item[(iv)] When $b=b_2(a), 1<a<a_{**}$, $M(h)$ has two different zeros on the interval $(0,+\infty)$, one of which is $h_{11}$ with multiplicity $2$, and the other simple zero lies in $(h_{12}, +\infty)$,  see the dark-green curve in Fig. \ref{fig-bif}.
\end{itemize}\end{lm}
\begin{proof}
(i) When $-\frac{3}{2}(a+1)^2<b<b_2(a)$, $1<a\leq a_{**}$, $M(h)$ has a local minimum at $h_{11}$ and a local maximum at $h_{12}$, moreover, $M(h_{11})>0$. Thus $M(h_{12})>M(h_{11})>0$ and $M(h)$ has a unique simple zero on the interval $(0,+\infty)$ which lies in $(h_{12},+\infty)$, owing to
the fact that $M(h)$ is monotonically deceasing on $(h_{12},+\infty)$ and $\lim_{h\rightarrow +\infty}M(h)=-\infty$.

(ii) When $-\frac{3}{2}(a+1)^2<b<b_1(a)$, $a>a_{**}$, $M(h)$ has a local minimum at $h_{11}$ and a local maximum at $h_{12}$, moreover, $M(h_{11})>0$ by using Lemma \ref{lm3-14}. Thus $M(h_{12})>M(h_{11})>0$. Similar to case (i), $M(h)$ has a unique simple zero on the interval $(0,+\infty)$.

(iii) When $b=b_1(a)$, $a>a_{**}$, from Lemma \ref{lm3-14}, $M(h_2)=M(h_{11})=M(h_{12})>0$, thus $M(h)$ has a unique simple zero on the interval $(0,+\infty)$.

(iv) When $b=b_2(a), 1<a<a_{**}$, by the definition of $b=b_2(a)$, $M(h)$ has a local minimum $0$ at $h_{11}$, thus it is a zero with multiplicity $2$. Since  $M(h_{12})>M(h_{11})=0$, $M(h)$ has an extra zero in $(h_{12}, +\infty)$.
\end{proof}

Consider the case when $b=b_1(a), 1<a<a_{*}$,  it follows from $M(0)<0$ and $M'(h_2)=0$ that $M(h_{11})=M(h_{12})=M(h_2)<0$. Combining the proof of Lemma \ref{lm3-15} (iv), there exists a function $b=b_3(a)\in(b_2(a),b_1(a))$ such that $M(h_{12}(a,b_3(a)),a,b_3(a))=0$. Similar to the proofs in \eqref{h11} and \eqref{b2a}, we can obtain
\begin{lm}\label{lm3-16}
The function $b=b_3(a)\in(b_2(a),b_1(a))$ satisfying  $M(h_{12}(a,b_3(a)),a,b_3(a))=0$ is unique and  decreasing on $a$, and it intersects the curves $b=b_2(a)$ and $b=b_1(a)$ at $B$, see the light-green curve in Fig. \ref{fig-bif}.
\end{lm}

Moreover, we must emphasize that
\begin{lm}\label{lm3-17}
The curve $b=b_3(a)$ intersects  the curve $b=-\frac{5}{4}(a+1)^2$ at a unique point $C$.
\end{lm}
\begin{proof}
Consider $b=-\frac{5}{4}(a+1)^2$.
Note that when $a=\frac{6}{5}$,
$$M\left(\frac{1}{2}\right)=\frac{1}{150} \left(-505-219 \pi +219 \arctan\frac{5}{11}+897 \arctan5\right)\approx0.882421.$$
Since $M(h)$ attains a local maximum at $h_{12}$, we have $M(h_{12})\geq M(\frac{1}{2})>0$ when $a=\frac{6}{5}$.
Moreover, when $a=a_*$, one has $M(h_{11})=M(h_{12})=M(h_2)<0$ owing to $M(0)=0$ and $M'(h)\leq0$, hence the curve $b=b_3(a)$ can intersect the curve $b=-\frac{5}{4}(a+1)^2$ for $a\in(\frac{6}{5},a_*)$. We will prove that the intersect point is unique,  denoted by $C=(a_3,-\frac{5}{4}(a_3^2+1)^2)$.
From Lemma \ref{lm3-11}, when $1<a<a_*<2$, one has $M'(h_2)>0$ and there exist $h_{11}$ and $h_{12}$ such that $M'(h_{11})=M'(h_{12})=0$ by \eqref{MM01}. Obviously, $h_{11}$ satisfies
\begin{equation*}
0<h_{11}<h_2<\frac{a^2-1}{2}<\frac{a_*^2-1}{2}<\frac{3}{2}
\end{equation*}
and $h_{12}>h_2$. It is easy to verify that for $1<a<a_*<2$,
\begin{equation*}\begin{split}
M'\left(\frac{3}{2}\right)=&-6 \sqrt{3}-\frac{1}{2} \pi  \left(a^2+2 a+7\right)+\frac{1}{2} \left(a^2+2 a+7\right) \arctan\left(\frac{\sqrt{3}}{a+1}\right)\\
&+\frac{1}{2} \left(-a^2+22 a-7\right) \arctan\left(\frac{\sqrt{3}}{a-1}\right)\\
<&-6 \sqrt{3}-\frac{1}{2} \pi  \left(a^2+2 a+7\right)+\frac{1}{2} \left(a^2+2 a+7\right)\frac{\sqrt{3}}{a+1}+\frac{1}{2} \left(-a^2+22 a-7\right)\frac{\pi}{2}\\
<&\frac{3 \pi }{4}-\frac{7 \sqrt{3}}{2}<0.
\end{split}\end{equation*}
By the continuity of $M'(h)$, we have
$$h_2<h_{12}<\frac{3}{2}.$$
In the following, we consider
\begin{equation}\begin{split}\label{J}
\frac{\mathrm{d}M(h_{12}(a),a)}{\mathrm{d}a}=&\frac{\mathrm{\partial}M(h,a)}{\mathrm{\partial}h}\Big|_{h=h_{12}}\frac{\mathrm{d}h_{12}}{\mathrm{d}a}
+\frac{\mathrm{\partial}M(h,a)}{\mathrm{\partial}a}\Big|_{h=h_{12}}\\
=&\frac{\mathrm{\partial}M(h,a)}{\mathrm{\partial}a}\Big|_{h=h_{12}}\\
=&5 \sqrt{2} (1-3 a) \sqrt{h_{12}}-\pi  (a+1) h_{12}+(a+1) h_{12} \arctan\left(\frac{\sqrt{2}
   \sqrt{h_{12}}}{a+1}\right)\\
&+\left(15 a^2-20a+5+(11-a)h_{12}\right) \arctan\left(\frac{\sqrt{2} \sqrt{h_{12}}}{a-1}\right)\\
\triangleq &J(h)\Big|_{h=h_{12}}.
\end{split}\end{equation}
For $1<a<2$,
\begin{equation*}\begin{split}
J''(h)=\frac{\sqrt{2}J_2(h)}{\sqrt{h} \left(a^2-2 a+2
   h+1\right)^2 \left(a^2+2 a+2 h+1\right)^2},
\end{split}\end{equation*}
where $J_2(h)=e_3h^3+e_2h^2+e_1h+e_0$ with
\begin{equation*}\begin{split}
e_3=&8 (11 a-5)>0, \\
e_2=&4 \left(26 a^3+25 a^2+16 a-15\right)>0,\\
e_1=&10 \left(3 a^3-a^2+5 a-3\right) (a+1)^2>0,\\
e_0=&-5 (a-1)^2 (a+1)^4<0.
\end{split}\end{equation*}
Since when $1<a<2$, $J_2(h)$ is a polynomial of $h$ with degree 3 and
\begin{equation*}\begin{split}
\lim_{h\rightarrow-\infty}J_2(h)=&-\infty,\\
J_2(-3)=&625 - 1780 a + 725 a^2 + 776 a^3 - 145 a^4 - 100 a^5 - 5 a^6>0,\\
J_2(0)=&-5 (-1 + a)^2 (1 + a)^4<0,\\
J_2\left(\frac{3}{2}\right)=&-320 + 416 a + 320 a^2 + 344 a^3 + 80 a^4 + 35 a^5 - 5 a^6>0,
\end{split}\end{equation*}
we obtain that $J_2(h)$ has only one simple zero on the interval $(0,\frac{3}{2})$, which means that $J''(h)$ has only one simple zero on the interval $(0,\frac{3}{2})$.
Further, $J'(h)$ achieves its maximum at $h=0$ or $h=\frac{3}{2}$ since $J'(h)$ decreases first and then increases on $(0,\frac{3}{2})$.
By a direct computation, $J'(0)=-(1 + a) \pi<0$ and
\begin{equation*}\begin{split}
J'\left(\frac{3}{2}\right)=&-\frac{\pi  a^5+\pi  a^4+2 \left(5 \sqrt{3}+2 \pi \right) a^3+4 \left(5 \sqrt{3}+\pi \right) a^2+2 \left(17 \sqrt{3}+8 \pi \right) a+16 \pi }{a^4+4 a^2+16}\\
&-(a-11) \arctan\left(\frac{\sqrt{3}}{a-1}\right)+(a+1) \arctan\left(\frac{\sqrt{3}}{a+1}\right).
\end{split}\end{equation*}
It is easy to verify that there exists a unique zero $\hat a\in(1,2)$ such that $J'\left(\frac{3}{2}\right)>0$ on $(1,\hat a)$ and $J'\left(\frac{3}{2}\right)<0$ on $(\hat a,2)$ by the monotonicity of $J'\left(\frac{3}{2}\right)$ on $a$. Thus, for $1<a<\hat a$, $J'(h)$ has a unique simple zero on $(0,\frac{3}{2})$, and for $\hat a\leq a<2$, $J'(h)<0$ on $(0,\frac{3}{2})$. From this, we conclude that for $1<a<\hat a$, $J(h)$ decreases first and then increases on $(0,\frac{3}{2})$,
and for $\hat a\leq a<2$, $J(h)$ decreases on $(0,\frac{3}{2})$.
Notice that $J(0)=0$, and
\begin{equation*}\begin{split}
J\left(\frac{3}{2}\right)=&\frac{1}{2} \left(-3 \pi  a-30 \sqrt{3} a-3 \pi +10 \sqrt{3}\right)+\frac{3}{2} (a+1)
   \arctan\left(\frac{\sqrt{3}}{a+1}\right)\\
   &+\frac{1}{2} \left(30 a^2-43 a+43\right) \arctan\left(\frac{\sqrt{3}}{a-1}\right)<0,\ \ 1<a<2.
\end{split}\end{equation*}
It follows that $J(h)<0$ on $(0,\frac{3}{2})$. Therefore, by \eqref{J} and $0<h_{12}<\frac{3}{2}$, along $b=-\frac{5}{4}(a+1)^2$, $M(h_{12}(a),a)$ is monotonically decreasing for $1<a<a_*$.
We conclude that there exists a unique zero $a_3$ such that $M(h_{12}(a_3),a_3)=0$. That is, the intersect point of the curves $b=b_3(a)$ and $b=-\frac{5}{4}(a+1)^2$
is unique.
\end{proof}

Following Lemmas \ref{lm3-16} and \ref{lm3-17}, we have
\begin{lm}\label{lm3-18}
\begin{itemize}
\item[(i)] When $b_2(a)<b<\min\{-\frac{5}{4}(a+1)^2,b_3(a)\}$, $M(h)$ has three simple zeros on $(0,+\infty)$, see the region $D_4$ in Fig. \ref{fig-bif}.
\item[(ii)] When $-\frac{5}{4}(a+1)^2\leq b< b_3(a)$, $1<a<a_3$, $M(h)$ has two simple zeros on $(0,+\infty)$, see the region $D_5$ with its lower boundary in Fig. \ref{fig-bif}.
\item[(iii)]  When $b=b_3(a)$, $1<a\leq a_3$, $M(h)$ has a zero  with multiplicity $2$ on $(0,+\infty)$, see the upper boundary of the region $D_5$ in Fig. \ref{fig-bif}.
\item[(iv)] When $\max\{-\frac{5}{4}(a+1)^2,b_3(a)\}<b<b_1(a)$, $1<a<a_*$, $M(h)$ has no zeros on $(0,+\infty)$, see the region $D_3$ in Fig. \ref{fig-bif}.
\end{itemize}
\end{lm}
\begin{proof}
(i) When $b_2(a)<b<\min\{-\frac{5}{4}(a+1)^2,b_3(a)\}$, $M(h)$ has a local minimum at $h_{11}$ and a local maximum at $h_{12}$, moreover, $M(0)>0,\  M(h_{11})<0$ and $M(h_{12})>0$. Thus there exist three simple zeros on $(0,+\infty)$ owing to $\lim_{h\rightarrow+\infty}M(h)=-\infty$.

(ii) When $-\frac{5}{4}(a+1)^2\leq b< b_3(a)$, $1<a<a_3$, $M(0)\leq0,\ M(h_{11})<0,\ M(h_{12})>0$ and $\lim_{h\rightarrow+\infty}M(h)=-\infty$. Hence there exist two simple zeros on $(0,+\infty)$.

(iii) When $b=b_3(a)$, $1<a\leq a_3$, $M(0)\leq0,\ M(h_{11})<0,\ M(h_{12})=0$ and $\lim_{h\rightarrow+\infty}M(h)=-\infty$. Thus $M(h)$ has only one zero at $h_{12}$ with multiplicity $2$ on $(0,+\infty)$.

(iv) When $\max\{-\frac{5}{4}(a+1)^2,b_3(a)\}<b<b_1(a)$, $1<a<a_*$, one has $M(0)<0,\ M(h_{11})<0,\ M(h_{12})<0$, thus $M(h)$ has no zeros on $(0,+\infty)$.
\end{proof}

Finally, we discuss the cases along the curved triangle $ACB$ and its interior.  Along the curve $\widehat{CA}$ including the end point $A$, $M(0)=0,\ M(h_{11})<0,\ M(h_{12})<0$, it follows that $M(h)$ has no zeros on $(0,+\infty)$. Along the curve $\widehat{AB}$, $M(0)>0,\ M(h_{11})<0,\ M(h_{12})<0$, it follows that $M(h)$ has one simple zero on $(0,+\infty)$. Along the curve $\widehat{CB}$, $M(0)>0,\ M(h_{11})<0,\ M(h_{12})=0$, it follows that $M(h)$ has one simple zero on the interval $(0,h_{11})$ and the other zero at $h_{12}$ with multiplicity $2$. At the point $B$, $M(h)$ has a zero at $h_{2}=h_{11}=h_{12}$ with multiplicity $3$ on $(0,+\infty)$. In the interior of  the curved triangle $ACB$, $M(0)>0,\ M(h_{11})<0,\ M(h_{12})<0$, thus $M(h)$ has one simple zero  on $(0,+\infty)$, see Fig. \ref{fig-bif} for the region $D_8$.

\begin{figure}[!htb]
	\centering
	\includegraphics[scale=0.6]{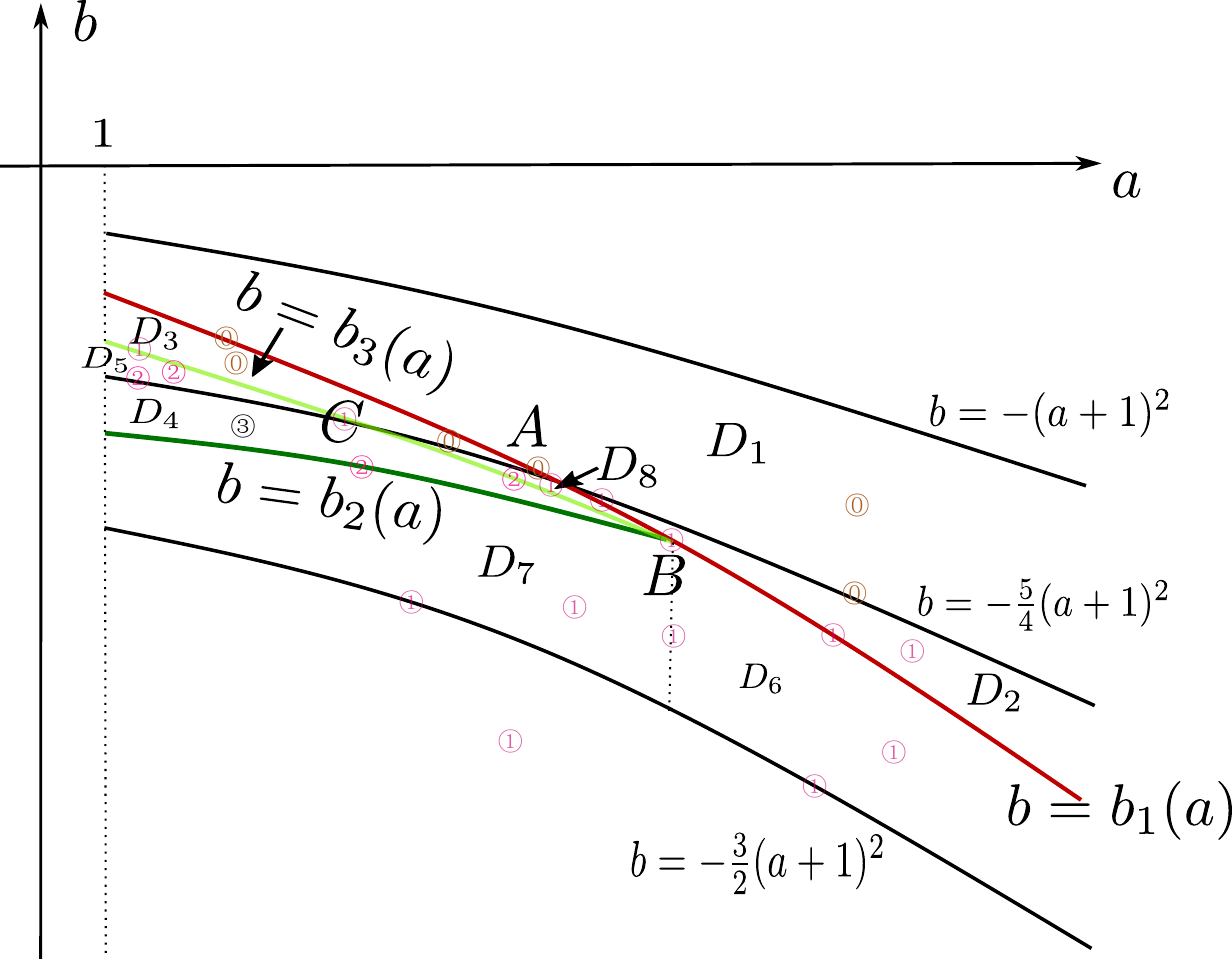}
	\caption{\footnotesize{The bifurcation diagram on the  number of zeros of $M(h)$ with respect to $(a,b)$, where $a>1$, $b<-(a+1)^2$.}}
	\label{fig-bif}
    \end{figure}


\begin{lm}
There is a value $a_0>1$ such that
\begin{description}
		\item[\bf (i)] for any  $1<a<a_0$ and $\delta>0$, system \eqref{SD} exhibits at most three crossing limit cycles when  $b<-(a+1)^2$;
        \item[\bf (ii)] for any $1<a<a_0$ and $\delta>0$, there exists a unique open interval $(c,d)\subset (-\infty, -(a+1)^2)$ of $b$
             such that system \eqref{SD} exhibits  exactly three  crossing limit cycles only when $b\in (c,d)$ and all these three crossing limit cycles are simple;
		\item[\bf (iii)] for any $a>a_0$ and $\delta>0$, system \eqref{SD} exhibits at most one crossing limit cycle when  $b<-(a+1)^2$;
        \item[\bf (iv)] for any $a>a_0$ and $\delta>0$, there exists a unique value $b_0$
             such that system \eqref{SD} exhibits exactly one  crossing limit cycle only when $b\in (-\infty, b_0)$ and this crossing limit cycle is simple.
\end{description}
\label{three-lc}
\end{lm}

\begin{proof}
By Lemmas \ref{lm3-9}-\ref{lm3-18},
there exists a value $\delta_0$ such that this lemma holds for system  \eqref{SD} when $0<\delta<\delta_0$.
Now we prove this lemma for system \eqref{ygf} when $\delta>0$ is a general value.

Assume that system \eqref{ygf} exhibits at least four crossing limit cycles when $b<-(a+1)^2$ and $\delta=\delta_*>\delta_0$.
Denote the
outermost four crossing limit cycles from outside to inside by
$\Gamma_1$, $\Gamma_2$, $\Gamma_3$ and  $\Gamma_4$.
Here we can only consider that $\Gamma_1$, $\Gamma_2$, $\Gamma_3$ and  $\Gamma_4$ are all simple.
If not, we can make small adjustments to the values of $\delta$ and $\tilde b$ so that the outermost four crossing limit cycles become simple.
By Lemma \ref{lm-infi},  all the orbits of system \eqref{SD} are positively bounded.
Then,  $\Gamma_1$ is stable,
implying that $\Gamma_2$, $\Gamma_3$ and $\Gamma_4$ are unstable, stable and unstable, respectively.

By Lemma \ref{lm-mono}, $\Gamma_1$, $\Gamma_3$ expand and  $\Gamma_2$, $\Gamma_4$ compress as $\delta$ decreases.
Keep decreasing $\delta$ from $\delta_*$ until $0<\delta<\delta_0$ or at least one of the following three cases happen:
(i) $\Gamma_2$ and $\Gamma_3$ merge into a semi-stable crossing limit cycle;
(ii) $\Gamma_4$ merges with the crossing limit cycle (if it exists) which is enclosed by  $\Gamma_4$ and closet to  $\Gamma_4$;
(iii)  $\Gamma_4$ becomes a small limit cycle.

If $\delta$ decreases to a value less than $\delta_0$, 
then system \eqref{ygf} exhibits at least four crossing limit cycles when $b<-(a+1)^2$ and $\delta<\delta_0$.
That is a contradiction.
If the decreasing of  $\delta$ stop as at least one of the cases (i)-(iii) happens, 
then increase $\tilde b$.
By Lemma \ref{lm-mono}, $\Gamma_1$, $\Gamma_3$ compress and  $\Gamma_2$, $\Gamma_4$ expand as $\tilde b$ increases.
Keep increasing $\tilde b$ until at least one of the following two cases happen:
(iv) $\Gamma_1$ and $\Gamma_2$ merge into a semi-stable crossing limit cycle;
(v) $\Gamma_3$ and $\Gamma_4$ merge into a semi-stable crossing limit cycle.
Clearly, $\tilde b<-\delta (a+1)^2$ throughout its increase,
because system \eqref{SD} exhibits no limit cycles when $b\ge-(a+1)^2$ by Lemmas \ref{lm-bg0} and \ref{lm-b0H}.

Once again decrease $\delta$ and repeat the above steps. From Theorem 1.1 of \cite{CTZ} and Lemma \ref{lm-mono},
the sum of the decrements of $\delta$ will tend to infinity as the number of  iterations tends to infinity.
Thus, after finite steps,  $\delta$  decreases to a value less than $\delta_0$ and at least four crossing limit cycles exist,
which is a contradiction.
Therefore, system \eqref{SD} exhibits at most three crossing limit cycles when $b<-(a+1)^2$. The statement {\bf (i)} is proven.


For $1<a<a_0$ and  $0<\delta<\delta_0$, there is a unique open interval $(c_\delta, d_\delta)\subset(-\infty, -(a+1)^2)$ of $b$
such that system \eqref{SD} exhibits exactly three crossing limit cycles only when $b\in (c_\delta, d_\delta)$
and these crossing limit cycles are simple.
Since the vector field of  system \eqref{SD} is rotated about $b$,
three simple crossing limit cycles of  system \eqref{SD} with $b=b_1$ could arise from the evolution of crossing limit cycles
of system  \eqref{SD} with $b=b_2$ when $b$ is changed continuously from $b_2$ to $b_1$, where $b_1, b_2\in (c_\delta, d_\delta)$.

Assume that for some $1<a<a_0$ and $\delta>0$ there are two open intervals $(c_1,d_1)$ and $(c_2,d_2)$ of $b$ such that
system \eqref{SD} exhibits exactly three crossing limit cycles when $b\in (c_1,d_1)\cup (c_2,d_2)$
and the number of crossing limit cycles is less than three when $b\in [d_2,c_1]$, where $d_2<c_1$.
Then,  at least two crossing limit cycles of system  \eqref{SD} with $b\in (c_1,d_1)$ do not arise from the evolution of crossing limit cycles
of system  \eqref{SD} with $b\in (c_2,d_2)$.
Choose two values $\bar b_1\in (c_1,d_1)$ and $\bar b_2\in (c_2,d_2)$.
Decrease $\delta$ and increase the corresponding $\tilde b_1$, $\tilde b_2$  alternately while ensuring the existence of these crossing limit cycles.
Then, $\tilde b_1$ and $\tilde b_2$ can not both be increased to the values in $(c_\delta, d_\delta)$,
which is a contradiction.
The statement {\bf (ii)} is proven.

Assume that system \eqref{ygf} exhibits at least two crossing limit cycles when $b<-(a+1)^2$, $a>a_0$ and $\delta=\delta_*>\delta_0$
and they are both simple.
It is similar to the proof of the first part of this lemma. Decrease $\delta$ and increase $\tilde b$ alternately while ensuring the existence of these two crossing limit cycles
until $0<\delta<\delta_0$.
Then, a contradiction could be derived.
Thus, system \eqref{SD} exhibits at most one crossing limit cycle when $b<-(a+1)^2$ and $a>a_0$. The statement {\bf (iii)} is proven.

Since for any $a>a_0$ and $\delta>0$ there exists a unique value $b_0$
             such that system \eqref{SD} exhibits exactly one  crossing limit cycle only when $b\in (-\infty, b_0)$ and this crossing limit cycle is simple,
increasing $\delta$ and decreasing $\tilde b$ alternately while ensuring the existence of this crossing limit cycle,
$\delta$ could be increased to any positive value.
The statement {\bf (iv)} is proven.
\end{proof}

\begin{pro}
	For any $\delta>0$, there exists a decreasing $\mathcal{C}^{\infty}$ function  $\varphi_1(a, \delta)$ with respect to $a>1$
    such that
    $-3(a+1)^2/4<\varphi_1(a, \delta)<-(a+1)^2$
    and
	\begin{description}
		\item[\bf (a)] system \eqref{SD} exhibits a grazing  loop surrounding $E_r$ if and only if $b=\varphi_1(a, \delta)$;
		\item[\bf (b)] system \eqref{SD} exhibits  a unique small limit cycle when $\varphi_1(a, \delta)<b<-(a+1)^2$,
                       which is stable;
        \item[\bf (c)] system \eqref{SD} exhibits no small limit cycle when $b<\varphi_1(a, \delta)$;
        \item[\bf (d)] for any sufficiently small $|\varsigma|$ with $\varsigma<0$, system \eqref{SD} exhibits a stable crossing limit cycle which intersects the small interval $(\varsigma,0)$  when $b=\varphi_1(a, \delta)-\varepsilon$,
	\end{description}
    where $\varepsilon>0$ is sufficiently small.
	\label{pro-1}
\end{pro}

\begin{proof}
It is enough to prove this proposition for system \eqref{ygf}, which has the same topological structure with system \eqref{SD}.
Clearly, system \eqref{ygf} exhibits a unique equilibrium $(a+1,0)$, which corresponds to $E_r$.
Consider the positive orbit $\gamma^+$ and negative orbit $\gamma^-$ passing through the origin of system \eqref{ygf}.
As defined above Lemma \ref{lm-mono}, $(x^0_{A^+}, 0)$ (resp. $(x^0_{A^-}, 0)$)
is the first intersection point of $\gamma^+$ (resp. $\gamma^-$)
and the positive $x$-axis.

On the one hand,  $E_r$ is a sink when $b=-(a+1)^2$ from Lemma \ref{lm-fi-e}.
By Lemma \ref{lm-b0H}, system \eqref{SD} exhibits no closed orbits when $b=-(a+1)^2$.
Then, $x^0_{A^+}<x^0_{A^-}$ when $\tilde{b}=-\delta(a+1)^2$.
On the other hand,
$E_r$ is a source when $b=-3(a+1)^2/4$ from Lemma \ref{lm-fi-e}.
By Lemma \ref{one-sl}, system \eqref{SD} exhibits no small limit cycle when $b=-3(a+1)^2/4$.
Then, $x^0_{A^+}>x^0_{A^-}$ when $\tilde{b}=-3\delta(a+1)^2/4$.
It follows from Lemma \ref{lm-mono} that $x^0_{A^+}$ decreases and $x^0_{A^-}$ increases as $\tilde{b}$ increases.
Then, there exists a unique function $b=\varphi_1(a, \delta)\in (-4(a-1)^2/3, -(a-1)^2)$ such that $x^0_{A^+}=x^0_{A^-}$.
The statement {\bf{(a)}} is proven.

By the monotonicity of $x^0_{A^+}$ and $x^0_{A^-}$,
we get  $x^0_{A^+}<x^0_{A^-}$ when $b >\varphi_1(a, \delta)$.
Since $E_r$ is a source when $b<-(a+1)^2$ from Lemma \ref{lm-fi-e},
then  an annular region whose $\omega$-limit set lies in itself  can be constructed when $\varphi_1(a, \delta)<b<-(a+1)^2$.
By Poincar\'{e}-Bendixson Theorem, at least one small limit cycle exists when $\varphi_1(a, \delta)<b<-(a+1)^2$.
It follows from Lemma \ref{one-sl} that
system \eqref{SD} exhibits a unique small limit cycle, which is stable.
The statement {\bf{(b)}} is proven.

Since $x^0_{A^+}>x^0_{A^-}$ when $b<\varphi_1(a, \delta)$,
the outermost small limit cycle must be externally unstable if it exists.
By Lemma \ref{one-sl},
system \eqref{SD} exhibits at most one limit cycle when $b<-(a+1)^2$, and the limit cycle is stable if it exists.
Then, system \eqref{SD} exhibits no small limit cycle when $b<\varphi_1(a, \delta)$.
The statement {\bf{(c)}} is proven.

Let $\gamma$ be the grazing  loop of system \eqref{ygf}.
Then, it can be regarded as a limit cycle in the strip $x\in(-\infty, a+1)$ of system  \eqref{F-g-x}.
Thus,
\[
     \int_{\Gamma}{\rm div} (y,-(x-a-1)-\delta x^2y-\tilde by)\mathrm{d}t<0.
\]
Consider the orbit segment
$\gamma_{\varsigma}$ which starts at  $(\varsigma,0)$
and intersect firstly the positive $x$-axis at $(x^\varsigma_{A^+}, 0)$ and $(x^\varsigma_{A^-}, 0)$) as $t>0$ and
 as $t<0$, where $\varsigma<0$.
The integral of divergency of system \eqref{ygf} along $\gamma_{\varsigma}$
\begin{eqnarray*}
     \int_{\Gamma_\varsigma}{\rm div} (y,-(x-a-1)-\delta x^2y-\tilde by)\mathrm{d}t
     &=&\int_{\Gamma_\varsigma, x<0}{\rm div} (y,-(x-a-1)-\delta x^2y-\tilde by)\mathrm{d}t \\
     & &+\int_{\Gamma_\varsigma,x>0}{\rm div} (y,-(x-a-1)-\delta x^2y-\tilde by)\mathrm{d}t.
\end{eqnarray*}
Since $\int_{\Gamma_\varsigma, x<0}{\rm div} (y,-(x-a-1)-\delta x^2y-\tilde by)\mathrm{d}t \rightarrow 0$ as $\varsigma\rightarrow 0$,
we have
\[
\int_{\Gamma_\varsigma}{\rm div} (y,-(x-a-1)-\delta x^2y-\tilde by)\mathrm{d}t<0
\]
as $|\varsigma|$ is sufficiently small. It follows that $x^\varsigma_{A^+}<x^\varsigma_{A^-}$.

For sufficiently small $\varepsilon>0$, one can get that $x^\varsigma_{A^+}<x^\varsigma_{A^-}$
when $\tilde b=\delta\varphi_1(a, \delta)-\delta\varepsilon$.
By the monotonicity of $x^0_{A^+}$ and $x^0_{A^-}$, we have
$x^0_{A^+}>x^0_{A^-}$ when $\tilde b=\delta\varphi_1(a, \delta)-\delta\varepsilon$.
Then,  an annular region whose $\omega$-limit set lies in itself  can be constructed when $\tilde b=\delta\varphi_1(a, \delta)-\delta\varepsilon$.
By Poincar\'{e}-Bendixson Theorem,
system \eqref{ygf} exhibits a crossing limit cycle which intersects the small interval $(\varsigma,0)$ when $\tilde b=\delta\varphi_1(a, \delta)-\delta\varepsilon$.
Moreover, the innermost one in the annular region is internally stable.
If it is internally stable and externally unstable, then  a stable crossing limit cycle will be bifurcated from it when $\tilde b$ become smaller slightly
because vector field of  system \eqref{ygf} is rotated about $\tilde b$.
The statement {\bf{(d)}} is proven.

Furthermore, similar to the proof  of \cite[Proposition 3.3]{CTW}, one can prove that
$\varphi_1(a, \delta)$ is decreasing and $\mathcal{C}^{\infty}$ function with respect to $a>1$.
\end{proof}

\begin{pro}
    Assume that $a>1$ and $a-1$ is sufficiently small.
	For any $\delta>0$, there exists a $\mathcal{C}^{0}$ function  $\varrho_1(a, \delta)$
    such that
     $\varphi_1(a, \delta)<\varrho_1(a, \delta)<-(a+1)^2$
    and
    \begin{description}
		\item{\bf (a)} system \eqref{SD} exhibits a unique crossing limit cycle when $b=\varrho_1(a,\delta)$, which is internally unstable and externally stable;

        \item{\bf (b)}  system \eqref{SD} exhibits exactly two crossing limit cycles when $\varphi_1(a, \delta)\le b<\varrho_1(a,\delta)$, where the inner one is unstable and the outer one is stable.		

        \item{\bf (c)}  system \eqref{SD} exhibits no  crossing limit cycles when $b>\varrho_1(a,\delta)$.
	\end{description}
	\label{pro-2}
\end{pro}

\begin{proof}
By the proof of \cite[Proposition 3.1]{CTW},  the function $\varphi_1(a, \delta)$ can be extended continuously to $a=1$
and system \eqref{SD} exhibits a grazing  loop surrounding $E_r$ when $a=1$ and $b=\varphi_1(1, \delta)$.
Moreover, the grazing  loop is externally unstable when  $a=1$ and $b=\varphi_1(1, \delta)$.
Then, there exists a value $\varsigma_1<0$ such that $x^{\varsigma_1}_{A^+}>x^{\varsigma_1}_{A^-}$,
where $(x^{\varsigma_1}_{A^+},0)$ and $(x^{\varsigma_1}_{A^-},0)$ are the points where the orbit of system \eqref{ygf}
passing through $(\varsigma_1,0)$ crosses the positive $x$-axis firstly.
From continuous dependence of the solution on parameters, we have $x^{\varsigma_1}_{A^+}>x^{\varsigma_1}_{A^-}$ when $a=1+\epsilon$
and  $\tilde{b}=\delta\varphi_1(a, \delta)$,
where $\epsilon>0$ is sufficiently small.
Since the grazing  loop of system \eqref{SD} is externally stable when $a>1$ by the proof of Proposition \ref{pro-1},
there exists a value $\varsigma_2\in(\varsigma_1, 0)$ such that $x^{\varsigma_2}_{A^+}<x^{\varsigma_2}_{A^-}$ when $a=1+\epsilon$  and  $\tilde{b}=\delta\varphi_1(a, \delta)$.
By Lemma \ref{lm-infi},  all the orbits are positively bounded.
Then,  two annular regions can be constructed when $a=1+\epsilon$  and  $\tilde{b}=\delta\varphi_1(a, \delta)$,
where the outer one's $\omega$-limit set lies in itself and the inner one's $\alpha$-limit set lies in itself.
By Poincar\'{e}-Bendixson Theorem,
system \eqref{ygf} exhibits at least two crossing limit cycles when $a=1+\epsilon$  and $\tilde{b}=\delta\varphi_1(a, \delta)$.
It follows from the externally stability of the  grazing  loop  and positive boundedness of all the orbits that
the number of crossing limit cycles of system \eqref{ygf} is even if
a limit cycle of multiplicity $k$ is regarded as $k$ distinct limit cycles.
Notice that the vector field of system  \eqref{ygf} is rotated with respect 
Then, two limit cycles will be bifurcated from a limit cycle of  even multiplicity when $\tilde{b}$ slightly varies in one side.
By Lemma \ref{three-lc}, system \eqref{ygf} has at most three crossing limit cycles.
Then, system  \eqref{ygf} exhibits exactly two crossing limit cycles when $a=1+\epsilon$  and $\tilde{b}=\delta\varphi_1(a, \delta)$.
Moreover, the outer one is stable and the inner one is unstable.

Denote the intersection point of these two crossing limit cycles and the negative $x$-axis by $(c_1,0)$ and $(c_2,0)$, where $c_1<c_2<0$.
Since $x_{A^+}^c$ and $x_{A^-}^c$ both depend continuously on $c$,
the maximum value of $x_{A^+}^c-x_{A^-}^c$ over the interval $[c_1,c_2]$ exists.
Both $x_{A^+}^c$ and $x_{A^-}^c$ are functions of $\tilde{b}$ and $\delta$.
Here we consider how the  maximum value depends on  $\tilde{b}$, so denote the  maximum value of $x_{A^+}^c-x_{A^-}^c$ over the interval $[c_1,c_2]$ by $M(\tilde{b})$. For any $\eta>0$,
   \begin{eqnarray*}
		M(\tilde{b}-\eta)-M(\tilde{b})=M(\tilde{b}-\eta)-(x_{A^+}^{c_*}-x_{A^-}^{c_*})|_{\tilde b}
		>M(\tilde{b}-\eta)-(x_{A^+}^{c_*}-x_{A^-}^{c_*})|_{\tilde{b}-\eta}
		\ge0,
	\end{eqnarray*}
where $c_*$ is the point such that $x_{A^+}^{c_*}-x_{A^-}^{c_*}$ takes the maximum value on $[c_1,c_2]$.
	Then, $M(\tilde{b})$ decreases continuously as $\tilde{b}$ increases.

Since the  crossing limit cycle for system \eqref{ygf}  passing through $(c_1,0)$ is stable
and the  crossing limit cycle for system \eqref{ygf}  passing through $(c_2,0)$ is unstable
when $a=1+\epsilon$  and $\tilde{b}=\delta\varphi_1(a, \delta)$,
we have
$x_{A^+}^{c}-x_{A^-}^{c}>0$ for any $c\in (c_1,c_2)$  when $a=1+\epsilon$  and $\tilde{b}=\delta\varphi_1(a, \delta)$.
Then, $M(\delta\varphi_1(a, \delta))>0$ when $a=1+\epsilon$.
There are no limit cycles for system \eqref{ygf} by Lemma \ref{lm-b0H} when $\tilde{b}=-\delta (a+1)^2$,
and  all the orbits are positively bounded by Lemma \ref{lm-infi}.
Then,  $x_{A^+}^{c}-x_{A^-}^{c}<0$ for any  $c\in (c_1,c_2)$ when  $a=1+\epsilon$  and  $\tilde{b}=-\delta (a+1)^2$,
implying $M(-\delta (a+1)^2)<0$.
It follows from the monotonicity of $M(\tilde{b})$ that
there exists a unique
$\tilde{b}=\delta\varrho_1(a, \delta)\in( \delta\varphi_1(a, \delta), -\delta (a+1)^2)$
such that $M(\delta\varrho_1(a, \delta))=0$.
Moreover,
    $x_{A^+}^{c}-x_{A^-}^{c}\le0$  for any $c\in [c_1,c_2]$
and there exists a value of $c_*\in (c_1,c_2)$ such that
	$x_{A^+}^{c_*}-x_{A^-}^{c_*}=0$ when $a=1+\epsilon$  and $\tilde{b}=\delta\varrho_1(a, \delta)$.
That implies there is a crossing limit cycle for system \eqref{ygf} when $a=1+\epsilon$  and $\tilde{b}=\delta\varrho_1(a, \delta)$,
which is  internally unstable and externally stable.
Since $x_{A^+}^{0}-x_{A^-}^{0}<0$ when $\tilde b>\delta\varphi_1(a, \delta)$ and all the orbits are positively bounded,
then
the number of crossing limit cycles of system \eqref{ygf} is even if
a limit cycle of multiplicity $k$ is regarded as $k$ distinct limit cycles.
The semi-stable crossing limit cycle must be unique because system \eqref{ygf} has at most three crossing limit cycles by Lemma \ref{three-lc}.
The statement {\bf (a)} holds.

By the statement {\bf (a)} and the monotonicity of $x_{A^+}^{c}$ and $x_{A^-}^c$,
we have
$x_{A^+}^{c_*}-x_{A^-}^{c_*}>0$
when $\tilde b<\delta \varrho_1(a,\delta)$.
Proposition \ref{pro-1} tells us  system \eqref{SD} exhibits  a unique small limit cycle when $\varphi_1(a, \delta)<b<-(a+1)^2$,
which is stable.
Then, $x_{A^+}^{0}-x_{A^-}^{0}<0$ when $\delta\varphi_1(a, \delta)<\tilde b<\delta\varrho_1(a,\delta)$.
Notice that the grazing limit cycle is stable when $\tilde b=\delta\varphi_1(a, \delta)$.
Combining with that all the orbits are positively bounded,
 two annular regions can be constructed when $\delta\varphi_1(a, \delta)\le\tilde b<\delta\varrho_1(a,\delta)$,
where the outer one's $\omega$-limit set lies in itself and the inner one's $\alpha$-limit set lies in itself.
Since the number of crossing limit cycles of system \eqref{SD} is even if
a limit cycle of multiplicity $k$ is regarded as $k$ distinct limit cycles when  $b\ge\varphi_1(a, \delta)$,
and system \eqref{SD} has at most three crossing limit cycles,
system \eqref{SD} exhibits exactly two crossing limit cycles when  $\varphi_1(a, \delta)\le b<\varrho_1(a,\delta)$, where the inner one is unstable and the outer one is stable.
The statement {\bf (b)} holds.

From the stability of two crossing limit cycles of system  \eqref{ygf} with $a=1+\epsilon$  and $\tilde{b}=\delta\varphi_1(a, \delta)$,
one gets $x_{A^+}^{c}-x_{A^-}^{c}<0$  for any $c\in (-\infty, c_1)\cup (c_2,0)$.
Notice that $x_{A^+}^{c}-x_{A^-}^{c}\le0$  for any $c\in [c_1,c_2]$  when $a=1+\epsilon$  and $\tilde{b}=\delta\varrho_1(a, \delta)$.
By the monotonicity of $x_{A^+}^{c}$ and $x_{A^-}^c$,
we know $x_{A^+}^{c}-x_{A^-}^{c}<0$  for any $c\in (-\infty, 0)$ when  $a=1+\epsilon$ and $\tilde{b}>\delta\varrho_1(a, \delta)$.
That means  system \eqref{SD} exhibits no crossing limit cycles when  $a=1+\epsilon$ and ${b}>\varrho_1(a, \delta)$.
The statement {\bf (c)} holds.
\end{proof}

\begin{pro}
	For any $\delta>0$, the function $\varrho_1(a, \delta)$ is well defined for $ 1<a<a_0$.
Moreover,  
$\varrho_1(a, \delta)<-(a+1)^2$.
If $\varrho_1(a, \delta)>\varphi_1(a, \delta)$, then the statements {\bf (a)}-{\bf (c)} in Proposition \ref{pro-2} hold;
if $\varrho_1(a, \delta)<\varphi_1(a, \delta)$, then
    \begin{description}
		\item{\bf (d)} system \eqref{SD} exhibits exactly two crossing limit cycles when $b=\varrho_1(a,\delta)$,
               where the outer one is an internally unstable and externally stable crossing limit cycle, and the inner one is a stable limit cycle;
		
         \item{\bf (e)}  system \eqref{SD} exhibits exactly three crossing limit cycles when $b=\varrho_1(a,\delta)-\varepsilon$, which are stable, unstable and stable
               from the innermost to the outermost;

        \item{\bf (f)}  system \eqref{SD} exhibits a unique crossing limit cycle when $\varrho_1(a,\delta)<b<\varphi_1(a, \delta)$, which is stable;

	\end{description}
if $\varrho_1(a, \delta)=\varphi_1(a, \delta)$, then
    \begin{description}
		\item{\bf (g)} system \eqref{SD} exhibits a unique crossing limit cycle when $b=\varrho_1(a,\delta)$,
               which is internally unstable and externally stable;

        \item{\bf (h)}  system \eqref{SD} exhibits exactly three crossing limit cycles when $b=\varrho_1(a,\delta)-\varepsilon$, which are stable, unstable and stable
               from the innermost to the outermost,		

        \item{\bf (i)}  system \eqref{SD} exhibits no crossing limit cycles when $b>\varrho_1(a,\delta)$;

	\end{description}
    where $\varepsilon>0$ is sufficiently small.  
	\label{pro-3}
\end{pro}

\begin{proof}

From the proof of Proposition \ref{pro-2}, we know $\varrho_1(a, \delta)$ can be well defined if system \eqref{SD} exhibits two crossing simple limit cycles for some
values of $b$,
where the outer one is stable and the inner one is unstable.
From Lemma \ref{three-lc}, there exists a unique open interval of $b$
such that system \eqref{SD} exhibits exactly three simple crossing limit cycles whenever $b$ lies in this interval when $1<a<a_0$.
All the orbits of system \eqref{SD} are positively bounded by Lemma \ref{lm-infi}.
Then, the three crossing limit cycles must be stable, unstable and stable from the innermost to the outermost if they exist.
That means the function $\varrho_1(a, \delta)$ is well defined for $ 1<a<a_0$.
Moreover, system \eqref{SD} exhibits an internally unstable and externally stable crossing limit cycle when $b=\varrho_1(a,\delta)$.



If $\varrho_1(a, \delta)>\varphi_1(a, \delta)$, then system \eqref{ygf} has a unique small limit cycle when $\tilde b=\delta\varrho_1(a, \delta)$,  which is stable.
Notice that  the number of crossing
limit cycles of system \eqref{SD} is even if a limit cycle of multiplicity $k$ is regarded as $k$ distinct limit
cycles when  $b>\varphi_1(a, \delta)$.
Similarly to the proof of Proposition \ref{pro-2}, we can get that the statements {\bf (a)}-{\bf (c)} hold.

If $\varrho_1(a, \delta)<\varphi_1(a, \delta)$, then system \eqref{ygf} has no small limit cycles when $\tilde b=\delta\varrho_1(a, \delta)$,
implying that  $x^{0}_{A^+}>x^{0}_{A^-}$ when $\tilde b=\delta\varrho_1(a, \delta)$.
Since all the orbits are positively bounded by Lemma \ref{lm-infi},
the number of crossing
limit cycles of system \eqref{ygf} is odd if a limit cycle of multiplicity $k$ is regarded as $k$ distinct limit
cycles when $\tilde b=\delta\varrho_1(a, \delta)$.
From Lemma \ref{three-lc}, system \eqref{SD} has at most three crossing limit cycles.
Thus, except the internally unstable and externally stable crossing limit cycle, system \eqref{SD} also has a stable crossing limit cycle which lies in the region enclosed by the
internally unstable and externally stable crossing limit cycle. The statement {\bf (d)} holds.

We still let the intersection point of internally unstable and externally stable crossing limit cycle of system \eqref{ygf} with $\tilde b=\delta\varrho_1(a,\delta)$ and the positive $x$-axis be $(c_*,0)$. Then
$x^{c_*}_{A^+}= x^{c_*}_{A^-}$ and  $x^{c_3}_{A^+}= x^{c_3}_{A^-}$, where $c_*<c_3<0$.
Moreover,  there exist $d_1\in (-\infty, c_*)$ and $d_2\in (c_*, c_3)$
such that $x^{d_1}_{A^+}< x^{d_1}_{A^-}$ and  $x^{d_2}_{A^+}< x^{d_2}_{A^-}$.
From the continuous dependence of the solution on parameters and the monotonicity of $x_{A^+}^{c}$ and $x_{A^-}^c$,
we have
 $x^{d_1}_{A^+}< x^{d_1}_{A^-}$, $x^{c_*}_{A^+}> x^{c_*}_{A^-}$, $x^{d_2}_{A^+}< x^{d_2}_{A^-}$ and $x^{c_3}_{A^+}> x^{c_3}_{A^-}$
when  $\tilde{b}=\delta\varrho_1(a, \delta)-\delta\varepsilon$.
Then,  three annular regions can be constructed.
By Poincar\'{e}-Bendixson Theorem,
system \eqref{ygf} exhibits at least three crossing limit cycles.
It follows from Lemma \ref{three-lc} that system \eqref{SD} exhibits exactly three crossing limit cycles when $b=\varrho_1(a,\delta)-\varepsilon$
and they are simple.
Since all the orbits are positively bounded, these three limit cycles are stable, unstable and stable
from the innermost to the outermost.
The statement {\bf (e)} holds.

Denote the stable crossing limit cycle of system \eqref{ygf} with $\tilde b=\delta\varrho_1(a,\delta)$ intersects the negative $x$-axis at $(c_3,0)$.
Then,  $x^{c_3}_{A^+}\le x^{c_3}_{A^-}=0$ when  $\tilde b=\delta\varrho_1(a,\delta)$.
By the monotonicity of $x^{c}_{A^+}$ and $x^{c}_{A^-}$, we know $x^{c_3}_{A^+}<x^{c_3}_{A^-}$
when  $\tilde b>\delta\varrho_1(a, \delta)$.
It follows from  $x^{0}_{A^+}>x^{0}_{A^-}$ when $\tilde b<\delta\varphi_1(a, \delta)$
that  an annular region whose $\omega$-limit set lies in itself  can be constructed when $\delta\varrho_1(a, \delta)<\tilde b<\delta\varphi_1(a, \delta)$.
By Poincar\'{e}-Bendixson Theorem,
system \eqref{ygf} exhibits at least one crossing limit cycle.
From Lemma \ref{three-lc}, all the values of $b$ in which system \eqref{SD} exhibits exactly three crossing limit cycles form an open interval and all these three crossing
limit cycles are simple.
There are no new crossing limit cycles appearing when $\varrho_1(a, \delta)< b<\varphi_1(a, \delta)$
because system \eqref{SD} exhibits three  crossing limit cycles when $b=\varrho_1(a,\delta)-\varepsilon$ and a semi-stable limit cycle when $b=\varrho_1(a, \delta)$.
The statement {\bf (f)} holds.

If $\varrho_1(a, \delta)=\varphi_1(a, \delta)$, then $x^{0}_{A^+}=x^{0}_{A^-}$ when $\tilde b=\delta\varrho_1(a, \delta)$.
By the  monotonicity of $x^{0}_{A^+}$ and $x^{0}_{A^-}$, we get $x^{0}_{A^+}< x^{0}_{A^-}$  when $\tilde b>\delta\varrho_1(a, \delta)$.
Notice that the grazing  loop  of system \eqref{SD} is stable.
Then,  the number of crossing
limit cycles of system \eqref{ygf} is even if a limit cycle of multiplicity $k$ is regarded as $k$ distinct limit
cycles when  $\tilde b\ge\delta\varrho_1(a, \delta)$.
Thus, the internally unstable and externally stable crossing limit cycle is the unique crossing limit cycle when $b=\varrho_1(a,\delta)$.
The statement {\bf (g)} holds.

Moreover, by Proposition \ref{pro-1}, for any sufficiently small $|\varsigma|$ with $\varsigma<0$ there exists a stable crossing limit cycle of system \eqref{SD} which
intersects the small interval $(\varsigma,0)$  when $b=\varrho_1(a,\delta)-\varepsilon$.
It is similar to  the proof of Proposition \ref{pro-2}, there are two crossing limit cycles bifurcated
from the internally unstable and externally stable crossing limit cycle when $b=\varrho_1(a,\delta)-\varepsilon$,
where the inner one is unstable and the outer one is stable.
Therefore, system \eqref{SD} exhibits three crossing limit cycles when $b=\varrho_1(a,\delta)-\varepsilon$, which are stable, unstable and stable
 from the innermost to the outermost.
 The statement {\bf (h)} holds.

Furthermore,  $x^{c}_{A^+}\le x^{c}_{A^-}$ when  $\tilde b=\delta\varrho_1(a,\delta)$ for all $c\in(-\infty,0)$.
It follows from  monotonicity of $x^{c}_{A^+}$ and $x^{c}_{A^-}$ that $x^{c}_{A^+}< x^{c}_{A^-}$ when  $\tilde b>\delta\varrho_1(a,\delta)$ for all $c\in(-\infty,0)$,
implying that  system \eqref{SD} exhibits no crossing limit cycles when $b>\varrho_1(a,\delta)$.
The statement {\bf (i)} holds.

\end{proof}

\begin{pro}
For any $\delta>0$, there exists a $\mathcal{C}^{0}$ function  $\varrho_2(a, \delta)$  for $ 1<a<a_0$
    such that
     $-3(a+1)^2<\varrho_2(a, \delta)<\min\{\varphi_1(a, \delta), \varrho_1(a, \delta)\}$
    and
    \begin{description}
		\item{\bf (a)} system \eqref{SD} exhibits exactly two crossing limit cycles when $b=\varrho_2(a,\delta)$, where the outer one
           is stable and the inner one is internally stable and externally unstable;

        \item{\bf (b)}  system \eqref{SD} exhibits exactly three crossing limit cycles when $\varrho_2(a,\delta)<b<\min\{\varphi_1(a, \delta), \varrho_1(a, \delta)\}$, which are stable, unstable and stable
        from the innermost to the outermost;

        \item{\bf (c)} system \eqref{SD} exhibits a unique crossing limit cycle when $b<\varrho_2(a,\delta)$, which is stable.

	\end{description}
	\label{pro-4}
\end{pro}

\begin{proof}
From Propositions \ref{pro-1}-\ref{pro-3}, system \eqref{SD} exhibits exactly three crossing limit cycles when $b=\min\{\varphi_1(a, \delta), \varrho_1(a, \delta)\}-\varepsilon$,
which are stable, unstable and stable from the innermost to the outermost.
Denote the intersection point of these three crossing limit cycles and the negative $x$-axis by $(c_1,0)$, $(c_2,0)$ and $(c_3,0)$, where $c_1<c_2<c_3<0$.
Notice that $x_{A^+}^c-x_{A^-}^c$ could takes on its minimum value on the interval $[c_2,c_3]$, denoted by $m(\tilde{b})$,
because $x_{A^+}^c$ and $x_{A^-}^c$ both depend continuously on $c$.
For any $\eta>0$,
   \begin{eqnarray*}
		m(\tilde{b}+\eta)-m(\tilde{b})=m(\tilde{b}+\eta)-(x_{A^+}^{c_o}-x_{A^-}^{c_o})|_{\tilde b}
		<m(\tilde{b}+\eta)-(x_{A^+}^{c_o}-x_{A^-}^{c_o})|_{\tilde{b}+\eta}
		\le0,
	\end{eqnarray*}
where $c_o$ is the point such that $x_{A^+}^{c_o}-x_{A^-}^{c_o}$ takes the minimum value on $[c_2,c_3]$.
	Then, $m(\tilde{b})$ decreases continuously as $\tilde{b}$ increases.

The  crossing limit cycle for system \eqref{ygf}  passing through $(c_2,0)$ is unstable
and the  crossing limit cycle for system \eqref{ygf}  passing through $(c_3,0)$ is stable
when $\tilde{b}=\delta\min\{\varphi_1(a, \delta), \varrho_1(a, \delta)\}-\delta\varepsilon$.
Then
$x_{A^+}^{c}-x_{A^-}^{c}<0$ for any $c\in (c_2,c_3)$ when $\tilde{b}=\delta\min\{\varphi_1(a, \delta), \varrho_1(a, \delta)\}-\delta\varepsilon$,
implying that
$m(\delta\min\{\varphi_1(a, \delta), \varrho_1(a, \delta)\}-\delta\varepsilon)<0$.

By Lemma \ref{three-lc}, the number of crossing limit cycles of system  \eqref{SD} is less than three when $b$ is not in the open interval $[c,d]$.
Then, there is a value of $b^*<\min\{\varphi_1(a, \delta), \varrho_1(a, \delta)\}-\varepsilon$
such that system \eqref{SD} exhibits at most two crossing limit  cycle when $b=b^*$.
Since  $x_{A^+}^{c}$ increases and $x_{A^-}^c$ decreases as $\tilde b$ decreases by Lemma \ref{lm-mono},
we know
$x_{A^+}^{c_1}-x_{A^-}^{c_1}>0$  when  $\tilde b= \delta b^*$.
It follows from the positive boundedness of all the orbits that system  \eqref{ygf} exhibits at least one crossing limit cycle which passes through the negative $x$-axis
at a point on the interval $(-\infty, c_1)$ when  $\tilde b<\min\{\varphi_1(a, \delta), \varrho_1(a, \delta)\}-\varepsilon$.
Thus, the number of  crossing limit cycles of system \eqref{ygf} intersecting with
 the negative $x$-axis
at a point on the interval $(c_2, c_3)$  is zero or one when $\tilde b= \delta b^*$.
Since both $x_{A^+}^{c_2}-x_{A^-}^{c_2}>0$  and $x_{A^+}^{c_3}-x_{A^-}^{c_3}>0$  when  $\tilde b= \delta b^*$,
we get $x_{A^+}^{c}-x_{A^-}^{c}\ge0$  for  all $c\in (c_2,c_3)$,
implying $m(\delta b^*)\ge0$.

Due to the monotonicity of $M(\tilde{b})$,
there exists a unique
$\tilde{b}=\delta\varrho_2(a, \delta)\in( \delta\min\{\varphi_1(a, \delta), \varrho_1(a, \delta)\}-\delta\varepsilon, \delta b^*]$
such that $m(\delta\varrho_1(a, \delta))=0$.
Moreover,
    $x_{A^+}^{c}-x_{A^-}^{c}\ge0$  for any $c\in [c_2,c_3]$
and there exists a $c_o\in (c_2,c_3)$ such that
	$x_{A^+}^{c_o}-x_{A^-}^{c_o}=0$ when $\tilde{b}=\delta\varrho_2(a, \delta)$.
That implies there is a crossing limit cycle for system \eqref{ygf} when $\tilde{b}=\delta\varrho_2(a, \delta)$,
which is  internally stable and externally unstable.

Because  $x_{A^+}^{0}-x_{A^-}^{0}<0$ when $\tilde b<\delta\varphi_1(a, \delta)$ and all the orbits are positively bounded,
the number of crossing limit cycles of system \eqref{ygf} is odd if
a limit cycle of multiplicity $k$ is regarded as $k$ distinct limit cycles.
Combining that  there is at least one crossing limit cycle passing through the interval $(-\infty, c_1)$ on the negative $x$-axis
when $\tilde b<\min\{\varphi_1(a, \delta), \varrho_1(a, \delta)\}-\varepsilon$ and the number of crossing limit cycles is less than or equal to three,
system \eqref{SD} exhibits exactly two crossing limit cycles when $b=\varrho_2(a,\delta)$, where the outer one is
           is stable and the inner one is internally stable and externally unstable.

From Lemma \ref{lm-div-a>1}, there are no internally stable and externally unstable crossing limit cycles of system \eqref{SD} intersecting with $x=-\sqrt{-b}$,
implying that $c_0>-\sqrt{-b}$.
Lemma \ref{lm-ag1clu} tells us the crossing limit cycle in the strip $x\ge -\sqrt{-b}$ is stable when $b\le -3(a+1)^2$.
So we have $\varrho_2(a,\delta)>-3(a + 1)^2$.
The statement {\bf (a)} holds.

By the statement {\bf (a)},
the internally stable and externally unstable crossing limit cycle of system  \eqref{ygf} passes through the point $(c_o,0)$
and the stable  crossing limit cycle intersects the negative $x$-axis at a point denoted by $(\bar c_1,0)$ when  $\tilde{b}=\delta\varrho_2(a, \delta)$,
where $\bar c_1<c_0$.
Then, $x_{A^+}^{\bar c_1}-x_{A^-}^{\bar c_1}=0$ and $x_{A^+}^{c_o}-x_{A^-}^{c_o}=0$,
and
there exist two values of $\sigma_1$ and  $\sigma_2$,
where $\bar c_1<\sigma_1<c_0<\sigma_2<0$,
such that
and
$x_{A^+}^{\sigma_1}-x_{A^-}^{\sigma_1}>0$
$x_{A^+}^{\sigma_2}-x_{A^-}^{\sigma_2}>0$
when $\tilde{b}=\delta\varrho_2(a, \delta)$.
From the continuous dependence of the solution on parameters and the  monotonicity of $x_{A^+}^{c}$ and $x_{A^-}^c$,
one can get that
$x_{A^+}^{\bar c_1}-x_{A^-}^{\bar c_1}<0$,
$x_{A^+}^{\sigma_1}-x_{A^-}^{\sigma_1}>0$,
$x_{A^+}^{c_o}-x_{A^-}^{c_o}<0$,
$x_{A^+}^{\sigma_2}-x_{A^-}^{\sigma_2}>0$  when $\tilde{b}=\delta\varrho_2(a, \delta)+\delta\varepsilon$
where  $\varepsilon>0$ is  sufficiently small.
Then,  three annular regions can be constructed,
where one's $\omega$-limit set, $\alpha$-limit set, $\omega$-limit set (from the innermost to the outermost)
lies in itself.
Since the maximum value of crossing limit cycles of system  \eqref{SD} is three,
system \eqref{SD} exhibits exactly three crossing limit cycles
when ${b}=\varrho_2(a, \delta)+\varepsilon$, which are stable, unstable and stable
        from the innermost to the outermost;

In case $\varphi_1(a, \delta)\ge \varrho_1(a, \delta)$,
according to the result in Proposition \ref{pro-3},
system \eqref{SD} exhibits exactly three crossing limit cycles
when ${b}=\min\{\varphi_1(a, \delta), \varrho_1(a, \delta)\}-\varepsilon$.
In the case  $\varphi_1(a, \delta)< \varrho_1(a, \delta)$,
system \eqref{SD} exhibits exactly two crossing limit cycles
when $\varphi_1(a, \delta)<{b}<\varrho_1(a, \delta)$ by  Proposition \ref{pro-3}
and there is another crossing limit cycle of system \eqref{SD} which intersects the
negative $x$-axis at a point very close to the origin by  Proposition \ref{pro-1}.
Then, in the case  $\varphi_1(a, \delta)< \varrho_1(a, \delta)$ system \eqref{SD} also exhibits exactly three crossing limit cycles
when ${b}=\min\{\varphi_1(a, \delta), \varrho_1(a, \delta)\}-\varepsilon$.

From Lemma \ref{three-lc}, all the values of $b$ for which three crossing limit cycles of system  \eqref{SD} exist
form an open interval $[c,d]$. Then
system \eqref{SD} exhibits exactly three crossing limit cycles and all the crossing  limit cycles are simple
when $\varrho_2(a,\delta)<b<\min\{\varphi_1(a, \delta), \varrho_1(a, \delta)\}$.
From the positive boundedness of all the orbits of system \eqref{SD},
the limit cycles are stable, unstable and stable from the innermost to the outermost.
The statement {\bf (b)} holds.

From the statement {\bf (a)}  and Lemma \ref{three-lc}, it is easy to see that  $b= \varrho_1(a, \delta)$
is not in the interval $[c,d]$.
Then, the number  of crossing limit cycle of system \eqref{SD} is less than three when $b< \varrho_1(a, \delta)$.
Notice that the
vector field of system \eqref{SD} is rotated about $b$,
implying that two limit cycles can be bifurcated
from any semi-stable limit cycle if $b$ slightly  changes to one side.
The number of crossing limit cycles of system \eqref{SD} is even if
a limit cycle of multiplicity $k$ is regarded as $k$ distinct limit cycles when $b<\varphi_1(a, \delta)$.
Since
$\varrho_2(a,\delta)< \varphi_1(a, \delta)$,
system \eqref{SD} exhibits a unique crossing limit cycle when $b<\varrho_2(a,\delta)$, which is stable.
The statement {\bf (c)} holds.
\end{proof}

\section{Proofs  of Theorems  \ref{mr1} and \ref{mr2} }

In this section, we prove Theorems \ref{mr1} and \ref{mr2} .

\noindent$\displaystyle {\bf Proof ~of~ Theorem ~\ref{mr1}}$
The statement {\bf (a)} can be obtained from Proposition \ref{hopfbif}. The statement {\bf (b)} is directly from Proposition \ref{pro-1}.
The statement {\bf (c)} is the result of Propositions \ref{pro-3} and \ref{pro-4}.
$\hfill{} \Box$

\noindent$\displaystyle {\bf Proof ~of~ Theorem ~\ref{mr2}}$
By Lemma \ref{lm-fi-e}, system  \eqref{SD-x} exhibits a unique equilibrium  $E_r$, which is a source when
$b<-(a+1)^2$   and a sink when $b\geq-(a+1)^2$.
The dynamic behavior of system \eqref{SD} near infinity is given in Lemma \ref{lm-infi}.
The infinite equilibria $I_x^+$, $I_x^-$ are unstable star nodes, and $I_y^+$ and $I_y^-$ are degenerate saddles.

Consider $(a,b, \delta)\in {\cal I} \cup H$.
From Lemma \ref{lm-fi-e}, $E_r$ is a sink.
By Lemmas \ref{lm-bg0} and \ref{lm-b0H},
system \eqref{SD} exhibits no closed orbits.
Then, $E_r$ is the $\omega$-limit set of all the orbits.

Consider $(a,b, \delta)\in {\cal II}$.
From Proposition \ref{pro-1} {\bf(b)}, system \eqref{SD} exhibits  a unique small limit cycle when $\varphi_1(a, \delta)<b<-(a+1)^2$,
which is stable.
For $1<a<a_0$ and $\varrho_1(a,\delta)>\varphi_1(a,\delta)$,  Proposition \ref{pro-3} {\bf(c)} tells us that system \eqref{SD} exhibits no  crossing limit cycles when $b>\varrho_1(a,\delta)$.
For $1<a<a_0$ and $\varrho_1(a,\delta)\le\varphi_1(a,\delta)$, system \eqref{SD} exhibits exactly three crossing limit cycles when $b=\varrho_1(a,\delta)-\varepsilon$
and a semi-stable limit cycle when $b=\varrho_1(a,\delta)$ by  Proposition \ref{pro-3} {\bf(e)} and {\bf(d)}.
Since all the values of $b$ in which system \eqref{SD} exhibits exactly three crossing limit cycles form an open interval and all these three crossing
limit cycles are simple by Lemma \ref{three-lc},  system \eqref{SD} exhibits no  crossing limit cycles when $b>\varphi_1(a,\delta)$.
For $a>a_0$,  we know that the number of crossing limit cycles is no more than one by Lemma \ref{three-lc}.
Since  the number of crossing
limit cycle of system \eqref{SD} is even if a limit cycle of multiplicity $k$ is regarded as $k$ distinct limit
cycles when  $b>\varphi_1(a, \delta)$,
system \eqref{SD} exhibits no  crossing limit cycles when $b>\varphi_1(a,\delta)$ and $a>a_0$.
Then, all the orbits approach the stable small limit cycle when $t\rightarrow +\infty$.

Consider $(a,b, \delta)\in DL$.
By Proposition \ref{pro-1} {\bf(a) (b) (c)} and Proposition \ref{pro-3} {\bf(a) (d) (g)},
system \eqref{SD} exhibits two limit cycles when $b=\varrho_1(a,\delta)$.
and  the outer one is an  internally
unstable and externally stable crossing limit cycle.
Moreover, the inner one is a stable small limit cycle when $b=\varrho_1(a,\delta)>\varphi_1(a,\delta)$,
as shown in  $DL_{11}$ of Fig. \ref{global-2},
the inner one is a stable crossing limit cycle when $b=\varrho_1(a,\delta)<\varphi_1(a,\delta)$,
as shown in $DL_{12}$ of Fig. \ref{global-2},
and
the inner one is a stable grazing limit cycle when $b=\varrho_1(a,\delta)=\varphi_1(a,\delta)$,
as shown in  $P$ of Fig. \ref{global-2}.

Consider $(a,b, \delta)\in \cal{III}$.
By Proposition \ref{pro-1} {\bf(b)} and Proposition \ref{pro-3} {\bf(b)}, system \eqref{SD} exhibits two crossing limit cycles
and a stable small limit cycle.
Moreover, the outer crossing limit cycle is stable and the inner crossing limit cycle is unstable.

Consider $(a,b, \delta)\in G$. From Proposition \ref{pro-1} {\bf(a)},
system \eqref{SD} exhibits a grazing limit cycle when $b=\varphi_1(a,\delta)$.
Moreover, by Proposition \ref{pro-3} {\bf(b)},
there are exactly two crossing limit cycles when $b=\varphi_1(a,\delta)<\varrho_1(a,\delta)$,
where the outer one is stable and the inner one  is unstable,
as in $G_{1}$ of Fig. \ref{global-2}.
The number of crossing
limit cycle of system \eqref{SD} is even if a limit cycle of multiplicity $k$ is regarded as $k$ distinct limit
cycles when  $b=\varphi_1(a, \delta)$.
If there are two simple crossing limit cycles when $b=\varphi_1(a,\delta)>\varrho_1(a,\delta)$,
then at least two crossing limit cycles exist when $b=\varphi_1(a,\delta)-\varepsilon$,
where $\varepsilon>0$ is sufficiently small, which is contradictory with   Proposition \ref{pro-3} {\bf(f)}.
If there is a semi-stable crossing  limit cycle when $b=\varphi_1(a,\delta)>\varrho_1(a,\delta)$,
then it must be  internally
unstable and externally stable because all the orbits are positively bounded.
Thus, it is similar to the bifurcation of a semi-stable crossing limit cycle when $b=\varrho_1(a,\delta)$
 that two crossing limit cycles will be bifurcated from this semi-stable crossing  limit cycle when $b$ changes into $\varphi_1(a,\delta)-\varepsilon$,
which is also contradictory with   Proposition \ref{pro-3} {\bf(f)}.
Therefore,
there are no crossing  limit cycles when $b=\varphi_1(a,\delta)>\varrho_1(a,\delta)$,
as in $G_{2}$ of Fig. \ref{global-2}.

Consider $(a,b, \delta)\in \cal{IV}$.  From Proposition \ref{pro-4} {\bf(b)},
system \eqref{SD} exhibits exactly three crossing limit cycles when $\varrho_2(a,\delta)<b<\min\{\varphi_1(a, \delta), \varrho_1(a, \delta)\}$,
which are stable, unstable and stable from the innermost to the outermost.

Consider $(a,b, \delta)\in DL_2$.  From Proposition \ref{pro-4} {\bf(a)},
system \eqref{SD} exhibits exactly two crossing limit cycles when $b=\varrho_2(a,\delta)$, where the outer one is
         stable and the inner one is internally stable and externally unstable.

Consider $(a,b, \delta)\in \cal{V}$.  From Proposition \ref{pro-4} {\bf(c)},
system \eqref{SD} exhibits a unique crossing limit cycle when $b<\varrho_2(a,\delta)$, which is stable.
$\hfill{} \Box$

\section{Numerical examples and discussions}

The phase portraits and bifurcations of  \eqref{SD}
by numerical simulations are shown in this section.

\begin{exmp}
Choose  $a=4$ and $\delta=0.1$.
$E_r$ is a sink and system \eqref{SD} exhibits no limit cycle when $b=-24.9$,
as shown in {\rm Fig. \ref{simu-1} (a)}.
$E_r$ is a source and  system \eqref{SD} exhibits a stable small limit cycle when $b=-25.7$,
as shown in {\rm Fig. \ref{simu-1} (b)}.
$E_r$ is a source and  system \eqref{SD} exhibits a stable small limit cycle which passes through a very
small neighborhood of the origin when  $b=-26.083$,
as shown in {\rm Fig. \ref{simu-1} (c)}.
$E_r$ is a source and  system \eqref{SD} exhibits a stable crossing limit cycle when  $b=-26.1$,
as shown in {\rm Fig. \ref{simu-1} (d)}.
\end{exmp}

\begin{figure}[!htb]
	\subfloat[$b=-24.9$]
	{\includegraphics[scale=0.25]{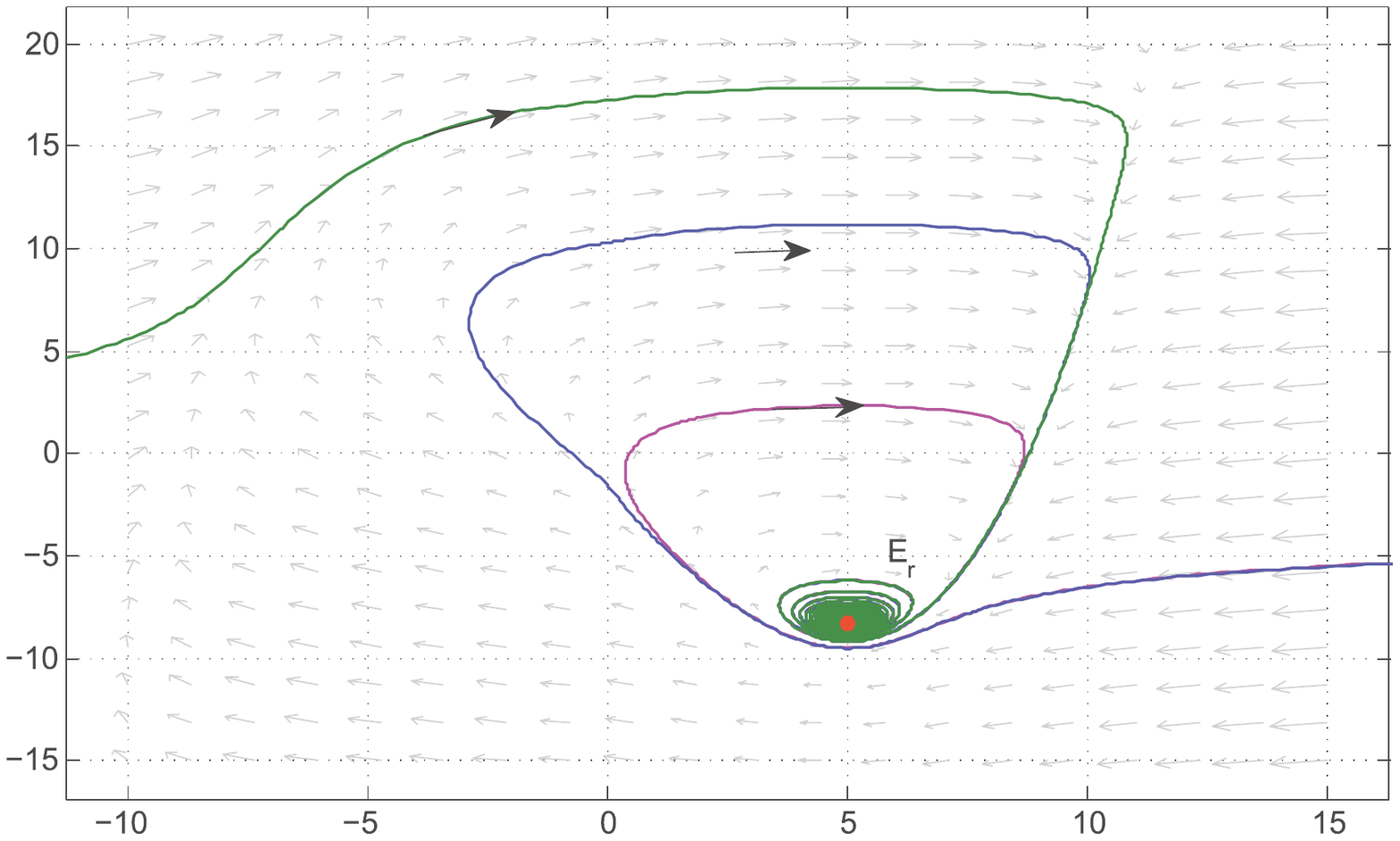}}
     \subfloat[$b=-25.7$]
	{\includegraphics[scale=0.25]{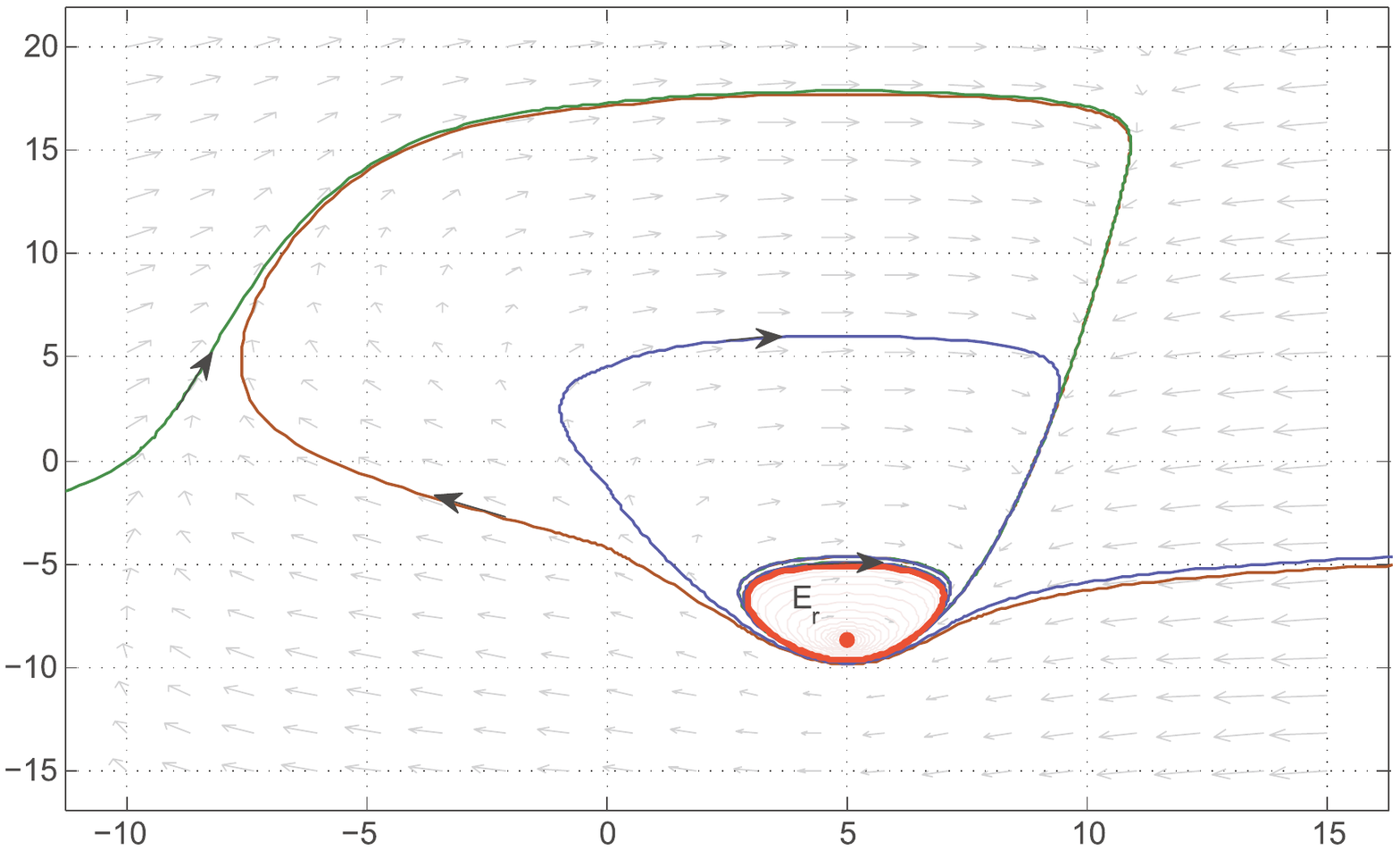}}\\
     \subfloat[$b=-26.083$]
	{\includegraphics[scale=0.25]{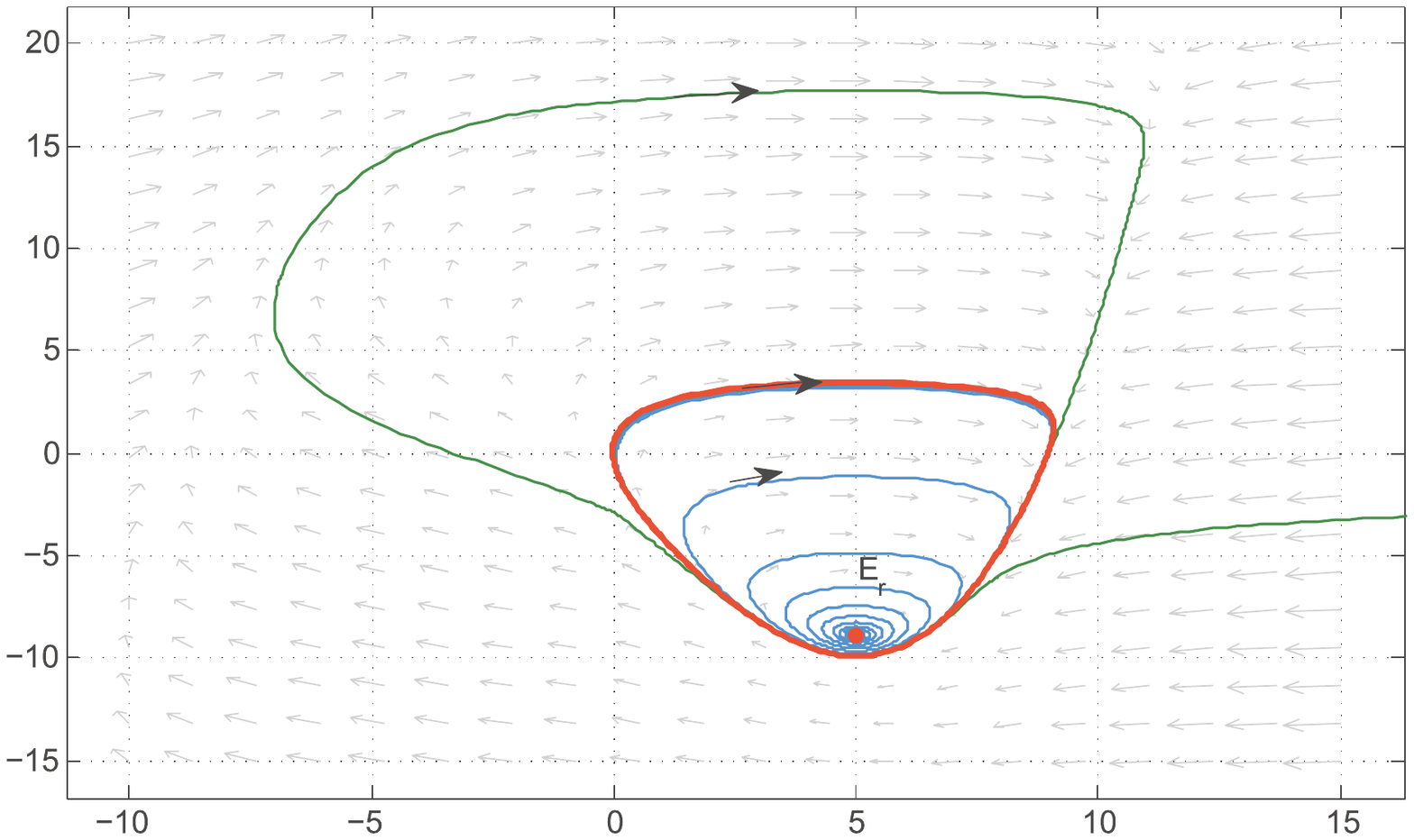}}
	\subfloat[$b=-26.1$]
	{\includegraphics[scale=0.25]{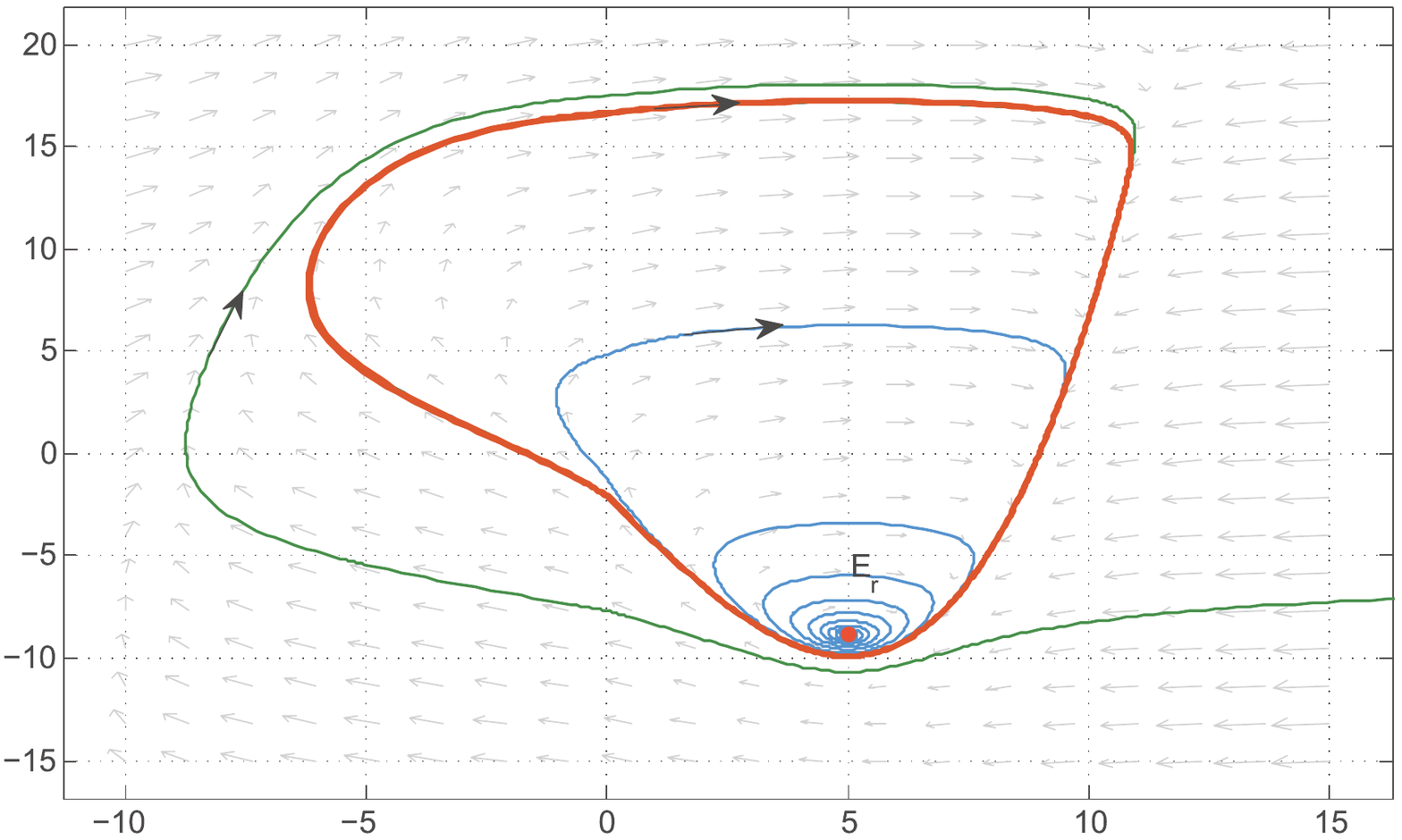}}
	\caption{{\footnotesize Numerical phase portraits with two equilibria when $a=4$, $\delta=0.1$.}}
	\label{simu-1}
\end{figure}

From Example 1, we find $\varphi_1(4,0.1)\thickapprox -26.083$.
The double limit cycle bifurcations $DL_1$ and $DL_2$ do not been observed when $a=4$ and $\delta=0.1$.

\begin{exmp}
Choose  $a=1.2$ and $\delta=0.1$.
System \eqref{SD} exhibits a stable small limit cycle when $b=-5.2$,
as shown in {\rm Fig. \ref{simu-2} (a)}.
Choose  $a=1.2$ and $\delta=0.1$.
System \eqref{SD} exhibits a stable small limit cycle and
a semi-stable (externally stable and internally unstable) crossing limit cycle
when $b=-5.908$,
as shown in {\rm Fig. \ref{simu-2} (b)}.
System \eqref{SD} exhibits a stable small limit cycle which passes through a very
small neighborhood of the origin,
and two crossing limit cycles (the outer is stable and the inner is unstable)
when  $b=-5.93$,
as shown in {\rm Fig. \ref{simu-2} (c)}.
System \eqref{SD} exhibits two crossing limit cycles (the outer is stable and the inner is externally unstable and internally stable) when  $b=-5.9337$,
as shown in {\rm Fig. \ref{simu-2} (d)}.
System \eqref{SD} exhibits a stable crossing limit cycle when  $b=-6$,
as shown in {\rm Fig. \ref{simu-2} (e)}.
\end{exmp}

\begin{figure}[!htb]
	\subfloat[$b=-5.2$]
	{\includegraphics[scale=0.25]{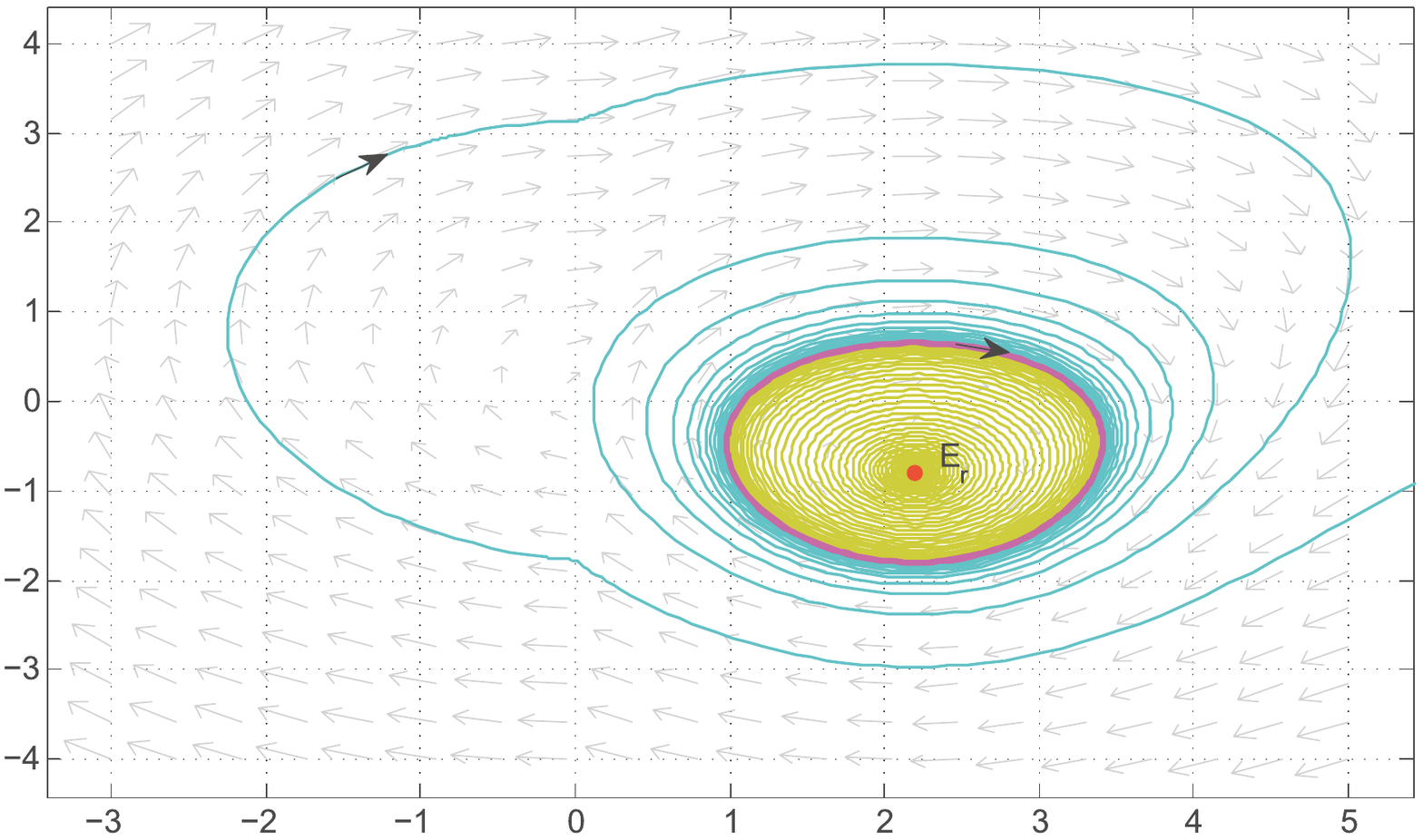}}
     \subfloat[$b=-5.908$]
	{\includegraphics[scale=0.25]{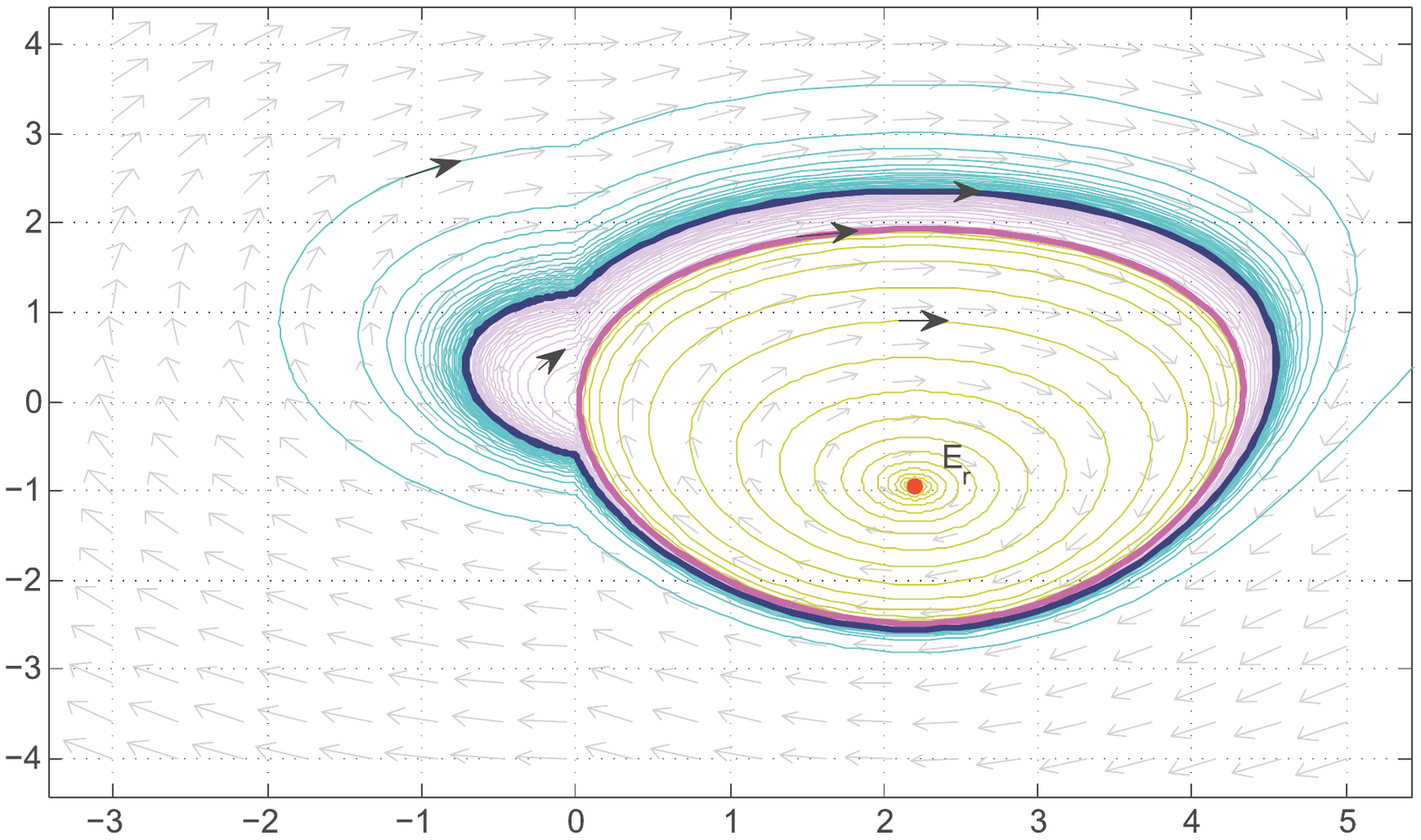}}\\
	\subfloat[$b=-5.93$]
	{\includegraphics[scale=0.25]{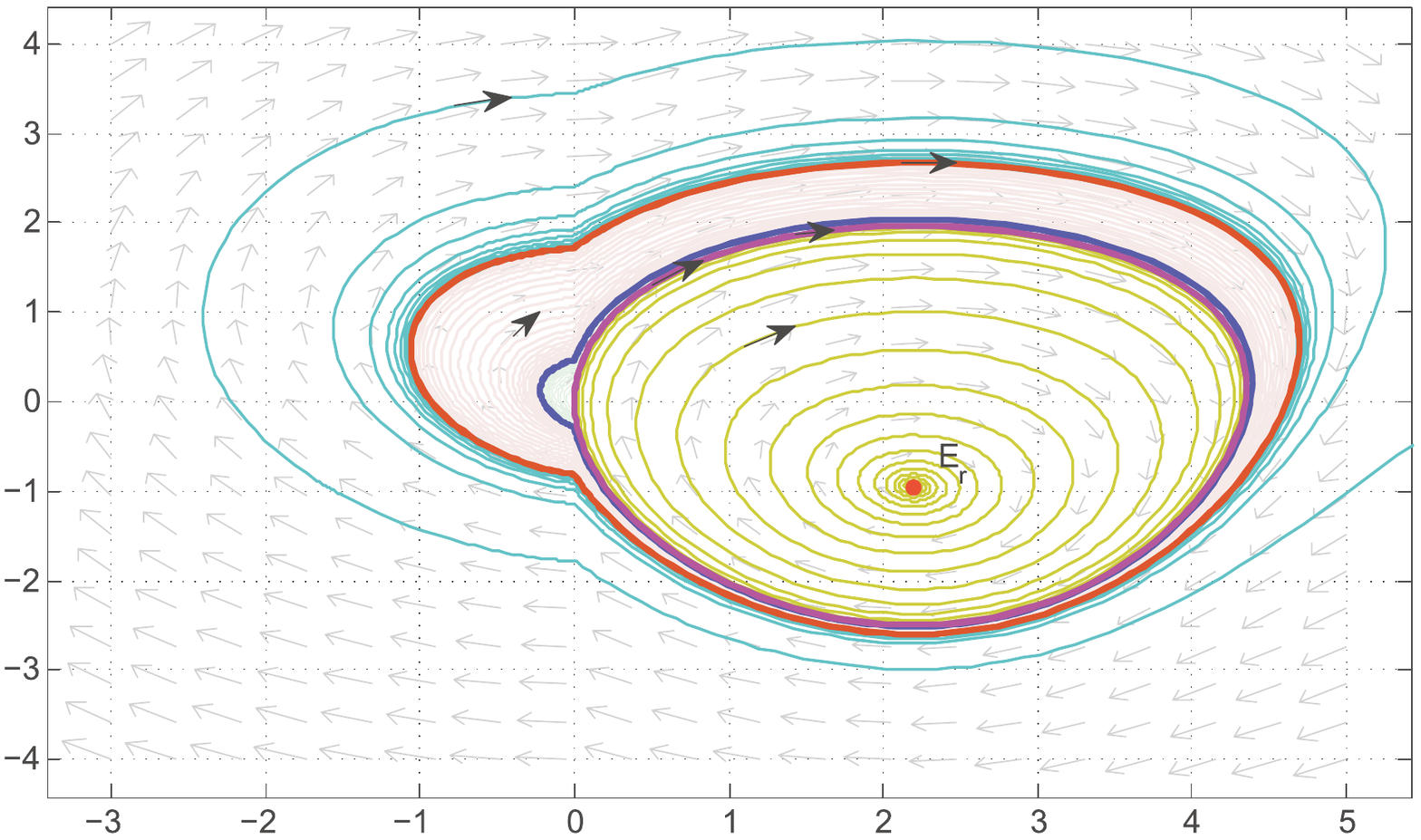}}
     \subfloat[$b=-5.9337$]
	{\includegraphics[scale=0.25]{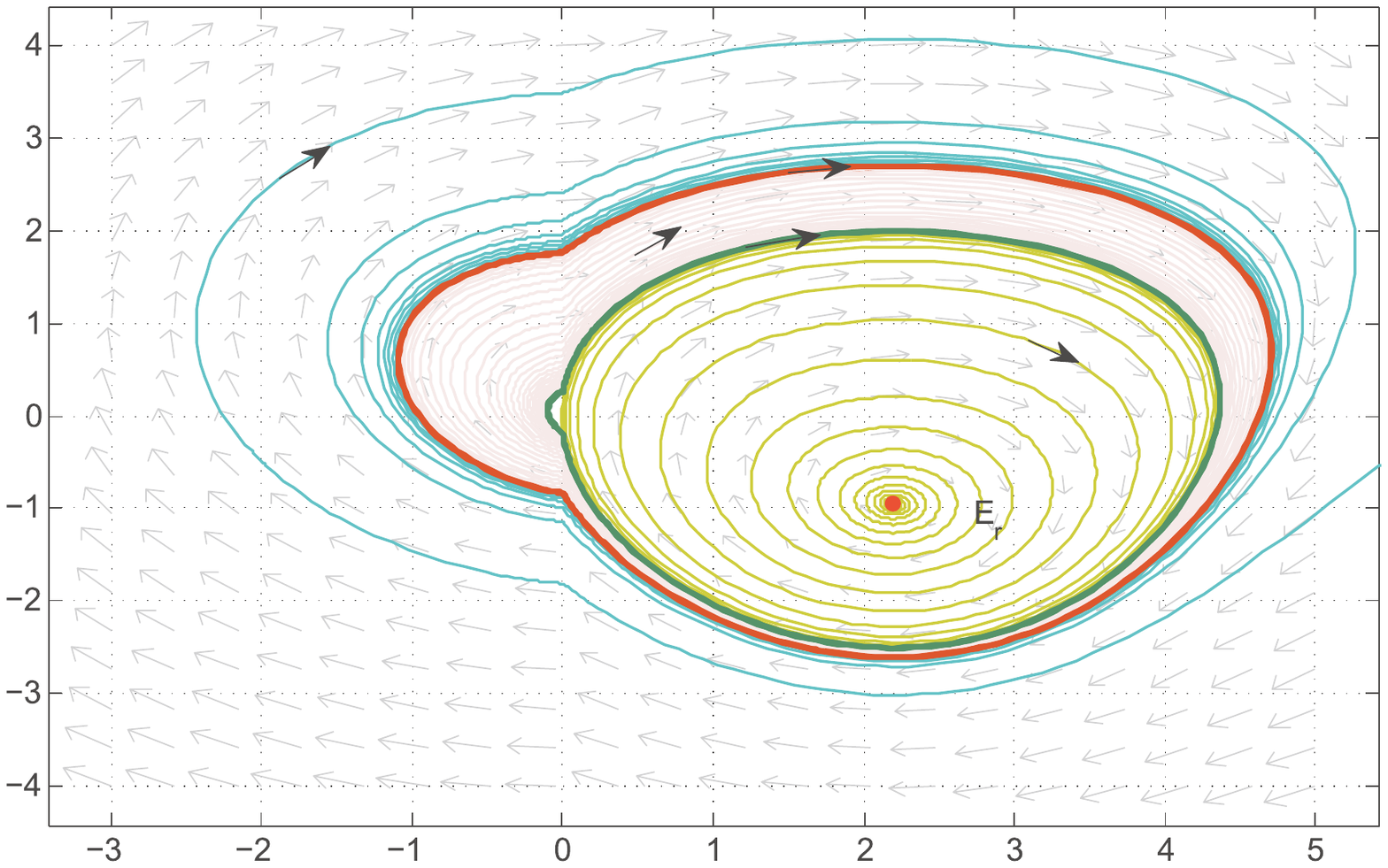}}\\
	\subfloat[$b=-6$]
	{\includegraphics[scale=0.25]{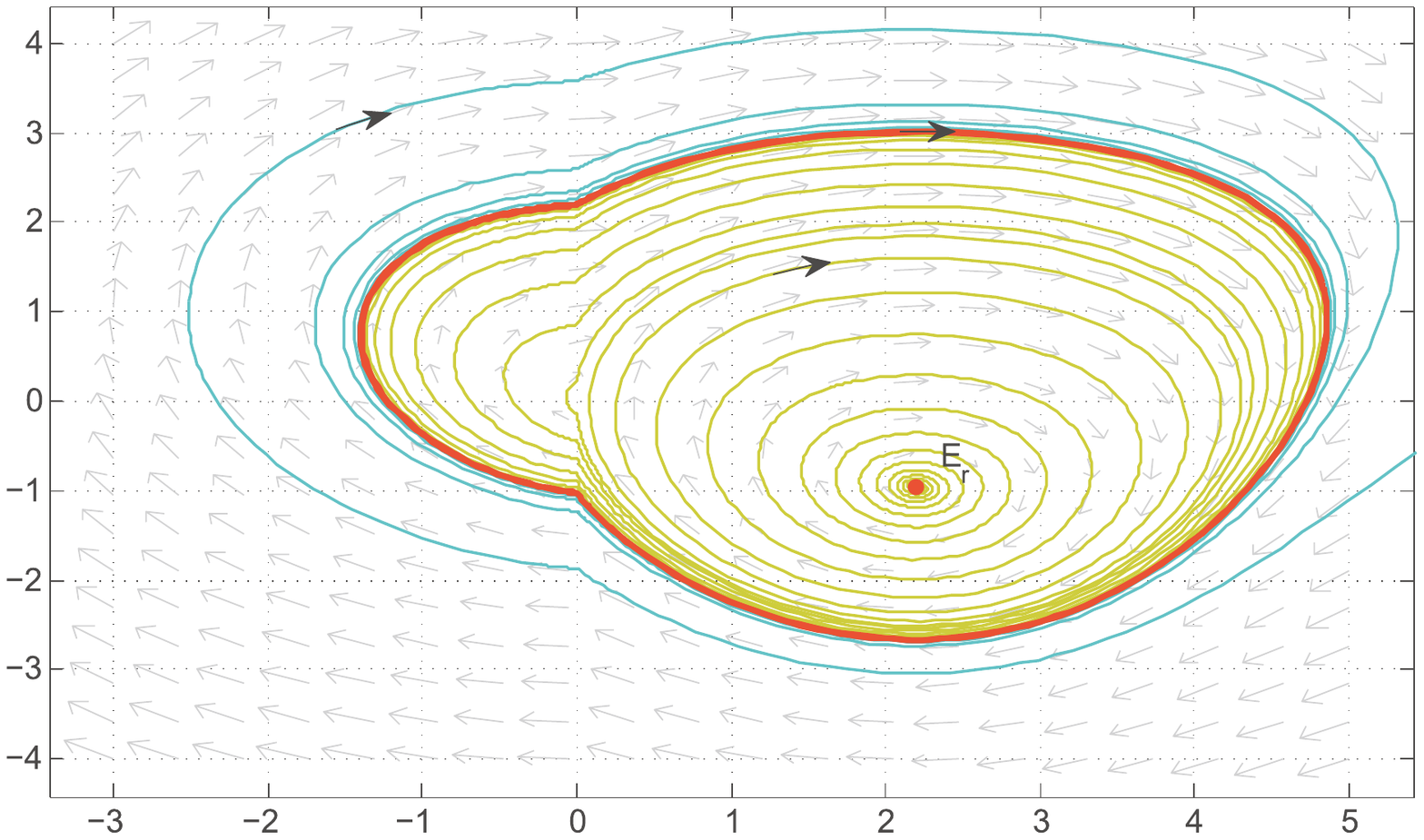}}
	\caption{{\footnotesize Numerical phase portraits with two equilibria when $a=1.2$, $\delta=0.1$.}}
	\label{simu-2}
\end{figure}

From Example 2, we find $\varphi_1(1.2,0.1)\thickapprox -5.93$,
$\varrho_1(1.2,0.1)\thickapprox -5.908$ and $\varrho_2(1.2,0.1)\thickapprox -5.9337$.
As $b$ decreases,
the double limit cycle bifurcation $DL_1$ occurs before the grazing bifurcation
when $a=1.2$ and $\delta=0.1$.
Since $\varphi_1(1.2,0.1)$ and $\varrho_2(1.2,0.1)$ are too close to each other,
the numerical simulation that exhibits three crossing limit cycles is not shown in  {\rm Fig. \ref{simu-2}}.

\begin{exmp}
Choose  $a=1.4$ and $\delta=0.1$.
System \eqref{SD} exhibits a stable small limit cycle which passes through a very
small neighborhood of the origin when $b=-7$,
as shown in {\rm Fig. \ref{simu-3} (a)}.
System \eqref{SD} exhibits at least three crossing limit cycles
when $b=-7.018$,
as shown in {\rm Fig. \ref{simu-3} (b)}.
System \eqref{SD}
 exhibits a stable crossing limit cycle when  $b=-7.1$,
as shown in {\rm Fig. \ref{simu-3} (c)}.
\end{exmp}

\begin{figure}[!htb]
	\subfloat[$b=-7$]
	{\includegraphics[scale=0.25]{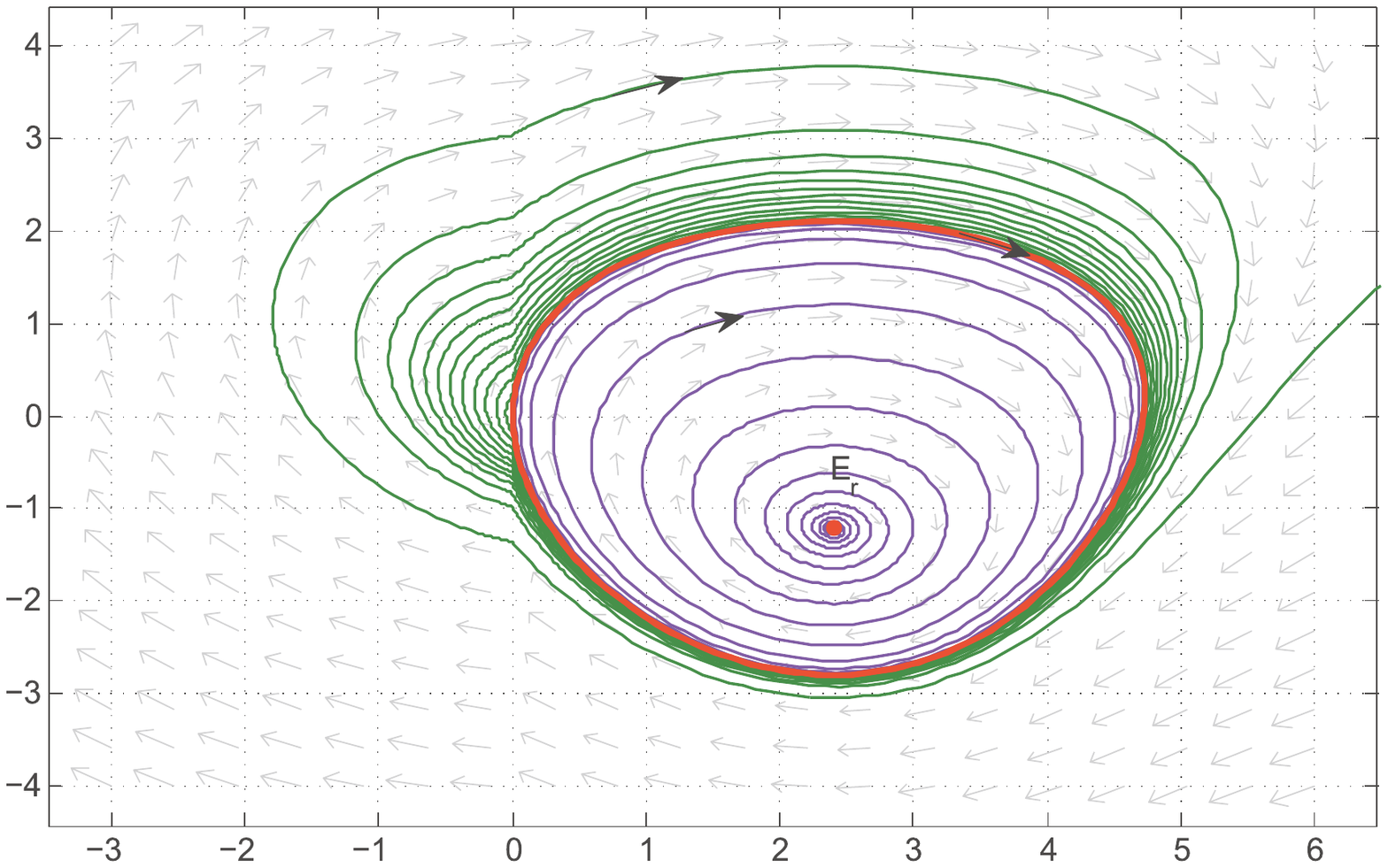}}
     \subfloat[$b=-7.018$]
	{\includegraphics[scale=0.25]{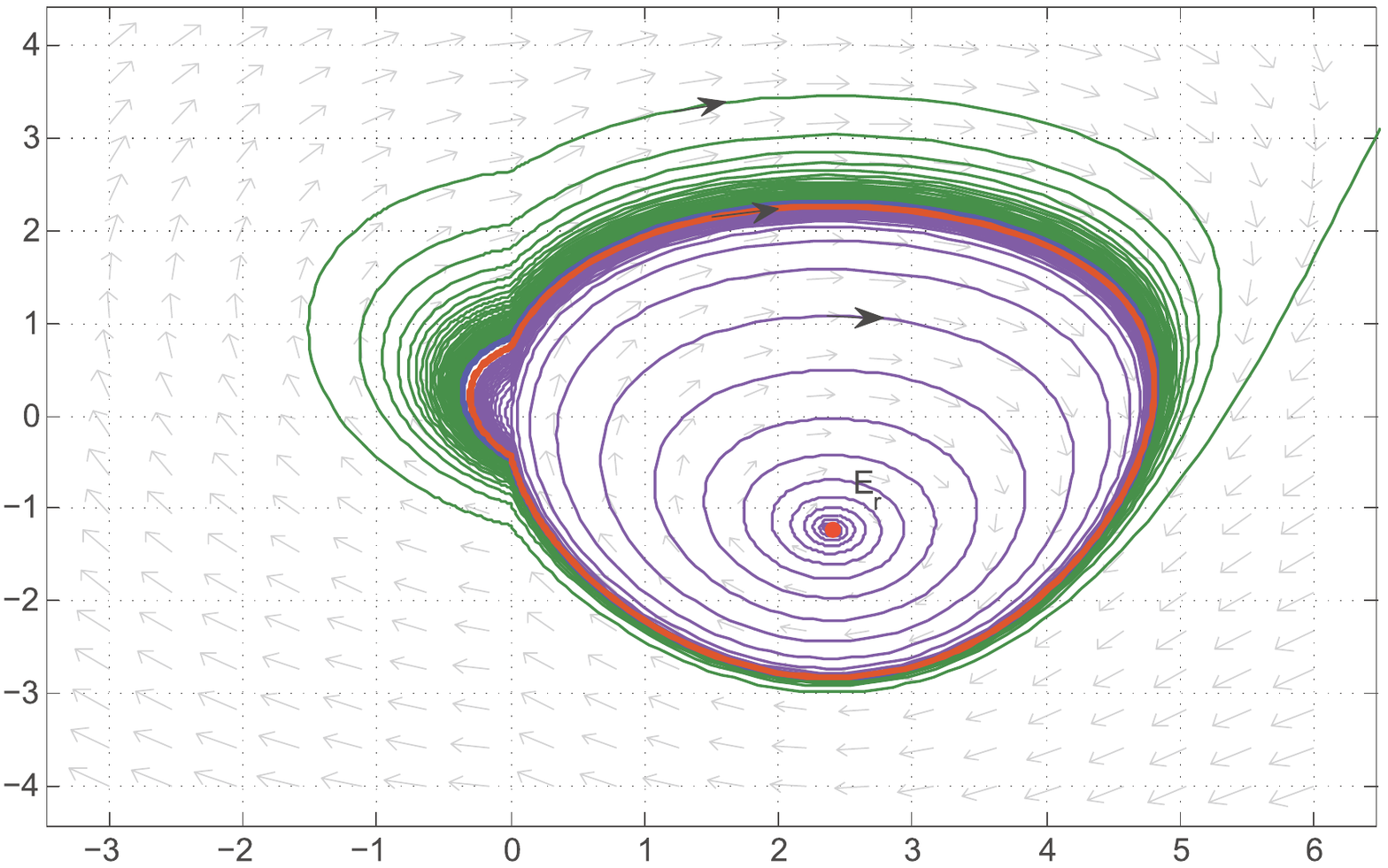}}\\
	\subfloat[$b=-7.1$]
	{\includegraphics[scale=0.25]{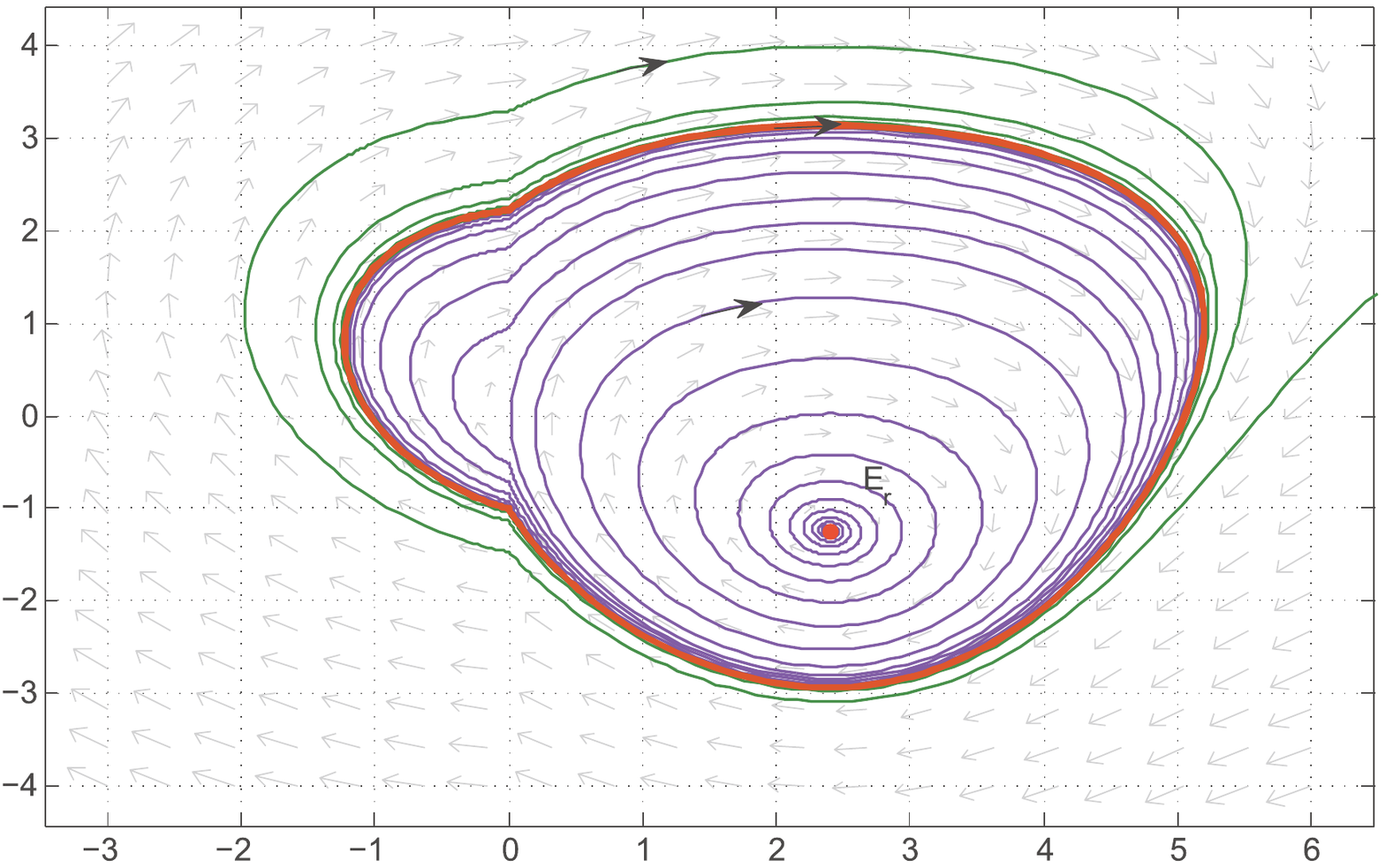}}
	\caption{{\footnotesize Numerical phase portraits with two equilibria when $a=1.4$, $\delta=0.1$.}}
	\label{simu-3}
\end{figure}

From Example 3, we find $\varphi_1(1.4,0.1)\thickapprox -7$,
$\varrho_1(1.4,0.1)\thickapprox \varrho_2(1.4,0.1)\thickapprox -7.018$.
As $b$ decreases,
the double limit cycle bifurcation $DL_1$ occurs after the grazing bifurcation
when $a=1.4$ and $\delta=0.1$.
In {\rm Fig. \ref{simu-3} (b)},
because the crossing limit cycles are too close to each other when $a=1.4$ and $\delta=0.1$, only the outermost and innermost crossing limit cycles are shown.

\section{Conclusions}

The bifurcation diagram and global phase portraits in the Poincar\'e disc of
system \eqref{SD} with $a=0$
are given by \cite{Chen}.
For system \eqref{SD} with $0<a\le 1$, the bifurcation diagram and global phase portraits in the Poincar\'e disc
are given in \cite{CTW}.
As shown in \cite{CTW} system  \eqref{SD} has three equilibria when $0<a<1$ and has two equilibria when $a=1$.
Moreover, system  \eqref{SD} exhibits at most one limit cycle surrounding the same single  equilibrium  and at most
two limit cycles surrounding all equilibria when $0<a\le 1$.
In the case $a>1$,  system  \eqref{SD} has a unique equilibrium.
Although the number of equilibria of  system  \eqref{SD} with $a>1$ is
less than the one of  system  \eqref{SD} with $0<a\le 1$,
Theorem \ref{mr2} tells us the upper bound on the number of limit cycles surrounding the unique equilibrium of system  \eqref{SD} with $a>1$
is three.
That means
the dynamical behavior of system  \eqref{SD} with $a>1$ is more complicated, and many classical techniques for
investigating the number of limit cycles for Li\'enard systems will no longer be applicable.

In this paper, the Melnikov method is used to get the number of crossing limit cycles of system \eqref{SD} with $a>1$ for $\delta>0$
sufficiently small.
To get the number of crossing limit cycles for general $\delta>0$,
the properties of rotated vector fields are used.
As a result, we get that
two  double crossing limit cycle bifurcations $DL_1$ and $DL_2$ will happen when $1<a<a_0$, where $a_0>1$ is a constant.

\begin{figure}[!htb]
	\centering
	\includegraphics[scale=0.27]{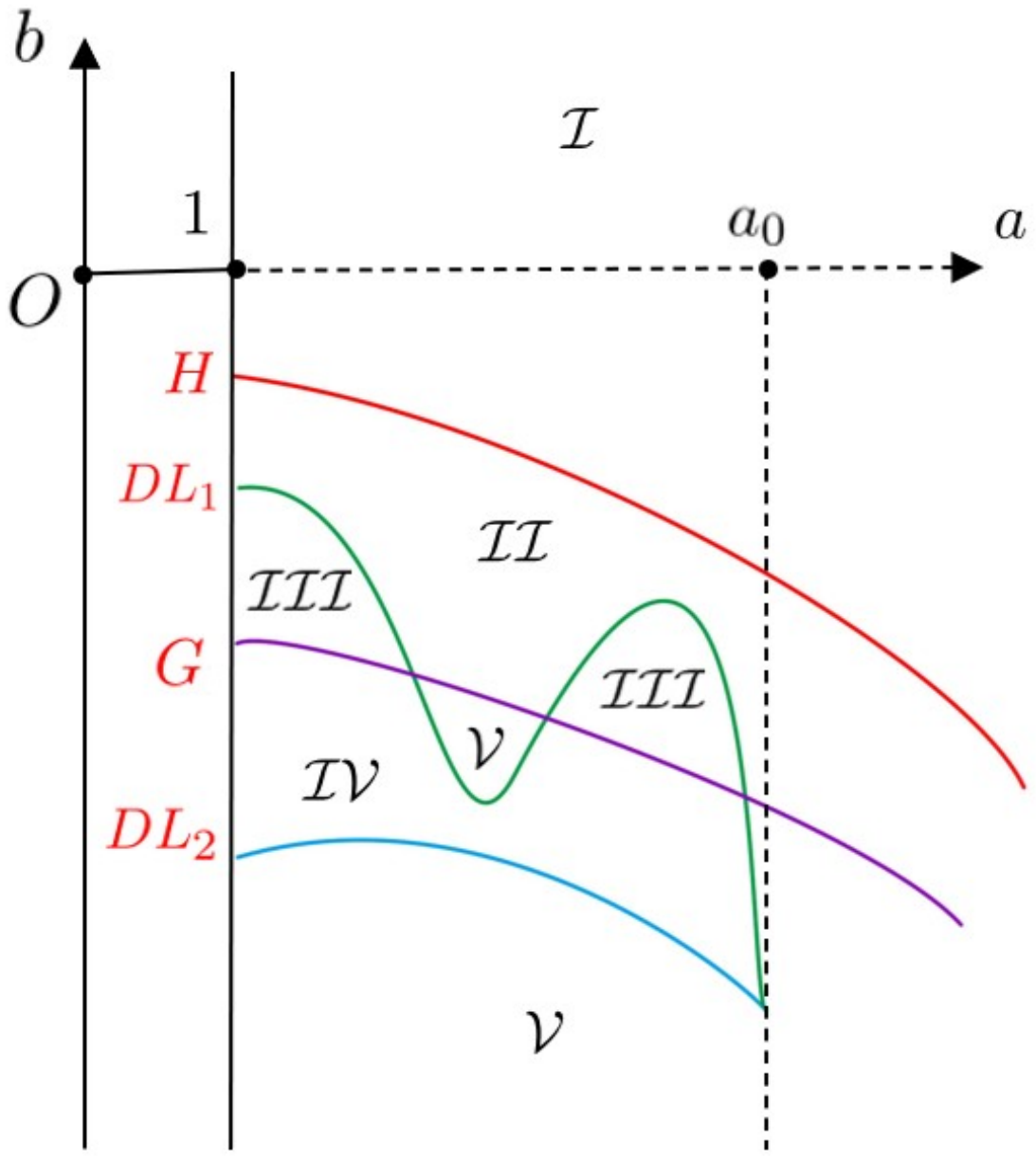}
	\caption{\footnotesize{  $DL_1$ and $G$ intersect more than once.}}
	\label{fig-c}
    \end{figure}

It is worth mentioning that
the curves of the double crossing limit cycle bifurcation $DL_1$  and the grazing bifurcation $G$ intersect at most once
when $\delta$ is a fixed and sufficiently small value.
However, this conclusion may be incorrect for a general fixed $\delta>0$.
Even if they intersect multiple times, there are no new phase portraits for  \eqref{SD} with $a>1$, as shown in Fig. \ref{fig-c}.
Thus, the analysis of global dynamics of  \eqref{SD} is given completely.

\section*{Acknowledgements}

This work is financially supported by the National Key R \& D Program of China (No. 2022YFA1005900).
The first author is supported by National Natural Science Foundation of China (No. 12371182), Hunan Provincial Natural Science Foundation of China (No. 2024JJ4048) and Science and Technology Innovation Program of Hunan Province (No. 2024RC3002).
The second author is supported by the National Natural Science Foundations of China (Nos. 12571190, 12322109 and 12171485), Hunan Basic Science Research Center for Mathematical Analysis (2024JC2002) and  Science and Technology Innovation Program of Hunan Province (No. 2023RC3040).
The third author is supported by the National Natural Science Foundation
of China (No. 12271355).
The fourth author is supported by the National Natural Science Foundation
of China (No. 12471155).

\section*{Appendix}
We first prove that $\frac{\mathrm{d}S_1(h,a)}{\mathrm{d}h}$ has at most one zero on $(0,\frac{4}{5})$. A simple calculation  leads to
\begin{equation}\label{ds1}
\frac{\mathrm{d}S_1(h,a)}{\mathrm{d}h}=\frac{\sqrt{2}\bar S_1(h,a)}{\sqrt{h} \left(a^2-2 a+2 h+1\right)^2 \left(a^2+2 a+2 h+1\right)^2},
\end{equation}
where
\begin{equation*}
\begin{split}
\bar S_1(h,a)=&-5 \pi  a^8+a^7 \left(32 \pi  h+30 \sqrt{2} \sqrt{h}-20 \pi \right)+2 a^6 \left(51 \pi  h-10 \sqrt{2} \sqrt{h}-10 \pi \right)\\
&+2 a^5 \left(60 \sqrt{2} h^{3/2}+84 \pi h^2+82 \pi  h-55 \sqrt{2} \sqrt{h}+10 \pi \right)\\
&+a^4 \left(80 \sqrt{2} h^{3/2}+412 \pi  h^2+234 \pi  h+40 \sqrt{2} \sqrt{h}+50 \pi \right)\\
&+2 a^3 \left(44 \sqrt{2}h^{5/2}+120 \sqrt{2} h^{3/2}+144 \pi  h^3+248 \pi  h^2+108 \pi  h+65 \sqrt{2} \sqrt{h}+10 \pi \right)\\
&+2 a^2 (2 h+1) \left(60 \sqrt{2} h^{3/2}+82 \pi  h^2+33 \pi  h-10\sqrt{2} \sqrt{h}-10 \pi \right)\\
&+2 a (2 h+1) \left(-16 \sqrt{2} h^{5/2}-50 \sqrt{2} h^{3/2}+40 \pi  h^3+8 \pi  h^2-26 \pi  h-25 \sqrt{2} \sqrt{h}-10 \pi \right)\\
&-\pi  (2 h+1)^3 (12h+5)\\
&+\left(-35 a^8+8 a^7 (4 h+5)-2 a^6 (51 h-50)+4 a^5 \left(42 h^2+11 h-30\right)\right.\\
&\left.+a^4 \left(68 h^2-74 h-90\right)+8 a^3 \left(36 h^3-8h^2+17 h+15\right)\right.\\
&\left.+2 a^2 \left(156 h^3+172 h^2+67 h+10\right)+4 a (2 h+1)^2 \left(10 h^2-13 h-10\right)\right.\\
&\left.+(2 h+1)^3 (12 h+5)\right) \arctan\left(\frac{\sqrt{2}\sqrt{h}}{a-1}\right)+\left(5 a^8-4 a^7 (8 h-5)+a^6 (20-102 h)\right.\\
&\left.-4 a^5 \left(42 h^2+41 h+5\right)-2 a^4 \left(206 h^2+117 h+25\right)-4 a^3 \left(72 h^3+124 h^2+54h+5\right)\right.\\
&\left.-2 a^2 \left(164 h^3+148 h^2+13 h-10\right)-4 a (2 h+1)^2 \left(10 h^2-3 h-5\right)\right.\\
&\left.+(2 h+1)^3 (12 h+5)\right) \arctan\left(\frac{\sqrt{2}
   \sqrt{h}}{a+1}\right).
\end{split}
\end{equation*}
It is easy to obtain that
\begin{eqnarray*}
\frac{\mathrm{d}\bar S_1(h,a)}{\mathrm{d}h}&=&2 \left(16 \pi  a^7+51 \pi  a^6+2 a^5 \left(84 \pi  h+35 \sqrt{2} \sqrt{h}+41 \pi \right)+a^4 \left(412 \pi  h+60 \sqrt{2} \sqrt{h}+117 \pi \right)\right.\\
&&\left.+4 a^3 \left(21 \sqrt{2} h^{3/2}+108\pi  h^2+124 \pi  h+30 \sqrt{2} \sqrt{h}+27 \pi \right)\right.\\
&&\left.+a^2 \left(200 \sqrt{2} h^{3/2}+492 \pi  h^2+296 \pi  h+20 \sqrt{2} \sqrt{h}+13 \pi \right)\right.\\
&&\left.+2 a \left(-64 \sqrt{2} h^{5/2}-134 \sqrt{2} h^{3/2}+160 \pi  h^3+84 \pi  h^2-44 \pi  h-55 \sqrt{2} \sqrt{h}-23 \pi \right)\right.\\
&&\left.-3 \pi  (2 h+1)^2 (16 h+7)\right)\\
&&+2 \left(16 a^7-51 a^6+2 a^5 (84 h+11)+a^4 (68 h-37)+a^3 \left(432 h^2-64 h+68\right)\right.\\
&&\left.+a^2 \left(468 h^2+344 h+67\right)+2 a \left(160 h^3-36 h^2-164 h-53\right)\right.\\
&&\left.+3 (2 h+1)^2(16 h+7)\right) \arctan\left(\frac{\sqrt{2} \sqrt{h}}{a-1}\right)\\
&&-2 \left(16 a^7+51 a^6+2 a^5 (84 h+41)+a^4 (412 h+117)+4 a^3 \left(108 h^2+124 h+27\right)\right.\\
&&\left.+a^2\left(492 h^2+296 h+13\right)+2 a \left(160 h^3+84 h^2-44 h-23\right)\right.\\
&&\left.-3 (2 h+1)^2 (16 h+7)\right) \arctan\left(\frac{\sqrt{2} \sqrt{h}}{a+1}\right)\\
&>&2 \left(16 \pi  a^7+51 \pi  a^6+2 a^5 \left(84 \pi  h+35 \sqrt{2} \sqrt{h}+41 \pi \right)+a^4 \left(412 \pi  h+60 \sqrt{2} \sqrt{h}+117 \pi \right)\right.\\
&&\left.+4 a^3 \left(21 \sqrt{2} h^{3/2}+108\pi  h^2+124 \pi  h+30 \sqrt{2} \sqrt{h}+27 \pi \right)\right.\\
&&\left.+a^2 \left(200 \sqrt{2} h^{3/2}+492 \pi  h^2+296 \pi  h+20 \sqrt{2} \sqrt{h}+13 \pi \right)\right.\\
&&\left.+2 a \left(-64 \sqrt{2} h^{5/2}-134 \sqrt{2} h^{3/2}+160 \pi  h^3+84 \pi  h^2-44 \pi  h-55 \sqrt{2} \sqrt{h}-23 \pi \right)\right.\\
&&\left.-3 \pi  (2 h+1)^2 (16 h+7)\right)\\
&&-4 \left(16 a^7+51 a^6+2 a^5 (84 h+41)+a^4 (412 h+117)+4 a^3 \left(108 h^2+124 h+27\right)\right.\\
&&\left.+a^2\left(492 h^2+296 h+13\right)+2 a \left(160 h^3+84 h^2-44 h-23\right)\right.\\
&&\left.-3 (2 h+1)^2 (16 h+7)\right)\frac{\pi}{2}\\
&=&4 \sqrt{2} a \sqrt{h} \left(35 a^4+30 a^3+6 a^2 (7 h+10)+10 a (10 h+1)-64 h^2-134 h-55\right)\\
&>&\frac{8}{125} \sqrt{2} a \sqrt{h} (13895 h+9076)\\
&>&0.
\end{eqnarray*}
Thus, for $a>\frac{7}{5}$, $\bar S_1(h,a)$ is monotonically increasing for $h\in(0,\frac{4}{5})$. Note that
$\bar S_1(0,a)=-5 \pi  (a-1)^2 (a+1)^6<0$, and
\begin{eqnarray*}
\bar S_1\left(\frac{4}{5},a\right)&=&\frac{1}{625} \left(-3125 \pi  a^8+500 \left(15 \sqrt{10}+7 \pi \right) a^7-500 \left(10 \sqrt{10}-77 \pi \right) a^6\right.\\
&&\left.-700 \left(5 \sqrt{10}-231 \pi \right) a^5+50\left(520 \sqrt{10}+6261 \pi \right) a^4+20 \left(4729 \sqrt{10}+20553 \pi \right) a^3\right.\\
&&\left.+260 \left(190 \sqrt{10}+861 \pi \right) a^2-52 \left(1881 \sqrt{10}+325 \pi\right) a-160381 \pi \right)\\
&&+\frac{1}{625} \left(-21875 a^8+41000 a^7+11500 a^6+14200 a^5-66050 a^4+209560 a^3\right.\\
&&\left.+316940 a^2-236600 a+160381\right) \arctan\left(\frac{2
   \sqrt{\frac{2}{5}}}{a-1}\right)\\
&&+\frac{1}{625} \left(3125 a^8-3500 a^7-38500 a^6-161700 a^5-313050 a^4-411060 a^3\right.\\
&&\left.-223860 a^2+16900 a+160381\right) \arctan\left(\frac{2 \sqrt{\frac{2}{5}}}{a+1}\right).
\end{eqnarray*}
Since $\bar S_1^{(6)}(\frac{4}{5},a)<0$ for $a>\frac{7}{5}$, and $\bar S_1^{(5)}(\frac{4}{5},\frac{7}{5})>0$, $\bar S_1^{(5)}(\frac{4}{5},10)<0$,
 $\bar S_1^{(5)}(\frac{4}{5},a)$ has a unique zero $\bar a_5\in(\frac{7}{5},10)$ for $a>\frac{7}{5}$.
One has  $\bar S_1^{(4)}(\frac{4}{5},a)$ increases on $(\frac{7}{5},\bar a_5)$ and then decreases on  $(\bar a_5,+\infty)$.
  Since  $\bar S_1^{(4)}(\frac{4}{5},\frac{7}{5})>0$ and $\bar S_1^{(4)}(\frac{4}{5},10)<0$,
 $\bar S_1^{(4)}(\frac{4}{5},a)$ has a unique zero $\bar a_4\in(\frac{7}{5},10)$ for $a>\frac{7}{5}$.
One has  $\bar S_1^{(3)}(\frac{4}{5},a)$ increases on $(\frac{7}{5},\bar a_4)$ and then decreases on  $(\bar a_4,+\infty)$.
And so on, we have  $\bar S_1^{(i)}(\frac{4}{5},a), i=3,2,1$ has a unique zero  $\bar a_i\in(\frac{7}{5},10)$, and $\bar S_1(\frac{4}{5},a)$ has a unique zero $\bar a\in(\frac{7}{5},10)$ for $a>\frac{7}{5}$.
Thus, when $\frac{7}{5}<a<\bar a$, $\bar S_1\left(\frac{4}{5},a\right)>0$ and $\bar S_1(h,a)$ has one zero on $(0,\frac{4}{5})$,
and when $a\geq\bar a$, $\bar S_1\left(\frac{4}{5},a\right)\leq0$ and  $\bar S_1(h,a)$ has no zero on $(0,\frac{4}{5})$.
So, by \eqref{ds1}, $\frac{\mathrm{d}S_1(h,a)}{\mathrm{d}h}<0$ holds for $h\in(0,\frac{4}{5})$ when $a\geq\bar a$, and $\frac{\mathrm{d}S_1(h,a)}{\mathrm{d}h}$ has one zero on $(0,\frac{4}{5})$ when $\frac{7}{5}<a<\bar a$. Finally, we conclude that $S_1(h,a)$ is either  decreasing on $(0,\frac{4}{5})$ for $a\geq\bar a$, or it decreases to a minimum then increases on $(0,\frac{4}{5})$ for $\frac{7}{5}<a<\bar a$.

\end{document}